\documentclass[english]{amsart}
\usepackage[latin9]{inputenc}

\usepackage{amstext}
\usepackage{amsthm}
\usepackage{amssymb}

\makeatletter

\newcommand{\lyxmathsym}[1]{\ifmmode\begingroup\def\b@ld{bold}
  \text{\ifx\math@version\b@ld\bfseries\fi#1}\endgroup\else#1\fi}

\numberwithin{equation}{section}
\numberwithin{figure}{section}
\theoremstyle{plain}
\newtheorem{thm}{\protect\theoremname}[section]
  \theoremstyle{definition}
  \newtheorem{defn}[thm]{\protect\definitionname}
  \theoremstyle{remark}
  \newtheorem*{rem*}{\protect\remarkname}
  \theoremstyle{remark}
  \newtheorem{rem}[thm]{\protect\remarkname}
  \theoremstyle{plain}
  \newtheorem{cor}[thm]{\protect\corollaryname}
  \theoremstyle{plain}
  \newtheorem{lem}[thm]{\protect\lemmaname}

\makeatother

\usepackage{babel}
  \providecommand{\corollaryname}{Corollary}
  \providecommand{\definitionname}{Definition}
  \providecommand{\lemmaname}{Lemma}
  \providecommand{\remarkname}{Remark}
\providecommand{\theoremname}{Theorem}

\usepackage{mathtools}
\mathtoolsset{showonlyrefs}
\begin{document}

\title[Strichartz Estimates]{Strichartz Estimates for Wave Equations with Charge Transfer Hamiltonians}

\author{Gong Chen}

\thanks{This work is part of the author\textquoteright s Ph.D. thesis at
the University of Chicago.}

\date{\today}

\urladdr{http://www.math.uchicago.edu/\textasciitilde{}gc/}

\email{gc@math.uchicago.edu}

\address{Department of Mathematics, The University of Chicago, 5734 South
University Avenue, Chicago, IL 60615, U.S.A}

\begin{abstract}
We prove Strichartz estimates (both regular and reversed) for a scattering
state to the wave equation with a charge transfer Hamiltonian in $\mathbb{R}^{3}$:
\[
\partial_{tt}u-\Delta u+\sum_{j=1}^{m}V_{j}\left(x-\vec{v}_{j}t\right)u=0.
\]
The energy estimate and the local energy decay of a scattering state
are also established. In order to study nonlinear multisoltion systems,
we will present the inhomogeneous generalizations of Strichartz estimates
and local decay estimates. As an application of our results, we show
that scattering states indeed scatter to solutions to the free wave
equation. These estimates for this linear models are also of crucial
importance for problems related to interactions of potentials and
solitons, for example, in \cite{GC4}.
\end{abstract}

\maketitle
\tableofcontents{}
\section{Introduction\label{sec:Intro}}

In this paper, we study wave equations with charge transfer Hamiltonian
in $\mathbb{R}^{3}$. To be more precise, consider the wave equation
with the time-dependent charge transfer Hamiltonian 
\begin{equation}
H(t)=-\Delta+\sum_{j=1}^{m}V_{j}\left(x-\vec{v}_{j}t\right)\label{eq:11}
\end{equation}
where $V_{j}(x)$'s are rapidly decaying smooth potentials and $\left\{ \vec{v}_{j}\right\} $
is a set of distinct constant velocities such that 
\begin{equation}
\left|\vec{v}_{i}\right|<1,\,1\leq i\leq m.
\end{equation}
Due to the nature of our problem, we focus on initial data in the
energy space. We will prove Strichartz estimates, energy estimates,
the local energy decay and the boundedness of the total energy for
a scattering state to the wave equation 
\begin{equation}
\partial_{tt}\psi+H(t)\psi=0\label{eq:12}
\end{equation}
associated with a charge transfer Hamiltonian $H(t)$. Throughout,
we use $\partial_{tt}u:=\frac{\partial^{2}}{\partial t\partial t}$,
$u_{t}:=\frac{\partial}{\partial_{t}}u$, $\Delta:=\sum_{i=1}^{n}\frac{\partial^{2}}{\partial x_{i}\partial x_{i}}$
and occasionally, $\square:=-\partial_{tt}+\Delta$.

\subsection{Historical background}

In this subsection, we briefly discuss some background of Strichartz
estimates, reversed Strichartz estimates.

\smallskip

Our starting point is the free wave equation ($H_{0}=-\Delta$) on
$\mathbb{R}^{3}$ 
\begin{equation}
\partial_{tt}u-\Delta u=0
\end{equation}
with initial data
\begin{equation}
u(x,0)=g(x),\,u_{t}(x,0)=f(x).
\end{equation}
We can write down $u$ explicitly, 
\begin{equation}
u=\frac{\sin\left(t\sqrt{-\Delta}\right)}{\sqrt{-\Delta}}f+\cos\left(t\sqrt{-\Delta}\right)g.
\end{equation}
Then for $p>\frac{2}{s}$ and $\left(p,q\right)$ satisfying 
\begin{equation}
\frac{3}{2}-s=\frac{1}{p}+\frac{3}{q},
\end{equation}
one has 
\begin{equation}
\|u\|_{L_{t}^{p}L_{x}^{q}}\lesssim\|g\|_{\dot{H}^{s}}+\|f\|_{\dot{H}^{s-1}}.\label{eq:IFStri}
\end{equation}
Strichartz estimates \eqref{eq:IFStri}, which are stated precisely
in Theorem \ref{thm:StrichF}, are estimates of solutions in terms
of space-time integrability properties. The non-endpoint estimates
for the wave equations can be found in Ginibre-Velo \cite{GV}. Keel\textendash Tao
\cite{KT} also obtained sharp Strichartz estimates for the free wave
equation in $\mathbb{R}^{n},\,n\geq4$ and everything except the endpoint
in $\mathbb{R}^{3}$. See Keel-Tao \cite{KT} and Tao's book \cite{Tao}
for more details on the subject's background and the history. 

\smallskip

In $\mathbb{R}^{3}$, there is no hope to obtain such  an estimate
with the $L_{t}^{2}L_{x}^{\infty}$ norm, the so-called endpoint Strichartz
estimate for free wave equations, cf.~Klainerman-Machedon \cite{KM}
and Machihara-Nakamura-Nakanishi-Ozawa \cite{MNNO}. But if we reverse
the order of space-time integration, one can obtain a version of reversed
Strichartz estimates from the Morawetz estimate, cf.~Theorem \ref{thm:EndRStrichF}:
\begin{equation}
\left\Vert \frac{\sin\left(t\sqrt{-\Delta}\right)}{\sqrt{-\Delta}}f\right\Vert _{L_{x}^{\infty}L_{t}^{2}}\lesssim\left\Vert f\right\Vert _{L^{2}\left(\mathbb{R}^{3}\right)},
\end{equation}
\begin{equation}
\left\Vert \cos\left(t\sqrt{-\Delta}\right)g\right\Vert _{L_{x}^{\infty}L_{t}^{2}}\lesssim\left\Vert g\right\Vert _{\dot{H}^{1}\left(\mathbb{R}^{3}\right)}.
\end{equation}
These types of estimates are extended to inhomogeneous cases and perturbed
Hamiltonians in Beceanu-Goldberg \cite{BecGo}. In Section \ref{sec:Prelim}
and Section \ref{sec:Slanted}, we will rely crucially on these estimates
and their generalizations.

\smallskip

Consider a linear wave equation with a real-valued stationary potential
in $\mathbb{R}^{3}$, 
\begin{equation}
H=-\Delta+V,
\end{equation}
\begin{equation}
\partial_{tt}u+Hu=\partial_{tt}u-\Delta u+Vu=0,
\end{equation}
\begin{equation}
u(x,0)=g(x),\,u_{t}(x,0)=f(x).
\end{equation}
One substantial difference between the perturbed Hamiltonian $H=-\Delta+V$
and the free Laplacian $-\Delta$ is the possible existence of eigenvalues
and bound states, i.e., $L^{2}$ eigenfunctions of $H$. For the class
of short-range potentials we consider in this paper, the essential
spectrum of $H$ is $[0,\infty)$ and the point spectrum may include
a countable number of eigenvalues in a bounded subset of the real
axis that is discrete away from zero. We further assume that zero
is a regular point of the spectrum of $H$. Under our hypotheses $H$
only has pure absolutely continuous spectrum on $[0,\infty)$ and
a finite number of negative eigenvalues. It is very crucial to notice
that if $E<0$ is a negative eigenvalue, the associated eigenfunction
responds to the wave equation propagators by scalar multiplication
by $\cos\left(t\sqrt{E}\right)$ or $\frac{\sin\left(t\sqrt{E}\right)}{E^{\frac{1}{2}}}$,
both of which will grow exponentially since $\sqrt{E}$ is purely
imaginary. Thus, dispersive estimates and Strichartz estimates for
$H$ must include a projection $P_{c}$ onto the continuous spectrum
in order to get away from this situation. 

\smallskip

The problem of dispersive decay and Strichartz estimates for the wave
equation with a potential has received much attention in recent years,
see the papers by Beceanu-Goldberg \cite{BecGo}, Krieger-Schlag \cite{KS}
and the survey by Schlag \cite{Sch} for further details and references. 

\smallskip

The Strichartz estimates in this case are in the form:
\begin{equation}
\left\Vert \frac{\sin\left(t\sqrt{H}\right)}{\sqrt{H}}P_{c}f+\cos\left(t\sqrt{H}\right)P_{c}g\right\Vert _{L_{t}^{p}L_{x}^{q}}\lesssim\|g\|_{\dot{H}^{1}}+\|f\|_{L^{2}{}^{,}}
\end{equation}
with $2<p,\,\frac{1}{2}=\frac{1}{p}+\frac{3}{q}.$ One also has the
endpoint reversed Strichartz estimates: 
\begin{equation}
\left\Vert \frac{\sin\left(t\sqrt{H}\right)}{\sqrt{H}}P_{c}f+\cos\left(t\sqrt{H}\right)P_{c}g\right\Vert _{L_{x}^{\infty}L_{t}^{2}}\lesssim\|f\|_{L^{2}}+\|g\|_{\dot{H}^{1}},
\end{equation}
see Theorem \ref{thm:PStriRStrich}.

\smallskip

There are extra difficulties when dealing with time-dependent potentials.
For example, given a general time-dependent potential $V(x,t)$, it
is not clear how to introduce an analog of bound states and a spectral
projection. The evolution might not satisfy group properties any more.
It might also result in the growth of certain norms of the solutions,
see the book by Bourgain \cite{Bou}. In this paper, we focus on the
charge transfer Hamiltonian \eqref{eq:11} in $\mathbb{R}^{3}$:
\begin{equation}
H(t)=-\Delta+\sum_{j=1}^{m}V_{j}\left(x-\vec{v}_{j}t\right),
\end{equation}
which appears naturally in the study of nonlinear multisoliton system,
see Rodnianski-Schlag-Soffer \cite{RSS2} for the Schr\"odinger model.
For the wave model, 
\begin{equation}
\partial_{tt}u-\Delta u+\sum_{j=1}^{m}V_{j}\left(x-\vec{v}_{j}t\right)u=0,\quad\left|\vec{v}_{i}\right|<1,\,1\leq i\leq m,
\end{equation}
in this paper, we prove Strichartz estimates, energy estimates, the
local energy decay which are essential to analyze the stability of
multi-soliton states. In Chen \cite{GC4}, relying on this linear
model, we construct a multisoliton structure to the defocusing energy
critical wave equation with potentials in $\mathbb{R}^{3}$:
\begin{equation}
\partial_{tt}u-\Delta u+\sum_{j=1}^{m}V_{j}\left(x-\vec{v}_{j}t\right)u+u^{5}=0.\label{eq:maineq}
\end{equation}
We also analyze the asymptotic stability of the multisoliton. Since
each bubble in the multisolion structure decays slowly like $\frac{1}{\text{\ensuremath{\left\langle x\right\rangle }}}$,
the interactions among each bubble are very strong. Our linear theory
and reversed Strichartz estimaes play a pivotal rule in the construction.
And it turns out that this model is the first multisoliton structure
for wave equations in $\mathbb{R}^{3}$. 

\smallskip

The study of Schr\"odinger equations with a charge transfer Hamiltonian
can be found in Rodnianski-Schlag-Soffer \cite{RSS}, Cai \cite{Cai},
Chen \cite{GC1} and Deng-Soffer-Yao \cite{DSY}. For the Schr\"odinger
model, there is no need to require $\left|\vec{v}_{i}\right|<1$.
In Rodnianski-Schlag-Soffer \cite{RSS}, the authors proved the dispersive
estimates for both the scalar and matrix Schr\"odinger charge transfer
models. They introduced Galilei transformations to interchange stationary
frames with respect to different potentials. Basically, they applied
a bootstrap argument via a semi-classical propagation lemma for low
frequencies and Kato's smoothing estimate for high frequencies. With
careful analysis of wave operators, the authors also obtain the results
on the asymptotic completeness. Their works inspired the subsequent
development in Cai \cite{Cai} where the $L^{1}\rightarrow L^{\infty}$
dispersive estimate is proved. Later on, by Chen \cite{GC1}, Strichartz
estimates for both the scalar and matrix Schr\"odinger charge transfer
models were presented based on a time-dependent local decay estimate
and the endpoint Strichartz estimate for the free equations. Alternatively,
Strichartz estimates can be obtained by analysis of wave operators,
see Deng-Soffer-Yao \cite{DSY}. 

\smallskip

Compared with Schr\"odinger equations, wave equations have some natural
difficulties, for example the evolution of bound states of wave equations
leads to exponential growth as we pointed out above, meanwhile the
evolution of bound states of Schr\"odinger equations are merely multiplied
by oscillating factors. The structure of wave operators in the wave
equation setting is not clear either. Moreover, the endpoint Strichartz
estimate for free equations, an important tool used in the paper \cite{GC1},
also fails for free wave equations in $\mathbb{R}^{3}$. Last but
not least, Lorentz transformations are space-time rotations, therefore
one can not hope to succeed by the approach used with Schr\"odinger
equations based on Galilei transformations. Galilei transformations
are bounded in any $L^{p}$ space, but it is not clear under Lorentz
transformations whether the energy with respect to the new frame stays
comparable to the energy in the original frame. To the author's knowledge,
for wave equations with even just one potential moving along a space-like
line, Strichartz estimates, scattering, and the asymptotic decomposition
of the evolution are new. We refer to \cite{GC2} for more information
on wave equations with one moving potential.

\smallskip

\subsection{Charge transfer model and main results}

Before we give the precise definition of our model, it is necessary
to introduce Lorentz transformations. Given a vector $\vec{\mu}\in\mathbb{R}^{3}$,
there is a Lorentz transformation $L\left(\vec{\mu}\right)$ acting
on $\left(x,t\right)\in\mathbb{R}^{3+1}$ such that it makes the moving
frame $\left(x-\vec{\mu}t,t\right)$ stationary. We can use a $4\times4$
matrix $B(\vec{\mu})$ to represent the transformation $L\left(\vec{\mu}\right)$.
Moreover, for the given vector $\vec{\mu}=(\mu_{1},\mu_{2},\mu_{3})\in\mathbb{R}^{3}$,
there is a $3\times4$ matrix $M\left(\vec{\mu}\right)$ such that
\begin{equation}
\left(x-\vec{\mu}t\right)^{T}=M\left(\vec{\mu}\right)\left(x,t\right)^{T},
\end{equation}
where the superscript $T$ denotes the transpose of a vector.

\smallskip

With the preparations above, we can set up our model. We consider
the scalar charge transfer model for wave equations in the following
sense:
\begin{defn}
\label{def: Charge} By a wave equation with a charge transfer Hamiltonian
we mean a wave equation 
\begin{equation}
\partial_{tt}u-\Delta u+\sum_{j=1}^{m}V_{j}\left(x-\vec{v}_{j}t\right)u=0,\label{eq:18}
\end{equation}
\[
u|_{t=0}=g,\,\,\partial_{t}u|_{t=0}=f,\,\,x\in\mathbb{R}^{3},
\]
where $\vec{v}_{j}$'s are distinct vectors in $\mathbb{R}^{3}$ with
\begin{equation}
\left|\vec{v}_{i}\right|<1,\,1\leq i\leq m.
\end{equation}
and the real potentials $V_{j}$ are such that $\forall1\leq j\leq m$ 

1) $V_{j}$ is time-independent and decays with rate $\left\langle x\right\rangle ^{-\alpha}$
with $\alpha>3$

2) $0$ is neither an eigenvalue nor a resonance of the operators
\begin{equation}
H_{j}=-\Delta+V_{j}\left(S\left(\vec{v}_{j}\right)x\right),\label{eq:19}
\end{equation}
where $S\left(\vec{v_{j}}\right)x=M\left(\vec{v}_{j}\right)B^{-1}\left(\vec{v}_{j}\right)\left(x,0\right)^{T}.$ 
\end{defn}
Recall that $\psi$ is a resonance at $0$ if it is a distributional
solution of the equation $H_{k}\psi=0$ which belongs to the space
$L^{2}\left(\left\langle x\right\rangle ^{-\sigma}dx\right):=\left\{ f:\,\left\langle x\right\rangle ^{-\sigma}f\in L^{2}\right\} $
for any $\sigma>\frac{1}{2}$, but not for $\sigma=\frac{1}{2}.$
\begin{rem*}
The construction of $S\left(\vec{v}_{j}\right)$ is clear from the
change between different frames under Lorentz transformations. In
our concrete problem below \eqref{eq:2V}, $S\left(\vec{v}_{j}\right)$
can be written down explicitly. 
\end{rem*}
To simplify our argument, throughout this paper, we discuss the wave
equation with a charge transfer Hamiltonian in the sense of Definition
\ref{def: Charge} with $m=2$, a stationary $V_{1}$ and a $V_{2}$
moving along $\overrightarrow{e_{1}}$ with speed $\left|v\right|<1$,
i.e., the velocity is 
\begin{equation}
\vec{v}=\left(v,0,0\right).
\end{equation}
Under this setting, by Definition \ref{def: Charge}, 
\begin{equation}
H_{1}=-\Delta+V_{1}(x),
\end{equation}
and 
\begin{equation}
\,H_{2}=-\Delta+V_{2}\left(\sqrt{1-\left|v\right|^{2}}x_{1},x_{2},x_{3}\right).\label{eq:2V}
\end{equation}
It is easy to see that our arguments work for $m>2$.

\smallskip

An indispensable tool we need to study the charge transfer model is
the Lorentz transformation. Throughout this paper, we apply Lorentz
transformations $L$ with respect to a moving frame with speed $\left|v\right|<1$
along the $x_{1}$ direction. After we apply the Lorentz transformation
$L$, under the new coordinates, $V_{2}$ is stationary meanwhile
$V_{1}$ will be moving.

\smallskip

Writing down the Lorentz transformation $L$ explicitly, we have 
\begin{equation}
\begin{cases}
t'=\gamma\left(t-vx_{1}\right)\\
x_{1}'=\gamma\left(x_{1}-vt\right)\\
x_{2}'=x_{2}\\
x_{3}'=x_{3}
\end{cases}\label{eq:LorentzT}
\end{equation}
with 
\begin{equation}
\gamma=\frac{1}{\sqrt{1-\left|v\right|^{2}}}.
\end{equation}
We can also write down the inverse transformation of the above:
\begin{equation}
\begin{cases}
t=\gamma\left(t'+vx_{1}'\right)\\
x_{1}=\gamma\left(x_{1}'+vt'\right)\\
x_{2}=x_{2}'\\
x_{3}=x_{3}'
\end{cases}.\label{eq:InvLorentT}
\end{equation}
Under the Lorentz transformation $L$, if we use the subscript $L$
to denote a function with respect to the new coordinate $\left(x',t'\right)$,
we have 
\begin{equation}
u_{L}\left(x_{1}',x_{2}',x_{3}',t'\right)=u\left(\gamma\left(x_{1}'+vt'\right),x_{2}',x_{3}',\gamma\left(t'+vx_{1}'\right)\right)\label{eq:Lcoordinate}
\end{equation}
and 
\begin{equation}
u(x,t)=u_{L}\left(\gamma\left(x_{1}-vt\right),x_{2},x_{3},\gamma\left(t-vx_{1}\right)\right).\label{eq:ILcoordinate}
\end{equation}

Let $w_{1},\,\ldots,\,w_{m}$ and $m_{1},\,\ldots,\,m_{\ell}$ be
the normalized bound states of $H_{1}$ and $H_{2}$ associated with
the negative eigenvalues $-\lambda_{1}^{2},\,\ldots,\,-\lambda_{m}^{2}$
and $-\mu_{1}^{2},\,\ldots,\,-\mu_{\ell}^{2}$ respectively (notice
that by our assumptions, $0$ is not an eigenvalue). In other words,
we assume 
\begin{equation}
H_{1}w_{i}=-\lambda_{i}^{2}w_{i},\,\,\,w_{i}\in L^{2},\,\lambda_{i}>0.
\end{equation}

\begin{equation}
H_{2}m_{i}=-\mu_{i}^{2}m_{i},\,\,\,m_{i}\in L^{2},\,\mu_{i}>0.
\end{equation}
We denote by $P_{b}\left(H_{1}\right)$ and $P_{b}\left(H_{2}\right)$
the projections on the the bound states of $H_{1}$ and $H_{2}$,
respectively, and let $P_{c}\left(H_{i}\right)=Id-P_{b}\left(H_{i}\right),\,i=1,2$.
To be more explicit, we have 
\begin{equation}
P_{b}\left(H_{1}\right)=\sum_{i=1}^{m}\left\langle \cdot,w_{i}\right\rangle w_{i},\,\,\,\,\,P_{b}\left(H_{2}\right)=\sum_{j=1}^{\ell}\left\langle \cdot,m_{j}\right\rangle m_{j}.
\end{equation}
In order to study the equation with time-dependent potentials, we
need to introduce a suitable projection. Again, with Lorentz transformations
$L$ associated with the moving potential $V_{2}(x-\vec{v}t)$, we use the
subscript $L$ to denote a function under the new frame $\left(x',t'\right)$. 
\begin{defn}[Scattering states]
\label{AO}Let 
\begin{equation}
\partial_{tt}u-\Delta u+V_{1}(x)u+V_{2}(x-\vec{v}t)u=0,\label{eq:eqBSsec-1}
\end{equation}
with initial data
\begin{equation}
u(x,0)=g(x),\,u_{t}(x,0)=f(x).
\end{equation}
If $u$ also satisfies 
\begin{equation}
\left\Vert P_{b}\left(H_{1}\right)u(t)\right\Vert _{L_{x}^{2}}\rightarrow0,\,\,\left\Vert P_{b}\left(H_{2}\right)u_{L}(t')\right\Vert _{L_{x'}^{2}}\rightarrow0\,\,\,t,t'\rightarrow\infty,\label{eq:ao2-1}
\end{equation}
we call it a scattering state. 
\end{defn}
\begin{rem}
Clearly, the set of $\left(g,\,f\right)\in H^{1}\left(\mathbb{R}^{3}\right)\times L^{2}\left(\mathbb{R}^{3}\right)$
which produces a scattering state forms a subspace of $H^{1}\left(\mathbb{R}^{3}\right)\times L^{2}\left(\mathbb{R}^{3}\right)$.
We will see a detailed discussion on this subspace later on in Section
\ref{sec:Inhom}.
\end{rem}
\begin{rem}
Notice that in order to perform Lorentz transformations, one needs
the existence of global solutions. The existence and the uniqueness
of global solutions to wave equatioans with more general time-dependent
potentials are presented by contraction arguments in \cite{GC2}.
\end{rem}
With the above preparations, we state our main results. First of all,
we have Strichartz estimates:
\begin{thm}[Strichartz estimates]
\label{thm:StriCharB-1}Suppose $u$ is a scattering
state in the sense of Definition \ref{AO} which solves 
\begin{equation}
\partial_{tt}u-\Delta u+V_{1}(x)u+V_{2}(x-\vec{v}t)u=0
\end{equation}
with initial data
\begin{equation}
u(x,0)=g(x),\,u_{t}(x,0)=f(x).
\end{equation}
Then for $p>2$ and $(p,q)$ satisfying 
\begin{equation}
\frac{1}{2}=\frac{1}{p}+\frac{3}{q},
\end{equation}
we have
\begin{equation}
\|u\|_{L_{t}^{p}\left([0,\infty),\,L_{x}^{q}\right)}\lesssim\|f\|_{L^{2}}+\|g\|_{\dot{H}^{1}}.
\end{equation}
\end{thm}
The above theorem is extended to the inhomogeneous case in Section
\ref{sec:Inhom}. 

\smallskip

As in \cite{MNNO}, if we all the norm to be inhomogeneous with respect
to the radial and angular variables, one can recover the endpoint
Strichartz estimate:
\begin{thm}[Endpoint Strichartz estimate]
\label{thm:EndStri-1}Let $\left|v\right|<1$. Suppose $u$ is a scattering
state in the sense of Definition \ref{AO} which solves 
\begin{equation}
\partial_{tt}u-\Delta u+V_{1}(x)u+V_{2}(x-\vec{v}t)u=0
\end{equation}
with initial data
\begin{equation}
u(x,0)=g(x),\,u_{t}(x,0)=f(x).
\end{equation}
Then for $1\leq p<\infty$, 
\begin{equation}
\|u\|_{L_{t}^{2}\left([0,\infty),\,L_{r}^{\infty}L_{\omega}^{p}\right)}\lesssim\|f\|_{L^{2}}+\|g\|_{\dot{H}^{1}}\label{eq:EndStri-1}
\end{equation}
\end{thm}
Next, one has the energy estimate:
\begin{thm}[Energy estimate]
\label{thm:EnergyCharge-1} Suppose $u$ is a scattering
state in the sense of Definition \ref{AO} which solves 
\begin{equation}
\partial_{tt}u-\Delta u+V_{1}(x)u+V_{2}(x-\vec{v}t)u=0
\end{equation}
with initial data
\begin{equation}
u(x,0)=g(x),\,u_{t}(x,0)=f(x).
\end{equation}
Then we have
\begin{equation}
\sup_{t\geq0}\left(\|\nabla u(t)\|_{L^{2}}+\|u_{t}(t)\|_{L^{2}}\right)\lesssim\|f\|_{L^{2}}+\|g\|_{\dot{H}^{1}}.\label{eq:StriCharWOB-1}
\end{equation}
\end{thm}
Associated with the energy estimate, we also have the local energy
decay:
\begin{thm}[Local energy decay]
\label{thm:LEnergyCharge-1} Suppose $u$ is
a scattering state in the sense of Definition \ref{AO} which solves
\begin{equation}
\partial_{tt}u-\Delta u+V_{1}(x)u+V_{2}(x-\vec{v}t)u=0
\end{equation}
with initial data
\begin{equation}
u(x,0)=g(x),\,u_{t}(x,0)=f(x).
\end{equation}
Then for $\forall\epsilon>0,\,\left|\mu\right|<1$, we have 
\[
\left\Vert \left(1+\left|x-\mu t\right|\right)^{-\frac{1}{2}-\epsilon}\left(\left|\nabla u\right|+\left|u_{t}\right|\right)\right\Vert _{L_{t,x}^{2}}\lesssim_{\mu,\epsilon}\|f\|_{L^{2}}+\|g\|_{\dot{H}^{1}}.
\]
\end{thm}
Even more importantly, we obtain the endpoint reversed Strichartz
estimates for $u$.
\begin{thm}[Endpoint reversed Strichartz estimate]
\label{thm:EndRStriCB-1}Suppose
$u$ is a scattering state in the sense of Definition \ref{AO} which
solves 
\begin{equation}
\partial_{tt}u-\Delta u+V_{1}(x)u+V_{2}(x-\vec{v}t)u=0
\end{equation}
with initial data
\begin{equation}
u(x,0)=g(x),\,u_{t}(x,0)=f(x).
\end{equation}
Then
\begin{equation}
\sup_{x\in\mathbb{R}^{3}}\int_{0}^{\infty}\left|u(x,t)\right|^{2}dt\lesssim\left(\|f\|_{L^{2}}+\|g\|_{\dot{H}^{1}}\right)^{2},
\end{equation}
and 
\begin{equation}
\sup_{x\in\mathbb{R}^{3}}\int_{0}^{\infty}\left|u(x+vt,t)\right|^{2}dt\lesssim\left(\|f\|_{L^{2}}+\|g\|_{\dot{H}^{1}}\right)^{2}.
\end{equation}
\end{thm}
With the endpoint estimate along $\left(x+vt,t\right)$, one can derive
the boundedness of the total energy. We denote the total energy of
the system as 
\begin{equation}
E(t)=\int\left|\nabla_{x}u\right|^{2}+\left|\partial_{t}u\right|^{2}+V_{1}\left|u\right|^{2}+V_{2}(x-\vec{v}t)\left|u\right|^{2}dx.
\end{equation}

\begin{cor}[Boundedness of the total energy]
\label{cor:eneCharB-1} Suppose
$u$ is a scattering state in the sense of Definition \ref{AO} which
solves 
\begin{equation}
\partial_{tt}u-\Delta u+V_{1}(x)u+V_{2}(x-\vec{v}t)u=0
\end{equation}
with initial data
\begin{equation}
u(x,0)=g(x),\,u_{t}(x,0)=f(x).
\end{equation}
 Assume 
\begin{equation}
\left\Vert \nabla V_{2}\right\Vert _{L^{1}}<\infty,
\end{equation}
then $E(t)$ is bounded by the initial energy independently of $t$,
\begin{equation}
\sup_{t\geq0}E(t)\lesssim\left\Vert \left(g,f\right)\right\Vert _{\dot{H}^{1}\times L^{2}}^{2}.
\end{equation}
\end{cor}

\subsection{Main ideas}

Here we briefly discuss the main ideas in our analysis and sketch
our proofs. We follow the philosophy from Rodnianski-Schlag \cite{RS}
that \emph{local decay estimates} imply Strichartz estimates. The
main stream of ideas is that \emph{the endpoint Strichartz estimate}
implies weighted estimates, based on which we can derive Strichartz
estimates, energy estimates, the local energy decay, and the boundedness
of the total energy. 

\smallskip

An essential step to approach wave equations with moving potentials
is to understand the change of energy under Lorentz transformations.
In subsection \ref{sec: Lorentz}, we show that the energy along a
space-like slanted line stays comparable to the energy of the initial
data. This in particular implies that under Lorentz transformations,
the initial energy with respect to the new frame is comparable to
the initial energy of the original frame. As a byproduct, we can also
obtain Agmon's estimates for the decay of eigenfunctions. The arguments
hold for all dimensions and even for wave equations with time-dependent
potentials, cf.~\cite{GC2}.

\smallskip

In order to handle time-dependent potentials, we need a time-dependent
weight in the local decay estimate, see Chen \cite{GC1}. More precisely,
we will show that for $\left|v\right|<1$, 
\begin{equation}
\int_{\mathbb{R}}\int_{\mathbb{R}^{3}}\frac{1}{\left\langle x\right\rangle ^{\alpha}}u^{2}(x,t)\,dxdt\lesssim\left\Vert \left(g,f\right)\right\Vert _{\dot{H}^{1}\times L^{2}}^{2}\label{eq:M1}
\end{equation}
and 
\begin{equation}
\int_{\mathbb{R}}\int_{\mathbb{R}^{3}}\frac{1}{\left\langle x-\vec{v}t\right\rangle ^{\alpha}}u^{2}(x,t)\,dxdt\lesssim\left\Vert \left(g,f\right)\right\Vert _{\dot{H}^{1}\times L^{2}}^{2}.\label{eq:M2}
\end{equation}
We notice that for $\alpha>3$, 
\begin{equation}
\int_{\mathbb{R}}\int_{\mathbb{R}^{3}}\frac{1}{\left\langle x\right\rangle ^{\alpha}}u^{2}(x,t)\,dxdt\lesssim\sup_{x\in\mathbb{R}^{3}}\int_{\mathbb{R}}u^{2}(x,t)\,dt
\end{equation}
and 
\begin{equation}
\int_{\mathbb{R}}\int_{\mathbb{R}^{3}}\frac{1}{\left\langle x-\vec{v}t\right\rangle ^{\alpha}}u^{2}(x,t)\,dxdt\lesssim\sup_{x\in\mathbb{R}^{3}}\int_{\mathbb{R}}u^{2}(x+vt,t)\,dt
\end{equation}
which are in the form of \emph{endpoint reversed Strichartz estimates}.
But we also need to \emph{integrate over a time-like slanted line}.
These are carefully analyzed in Section \ref{sec:Slanted}. Intuitively,
the reversed Strichartz estimates are based on the fact that the fundamental
solutions of the wave equation in $\mathbb{R}^{3}$ is supported on
the light cone. For fixed $x$, the propagation will only meet the
light cone once. Meanwhile, away from the light cone, the solution
decays fast. We note that a time-like slanted line will also only
intersect the light cone only once, hence we should have the same
endpoint estimate along it. Our analysis crucially relies on these
types of estimates. Many estimates in Section\ \ref{sec:Slanted}
also hold for more general trajectories provided that their speeds
are strictly less than $1$.

\smallskip

After performing the Lorentz transformation $L$, we have 
\begin{equation}
u_{L}\left(x_{1}',x_{2}',x_{3}',t'\right)=u\left(\gamma\left(x_{1}'+vt'\right),x_{2}',x_{3}',\gamma\left(t'+vx_{1}'\right)\right)
\end{equation}
and 
\begin{equation}
u(x,t)=u_{L}\left(\gamma\left(x_{1}-vt\right),x_{2},x_{3},\gamma\left(t-vx_{1}\right)\right).
\end{equation}
It is crucial to notice that from the expressions above, the standard
endpoint Strichartz estimate for $u$ is equivalent to the endpoint
Strichartz estimate along a slanted line for $u_{L}$ and vice versa.
It is important to note that with the above fact, we can always apply
Lorentz transformations to exchange different frames if we consider
several endpoint Strichartz estimates together. 

\smallskip

Based on the observations above, we apply a bootstrap procedure for
the case with two potentials. Let 
\[
u^{S}\left(x,t\right)=u(x+vt,t).
\]
For a scattering state in the sense of Definition \ref{AO}, we show
that the bootstrap assumptions with big constants $C_{1}(T)$ and
$C_{2}(T)$,

\begin{equation}
\sup_{x\in\mathbb{R}^{3}}\int_{0}^{T}\left|u(x,t)\right|^{2}dt\leq C_{1}(T)\left(\|f\|_{L^{2}}+\|g\|_{\dot{H}^{1}}\right)^{2}
\end{equation}
and
\begin{equation}
\sup_{x\in\mathbb{R}^{3}}\int_{0}^{T}\left|u^{S}\left(x,t\right)\right|^{2}dt\leq C_{2}(T)\left(\|f\|_{L^{2}}+\|g\|_{\dot{H}^{1}}\right)^{2},
\end{equation}
imply 
\begin{equation}
\sup_{x\in\mathbb{R}^{3}}\int_{0}^{T}\left|u(x,t)\right|^{2}dt\leq\left(\widetilde{C}_{1}+\frac{1}{2}C_{1}(T)\right)\left(\|f\|_{L^{2}}+\|g\|_{\dot{H}^{1}}\right)^{2}
\end{equation}
and
\begin{equation}
\sup_{x\in\mathbb{R}^{3}}\int_{0}^{T}\left|u^{S}\left(x,t\right)\right|^{2}dt\leq\left(\widetilde{C}_{2}+\frac{1}{2}C_{2}(T)\right)\left(\|f\|_{L^{2}}+\|g\|_{\dot{H}^{1}}\right)^{2}.
\end{equation}
Then we can conclude 
\begin{equation}
\sup_{x\in\mathbb{R}^{3}}\int_{0}^{T}\left|u(x,t)\right|^{2}dt\leq C_{1}\left(\|f\|_{L^{2}}+\|g\|_{\dot{H}^{1}}\right)^{2}
\end{equation}
and
\begin{equation}
\sup_{x\in\mathbb{R}^{3}}\int_{0}^{T}\left|u^{S}\left(x,t\right)\right|^{2}dt\leq C_{2}\left(\|f\|_{L^{2}}+\|g\|_{\dot{H}^{1}}\right)^{2},
\end{equation}
for some constants $C_{1}$ and $C_{2}$ independent of $T$ by the
bootstrap argument. Therefore, as we pointed out above, we obtain
two local decay estimates \eqref{eq:M1} and \eqref{eq:M2}. To run the
bootstrap argument, we use the fact that the distance between $V_{1}$
and $V_{2}$ becomes larger and larger and both potentials are of
short-range. Therefore, for different regions in $\mathbb{R}^{3}$,
the evolution will be dominated by different Hamiltonians. To make
this intuition precise, in Section \ref{sec:EndRSChar}, we apply
a partition of unity to carry out the decomposition into channels.
For each channel, we use Duhamel's formula to compare the evolution
to the associated dominating Hamiltonian. For every dominating Hamiltonian,
both of the endpoint estimates hold. In each Duhamel expansion, based
on the fact that $V_{1}$ and $V_{2}$ move far away from each other,
it suffices to consider the endpoint estimates of the following integrals, 

\begin{equation}
k_{A}(x,t):=\int_{0}^{t-A}\frac{\sin\left((t-s)\sqrt{H}\right)}{\sqrt{H}}P_{c}F\,ds.
\end{equation}
and 
\[
k_{A}^{S}(x,t):=k_{A}(x+vt,t).
\]
From Section \ref{sec:Slanted}, we have 

\begin{eqnarray}
\left\Vert k_{A}\right\Vert _{L_{x}^{\infty}L_{t}^{2}[A,T]} & = & \left\Vert \int_{0}^{t-A}\frac{\sin\left((t-s)\sqrt{H}\right)}{\sqrt{H}}P_{c}F\,ds\right\Vert _{L_{x}^{\infty}L_{t}^{2}[A,T]}\nonumber \\
 & \lesssim & \frac{1}{A}\left(\left\Vert F\right\Vert _{L_{x}^{1}L_{t}^{2}}+\left\Vert F\right\Vert _{L_{x}^{\frac{3}{2},1}L_{t}^{2}}\right)
\end{eqnarray}
and
\begin{equation}
\left\Vert k_{A}^{S}(x,t)\right\Vert _{L_{x}^{\infty}L_{t}^{2}[A,T]}\lesssim\frac{1}{A}\left(\left\Vert F\right\Vert _{L_{x}^{\frac{3}{2},1}L_{t}^{2}}+\left\Vert F\right\Vert _{L_{x}^{1}L_{t}^{2}}\right).
\end{equation}
Therefore for $A>0$ large but independent of $T$, this term can
be absorbed to the left-hand side to improve our bootstrap assumptions.

From \eqref{eq:M1} and \eqref{eq:M2}, Strichartz estimates follow from
the general scheme introduced in Rodnianski-Schlag \cite{RS,LSch}. 

\subsection*{Notation}

\textquotedblleft $A:=B\lyxmathsym{\textquotedblright}$ or $\lyxmathsym{\textquotedblleft}B=:A\lyxmathsym{\textquotedblright}$
is the definition of $A$ by means of the expression $B$. We use
the notation $\langle x\rangle=\left(1+|x|^{2}\right)^{\frac{1}{2}}$.
The bracket $\left\langle \cdot,\cdot\right\rangle $ denotes the
distributional pairing and the scalar product in the spaces $L^{2}$,
$L^{2}\times L^{2}$ . For positive quantities $a$ and $b$, we write
$a\lesssim b$ for $a\leq Cb$ where $C$ is some prescribed constant.
Also $a\simeq b$ for $a\lesssim b$ and $b\lesssim a$. We denote
$B_{R}(x)$ the open ball of centered at $x$ with radius $R$ in
$\mathbb{R}^{3}$. We also denote by $\chi$ a standard $C^{\infty}$
cut-off function, that is $\chi(x)=1$ for $\left|x\right|\leq1$,
$\chi(x)=0$ for $\left|x\right|>2$ and $0\leq\chi(x)\leq1$ for
$1\leq\left|x\right|\leq2$ .

\subsection*{Organization}

The paper is organized as follows: In Section \ref{sec:Prelim}, we
discuss some preliminary results for the free wave equation and the
wave equation with a stationary potential. We will also analyze the
change of the energy under Lorentz transformations. In Section \ref{sec:Slanted},
estimates of homogeneous and inhomogeneous forms of wave equations
along time-like slanted lines will be discussed. In Section \ref{sec:EndRSChar}
and Section \ref{sec:StrichWOB}, we show Strichartz estimates, energy
estimates, the local energy decay and the boundedness of the total
energy for a scattering state to the wave equation with a charge transfer
Hamiltonian. In order to consider nonlinear applications, in Section
\ref{sec:Inhom} we discuss inhomogeneous Strichartz estimates and local decay estimates. Finally,
in Section \ref{sec:Scattering}, we confirm that a scattering state
indeed scatters to a solution to the free wave equation.

\subsection*{Acknowledgment}

I feel deeply grateful to my advisor Professor Wilhelm Schlag for
suggesting this problem, his kind encouragement, discussions, comments
and all the support. I also want to thank Marius Beceanu for many
useful and enlightening discussions.

\section{Preliminaries\label{sec:Prelim}}

In this section, we present some preliminary results on wave equations
to prepare further discussions in later sections. Throughout, we will
only consider equations in $\mathbb{R}^{3}$. 

\subsection{Strichartz estimates and local energy decay}

We start with the regular Strichartz estimates for free wave equations. 

Consider the free wave equation,

\begin{equation}
\partial_{tt}u-\Delta u=F
\end{equation}
with initial data
\begin{equation}
u(x,0)=g(x),\,u_{t}(x,0)=f(x).
\end{equation}
We can write down the solution using the Fourier transform:
\begin{equation}
u=\frac{\sin\left(t\sqrt{-\Delta}\right)}{\sqrt{-\Delta}}f+\cos\left(t\sqrt{-\Delta}\right)g+\int_{0}^{t}\frac{\sin\left(\left(t-s\right)\sqrt{-\Delta}\right)}{\sqrt{-\Delta}}F(s)\,ds.
\end{equation}
It obeys the energy inequality, 
\begin{equation}
E_{F}(t)=\int_{\mathbb{R}^{3}}\left|\partial_{t}u(t)\right|^{2}+\left|\nabla u(t)\right|^{2}\,dx\lesssim\int_{\mathbb{R}^{3}}\left|f\right|^{2}+\left|\nabla g\right|^{2}\,dx+\int_{0}^{t}\int_{\mathbb{R}^{3}}\left|F(s)\right|^{2}\,dxds.
\end{equation}
 We also have the well-known dispersive estimates for the free wave
equation ($H_{0}=-\Delta$) on $\mathbb{R}^{3}$: 
\begin{equation}
\left\Vert \frac{\sin\left(t\sqrt{-\Delta}\right)}{\sqrt{-\Delta}}f\right\Vert _{L^{\infty}\left(\mathbb{R}^{3}\right)}\lesssim\frac{1}{\left|t\right|}\left\Vert \nabla f\right\Vert _{L^{1}\left(\mathbb{R}^{3}\right)},\label{eq:disper1}
\end{equation}

\begin{equation}
\left\Vert \cos\left(t\sqrt{-\Delta}\right)g\right\Vert _{L^{\infty}\left(\mathbb{R}^{3}\right)}\lesssim\frac{1}{\left|t\right|}\left\Vert \Delta g\right\Vert _{L^{1}\left(\mathbb{R}^{3}\right)}.\label{eq:disper2}
\end{equation}
Notice that the estimate \eqref{eq:disper2} is slightly different from
the estimates commonly in the literature. For example, in Krieger-Schlag
\cite{KS}, one needs the $L^{1}$ norm of $D^{2}g$ instead of $\Delta g$.
One can find a detailed proof in \cite{GC2} based on an idea similar
to the endpoint reversed Strichartz estimate.

Strichartz estimates can be derived abstractly from these dispersive
inequalities and the energy inequality. The following theorem is standard.
One can find a detailed proof in, for example, Keel-Tao \cite{KT}.
\begin{thm}[Strichartz estimate]
\label{thm:StrichF}Suppose 
\begin{equation}
\partial_{tt}u-\Delta u=F
\end{equation}
with initial data 
\begin{equation}
u(x,0)=g(x),\,u_{t}(x,0)=f(x).
\end{equation}
Then for $p,\,a>\frac{2}{s}$, $\left(p,q\right),\,\left(a,b\right)$
satisfying 
\begin{equation}
\frac{3}{2}-s=\frac{1}{p}+\frac{3}{q}
\end{equation}
\begin{equation}
\frac{3}{2}-s=\frac{1}{a}+\frac{3}{b}
\end{equation}
we have
\begin{equation}
\|u\|_{L_{t}^{p}L_{x}^{q}}\lesssim\|g\|_{\dot{H}^{s}}+\|f\|_{\dot{H}^{s-1}}+\left\Vert F\right\Vert _{L_{t}^{a'}L_{x}^{b'}}\label{eq:StrichF}
\end{equation}
where $\frac{1}{a}+\frac{1}{a'}=1,\,\frac{1}{b}+\frac{1}{b'}=1.$
\end{thm}
The endpoint $\left(p,q\right)=\left(2,\infty\right)$ can be recovered
for radial functions in Klainerman-Machedon \cite{KM} for the homogeneous
case and Jia-Liu-Schlag-Xu \cite{JLSX} for the inhomogeneous case.
The endpoint estimate can also be obtained when a small amount of
smoothing (either in the Sobolev sense, or in relaxing the integrability)
is applied to the angular variable, by Machihara-Nakamura-Nakanishi-Ozawa
\cite{MNNO}. 
\begin{thm}[\cite{MNNO}]
\label{thm:inhomAR}For any $1\leq p<\infty$, suppose
$u$ solves the free wave equation 
\begin{equation}
\partial_{tt}u-\Delta u=0
\end{equation}
 with initial data
\begin{equation}
u(x,0)=g(x),\,u_{t}(x,0)=f(x).
\end{equation}
Then
\begin{equation}
\|u\|_{L_{t}^{2}L_{r}^{\infty}L_{\omega}^{p}}\le C(p)\left(\|f\|_{L^{2}}+\|g\|_{\dot{H}^{1}}\right).\label{eq:inhomoAR}
\end{equation}
\end{thm}
The regular Strichartz estimates fail at the endpoint. But if one
switches the order of space-time integration, it is possible to estimate
the solution using the fact that the solution decays quickly away
from the light cone. Therefore, we introduce reversed Strichartz estimates.
Since we will only use the endpoint reversed Stricharz estimate, we
will restrict our focus to that case. The detailed proof for free
equations is presented for the sake of completeness.
\begin{thm}[Endpoint reversed Strichartz estimate]
\label{thm:EndRStrichF}Suppose
\begin{equation}
\partial_{tt}u-\Delta u=F
\end{equation}
with initial data
\begin{equation}
u(x,0)=g(x),\,u_{t}(x,0)=f(x).
\end{equation}
Then 
\begin{equation}
\left\Vert u\right\Vert _{L_{x}^{\infty}L_{t}^{2}}\lesssim\|f\|_{L^{2}}+\|g\|_{\dot{H}^{1}}+\left\Vert F\right\Vert _{L_{x}^{\frac{3}{2},1}L_{t}^{2}},\label{eq:EndRStrichF}
\end{equation}
and for $T>0$, 
\begin{equation}
\left\Vert u\right\Vert _{L_{x}^{\infty}L_{t}^{2}[0,T]}\lesssim\|f\|_{L^{2}}+\|g\|_{\dot{H}^{1}}+\left\Vert F\right\Vert _{L_{x}^{\frac{3}{2},1}L_{t}^{2}[0,T]}.\label{eq:EndRStrichFT}
\end{equation}
\end{thm}
\begin{proof}
Writing down $u$ explicitly, we have 
\begin{equation}
u=\frac{\sin\left(t\sqrt{-\Delta}\right)}{\sqrt{-\Delta}}f+\cos\left(t\sqrt{-\Delta}\right)g+\int_{0}^{t}\frac{\sin\left((t-s)\sqrt{-\Delta}\right)}{\sqrt{-\Delta}}F(s)\,ds.
\end{equation}
We will analyze each term separately. By symmetry, we may assume that
$t\geq0$.

For the first term, 
\begin{equation}
\frac{\sin\left(t\sqrt{-\Delta}\right)}{\sqrt{-\Delta}}f=\frac{1}{4\pi t}\int_{\left|x-y\right|=t}f(y)\,\sigma\left(dy\right).
\end{equation}
So in polar coordinates, 
\begin{eqnarray*}
\left\Vert \frac{\sin\left(t\sqrt{-\Delta}\right)}{\sqrt{-\Delta}}f\right\Vert _{L_{t}^{2}}^{2} & \lesssim & \int_{0}^{\infty}\left(\int_{\mathbb{S}}f(x+r\omega)r\,d\omega\right)^{2}dr\\
 & \lesssim & \left(\int_{0}^{\infty}\int_{\mathbb{S}}f(x+r\omega)^{2}r^{2}\,d\omega dr\right)\left(\int_{\mathbb{S}^{2}}d\omega\right)\\
 & \lesssim & \left\Vert f\right\Vert _{L^{2}}^{2}.
\end{eqnarray*}
Therefore, 
\begin{equation}
\left\Vert \frac{\sin\left(t\sqrt{-\Delta}\right)}{\sqrt{-\Delta}}f\right\Vert _{L_{x}^{\infty}L_{t}^{2}}\lesssim\left\Vert f\right\Vert _{L^{2}}.
\end{equation}

For the second term, 
\begin{eqnarray*}
\left\Vert \cos\left(t\sqrt{-\Delta}\right)g\right\Vert _{L_{t}^{2}}^{2} & = & \int_{0}^{\infty}\left(\int_{\mathbb{S}^{2}}g\left(x+r\omega\right)d\omega+r\partial_{r}g\left(x+r\omega\right)\,d\omega\right)^{2}dr\\
 & \lesssim & \int_{0}^{\infty}\int_{\mathbb{S}^{2}}\left(g\left(x+r\omega\right)d\omega\right)^{2}dr+\int_{0}^{\infty}\int_{\mathbb{S}^{2}}\left(\partial_{r}g\left(x+r\omega\right)d\omega\right)^{2}r^{2}\,dr\\
 & \lesssim & \left(\int_{0}^{\infty}\int_{\mathbb{S}^{2}}g\left(x+r\omega\right)^{2}d\omega dr\right)\left(\int_{\mathbb{S}^{2}}d\omega\right)\\
 &  & +\left(\int_{0}^{\infty}\int_{\mathbb{S}^{2}}\partial_{r}g\left(x+r\omega\right)^{2}d\omega r^{2}dr\right)\left(\int_{\mathbb{S}^{2}}d\omega\right)\\
 & \lesssim & \left\Vert \nabla g\right\Vert _{L^{2}}^{2},
\end{eqnarray*}
where for the last inequality, we applied Hardy's inequality
\begin{equation}
\left\Vert \left|x\right|^{-1}g\right\Vert _{L^{2}}\lesssim\left\Vert g\right\Vert _{\dot{H}^{1}}.
\end{equation}
Therefore, 
\begin{equation}
\left\Vert \cos\left(t\sqrt{-\Delta}\right)g\right\Vert _{L_{x}^{\infty}L_{t}^{2}}\lesssim\left\Vert \nabla g\right\Vert _{L^{2}}.
\end{equation}

For the third term, 
\begin{eqnarray*}
\left\Vert \int_{0}^{t}\frac{\sin\left((t-s)\sqrt{-\Delta}\right)}{\sqrt{-\Delta}}F(s)\,ds\right\Vert _{L_{t}^{2}} & = & \left\Vert \int_{0}^{t}\int_{\left|x-y\right|=t-s}\frac{1}{\left|x-y\right|}F(y,s)\,\sigma\left(dy\right)ds\right\Vert _{L_{t}^{2}}\\
 & = & \left\Vert \int_{\left|x-y\right|\leq t}\frac{1}{\left|x-y\right|}F\left(y,t-\left|x-y\right|\right)\,dy\right\Vert _{L_{t}^{2}}\\
 & \lesssim & \int\frac{1}{\left|x-y\right|}\left\Vert F\left(y,t-\left|x-y\right|\right)\right\Vert _{L_{t}^{2}}dy\\
 & \lesssim & \sup_{x\in\mathbb{R}^{3}}\int\frac{1}{\left|x-y\right|}\left\Vert F\left(y,t\right)\right\Vert _{L_{t}^{2}}dy\\
 & \lesssim & \left\Vert F\right\Vert _{L_{x}^{\frac{3}{2},1}L_{t}^{2}},
\end{eqnarray*}
where we applied Minkowski's inequality in the third line. Here $L^{\frac{3}{2},1}$
is the Lorentz space. In the last inequality, we apply the following
fact: 
\begin{equation}
\sup_{y\in\mathbb{R}^{3}}\int_{\mathbb{R}^{3}}\frac{\left|h(x)\right|}{\left|x-y\right|}\,dx=\sup_{y\in\mathbb{R}^{3}}\int_{\mathbb{R}^{3}}\frac{\left|h(x-y)\right|}{\left|x\right|}\,dx.
\end{equation}
Then for fixed $y$, we apply H\"older's inequality for Lorentz spaces,
see Theorem 3.5 in O'Neil \cite{ON},
\begin{equation}
\int_{\mathbb{R}^{3}}\frac{\left|h(x-y)\right|}{\left|x\right|}\,dx=\left\Vert \frac{\left|h(x-y)\right|}{\left|x\right|}\right\Vert _{L^{1,1}}\lesssim\left\Vert \frac{1}{\left|x\right|}\right\Vert _{L^{3,\infty}}\left\Vert h(x-y)\right\Vert _{L_{x}^{\frac{3}{2},1}}.
\end{equation}
Notice that 
\begin{equation}
\left\Vert h(x-y)\right\Vert _{L_{x}^{\frac{3}{2},1}}=\left\Vert h(x)\right\Vert _{L_{x}^{\frac{3}{2},1}}.
\end{equation}
Therefore, 
\begin{equation}
\sup_{y\in\mathbb{R}^{3}}\int_{\mathbb{R}^{3}}\frac{\left|h(x)\right|}{\left|x-y\right|}\,dx\lesssim\left\Vert h\right\Vert _{L_{x}^{\frac{3}{2},1}}.\label{eq:KatoL}
\end{equation}

Hence 
\begin{equation}
\left\Vert \int_{0}^{t}\frac{\sin\left((t-s)\sqrt{-\Delta}\right)}{\sqrt{-\Delta}}F(s)\,ds\right\Vert _{L_{x}^{\infty}L_{t}^{2}}\lesssim\left\Vert F\right\Vert _{L_{x}^{\frac{3}{2},1}L_{t}^{2}}.
\end{equation}
We also notice that for $T>0$, 
\[
\left\Vert \int_{0}^{t}\frac{\sin\left((t-s)\sqrt{-\Delta}\right)}{\sqrt{-\Delta}}F(s)\,ds\right\Vert _{L_{t}^{2}[0,T]}\lesssim\left\Vert \int_{\left|x-y\right|\leq t}\frac{1}{\left|x-y\right|}F\left(y,t-\left|x-y\right|\right)\,dy\right\Vert _{L_{t}^{2}[0,T]}
\]
with 
\[
0\leq t-\left|x-y\right|\leq T,
\]
whence
\begin{eqnarray*}
\left\Vert \int_{0}^{t}\frac{\sin\left((t-s)\sqrt{-\Delta}\right)}{\sqrt{-\Delta}}F(s)\,ds\right\Vert _{L_{t}^{2}[0,T]} & \lesssim & \left\Vert \int_{\left|x-y\right|\leq t}\frac{1}{\left|x-y\right|}F\left(y,t-\left|x-y\right|\right)\,dy\right\Vert _{L_{t}^{2}[0,T]}\\
 & \lesssim & \int\frac{1}{\left|x-y\right|}\left\Vert F\left(y,t\right)\right\Vert _{L_{t}^{2}[0,T]}dy\\
 & \lesssim & \left\Vert F\right\Vert _{L_{x}^{\frac{3}{2},1}L_{t}^{2}[0,T]}.
\end{eqnarray*}
Therefore, 
\begin{equation}
\left\Vert \int_{0}^{t}\frac{\sin\left((t-s)\sqrt{-\Delta}\right)}{\sqrt{-\Delta}}F(s)\,ds\right\Vert _{L_{t}^{2}[0,T]}\lesssim\left\Vert F\right\Vert _{L_{x}^{\frac{3}{2},1}L_{t}^{2}[0,T]}.\label{eq:truEnd}
\end{equation}
The theorem is proved.
\end{proof}
The above results from Theorem \ref{thm:StrichF} and Theorem \ref{thm:EndRStrichF}
can be generalized to wave equations with real stationary potentials. 

Denote
\begin{equation}
H=-\Delta+V,
\end{equation}
where the potential $V$ satisfies the assumption in Definition \ref{def: Charge}.

Consider the wave equation with potential in $\mathbb{R}^{3}$:

\begin{equation}
\partial_{tt}u-\Delta u+Vu=0
\end{equation}
with initial data 
\begin{equation}
u(x,0)=g(x),\,u_{t}(x,0)=f(x).
\end{equation}
One can write down the solution to it explicitly: 
\begin{equation}
u=\frac{\sin\left(t\sqrt{H}\right)}{\sqrt{H}}f+\cos\left(t\sqrt{H}\right)g.
\end{equation}
Let $P_{b}$ be the projection onto the point spectrum of $H$, $P_{c}=I-P_{b}$
be the projection onto the continuous spectrum of $H$. 

With the above setting, we formulate the results from \cite{BecGo}.
\begin{thm}[Strichartz and reversed Strichartz estimates]
\label{thm:PStriRStrich}Suppose
$H$ has neither eigenvalues nor resonances at zero. Then for all
$0\leq s\leq1$, $p>\frac{2}{s}$, and $\left(p,q\right)$ satisfying
\begin{equation}
\frac{3}{2}-s=\frac{1}{p}+\frac{3}{q}
\end{equation}
we have
\begin{equation}
\left\Vert \frac{\sin\left(t\sqrt{H}\right)}{\sqrt{H}}P_{c}f+\cos\left(t\sqrt{H}\right)P_{c}g\right\Vert _{L_{t}^{p}L_{x}^{q}}\lesssim\|g\|_{\dot{H}^{s}}+\|f\|_{\dot{H}^{s-1}}.\label{PSrich}
\end{equation}
For the endpoint reversed Strichartz estimate, we have 
\begin{equation}
\left\Vert \frac{\sin\left(t\sqrt{H}\right)}{\sqrt{H}}P_{c}f+\cos\left(t\sqrt{H}\right)P_{c}g\right\Vert _{L_{x}^{\infty}L_{t}^{2}}\lesssim\|f\|_{L^{2}}+\|g\|_{\dot{H}^{1}},\label{eq:PEndRSch}
\end{equation}
\begin{equation}
\left\Vert \int_{0}^{t}\frac{\sin\left((t-s)\sqrt{H}\right)}{\sqrt{H}}P_{c}F(s)\,ds\right\Vert _{L_{x}^{\infty}L_{t}^{2}}\lesssim\left\Vert F\right\Vert _{L_{x}^{\frac{3}{2},1}L_{t}^{2}},\label{eq:PEndRSIn}
\end{equation}
and for $T>0$, 
\begin{equation}
\left\Vert \int_{0}^{t}\frac{\sin\left((t-s)\sqrt{H}\right)}{\sqrt{H}}P_{c}F(s)\,ds\right\Vert _{L_{x}^{\infty}L_{t}^{2}[0,T]}\lesssim\left\Vert F\right\Vert _{L_{x}^{\frac{3}{2},1}L_{t}^{2}[0,T]}.\label{eq:PEndRSIT}
\end{equation}
\end{thm}
One can find detailed arguments and more estimates in \cite{BecGo}.
The above theorem can also be established by passing the estimates
for free wave equations in Theorem \ref{thm:StrichF} and Theorem
\ref{thm:EndRStrichF} to the perturbed case via the structure of
wave operators. This general strategy is discussed in detail in \cite{GC2}.
\begin{rem}
In \cite{BecGo}, the above theorem is shown for potentials $V$ with
a finite global Kato norm. The Kato space $K$ is the Banach space
of measures with the property that 
\begin{equation}
\left\Vert V\right\Vert _{K}=\sup_{y\in\mathbb{R}^{3}}\int_{\mathbb{R}^{3}}\frac{\left|V(x)\right|}{\left|x-y\right|}\,dx.
\end{equation}
They consider the space of potentials $V$ which are taken in the
Kato norm closure of the set of bounded, compactly supported functions,
which is denoted by $K_{0}$. Note that from estimate \eqref{eq:KatoL},
$L_{x}^{\frac{3}{2},1}\subset K$ .
\end{rem}
Next, we formulate one fundamental mechanism of wave equations: local
energy decay. It suffices to consider the half-wave operator.
\begin{thm}[Local energy decay]
\label{thm:local}$\forall\epsilon>0$, one has
\begin{equation}
\left\Vert \left(1+\left|x\right|\right)^{-\frac{1}{2}-\epsilon}e^{it\sqrt{-\Delta}}f\right\Vert _{L_{t,x}^{2}}\lesssim_{\epsilon}\left\Vert f\right\Vert _{L_{x}^{2}}.\label{eq:fullwave}
\end{equation}
\end{thm}
See Corollary \ref{thm:local-1} for a more general formulation with
time-dependent weight. A detailed proof can be found in the appendices
in \cite{GC2}. 

The following Christ-Kiselev Lemma is important in our derivation
of Strichartz estimates.
\begin{lem}[Christ-Kiselev]
\label{lem:Christ-Kiselev} Let $X$, $Y$ be two
Banach spaces and let $T$ be a bounded linear operator from $L^{\beta}\left(\mathbb{R}^{+};X\right)$
to $L^{\gamma}\left(\mathbb{R}^{+};Y\right)$, such that 
\begin{equation}
Tf(t)=\int_{0}^{\infty}K(t,s)f(s)\,ds.
\end{equation}
Then the operator 
\begin{equation}
\widetilde{T}f=\int_{0}^{t}K(t,s)f(s)\,ds
\end{equation}
 is bounded from $L^{\beta}\left(\mathbb{R}^{+};X\right)$ to $L^{\gamma}\left(\mathbb{R}^{+};Y\right)$
provided $\beta<\gamma$, and the 
\begin{equation}
\left\Vert \widetilde{T}\right\Vert \leq C(\beta,\gamma)\left\Vert T\right\Vert 
\end{equation}
with 
\begin{equation}
C(\beta,\gamma)=\left(1-2^{\frac{1}{\gamma}-\frac{1}{\beta}}\right)^{-1}.
\end{equation}
\end{lem}

\subsection{Lorentz Transformations and Energy\label{sec: Lorentz}}

In this paper, Lorentz transformations will be important for us to
reduce some estimates to stationary cases. In order to approach our
problem from the viewpoint of Lorentz transformations, the first natural
step is to understand the change of energy under Lorentz transformations.
In this subsection, we show that under Lorentz transformations, the
energy stays comparable to that of the initial data. Recall that after
we apply the Lorentz transformation, for function $u$, under the
new coordinates, we denote
\begin{equation}
u_{L}\left(x_{1}',x_{2}',x_{3}',t'\right)=u\left(\gamma\left(x_{1}'+vt'\right),x_{2}',x_{3}',\gamma\left(t'+vx_{1}'\right)\right).\label{eq:l6}
\end{equation}
Now let $u$  be a solution to some wave equation and set $t'=0$.
We notice that in order to show under Lorentz transformations, the
energy stays comparable to that of the initial data up to an absolute
constant, it suffices to prove 
\begin{eqnarray}
\int\left|\nabla_{x}u\left(x_{1},x_{2},x_{3},vx_{1}\right)\right|^{2}+\left|\partial_{t}u\left(x_{1},x_{2},x_{3},vx_{1}\right)\right|^{2}dx\nonumber \\
\simeq\int\left|\nabla_{x}u\left(x_{1},x_{2},x_{3},0\right)\right|^{2}+\left|\partial_{t}u\left(x_{1},x_{2},x_{3},0\right)\right|^{2}dx.
\end{eqnarray}
provided $\left|v\right|<1$.

Throughout this subsection, we will assume all functions are smooth
and decay fast. We will obtain estimates independent of the additional
smoothness assumption. It is easy to pass the estimates to general
cases with a density argument.
\begin{rem}
\label{rem:dim}One can observe that all discussions in this section
hold for $\mathbb{R}^{n}$. We choose $n=3$ since we will only consider
the charge transfer model in $\mathbb{R}^{3}$ in later parts of this
paper. 
\end{rem}
In this subsection, a more general situation is analyzed. We consider
wave equations with time-dependent potentials
\begin{equation}
\partial_{tt}u-\Delta u+V(x,t)u=0
\end{equation}
under some uniform decay conditions
\begin{equation}
\left|V(x,\mu x_{1})\right|\lesssim\frac{1}{\left\langle x\right\rangle ^{3}}
\end{equation}
uniformly for $0\leq\left|\mu\right|\leq1$. These in particular
apply to wave equations with moving potentials with speed strictly
less than $1$. For example,
\begin{equation}
V(x,t)=V(x-\vec{v}t)
\end{equation}
 with 
\begin{equation}
\left|V(x)\right|\lesssim\frac{1}{\left\langle x\right\rangle ^{3}}.
\end{equation}

\begin{thm}
\label{thm:generalC}Let $\left|v\right|<1$. Suppose 
\begin{equation}
\partial_{tt}u-\Delta u+V(x,t)u=0\label{eq:equationcom}
\end{equation}
and 
\begin{equation}
\left|V(x,\mu x_{1})\right|\lesssim\frac{1}{\left\langle x\right\rangle ^{3}}
\end{equation}
 uniformly with respect to $0\leq\left|\mu\right|<1$. Then 

\begin{eqnarray}
\int\left|\nabla_{x}u\left(x_{1},x_{2},x_{3},vx_{1}\right)\right|^{2}+\left|\partial_{t}u\left(x_{1},x_{2},x_{3},vx_{1}\right)\right|^{2}dx\nonumber \\
\simeq\int\left|\nabla_{x}u\left(x_{1},x_{2},x_{3},0\right)\right|^{2}+\left|\partial_{t}u\left(x_{1},x_{2},x_{3},0\right)\right|^{2}dx,\label{eq:generalC}
\end{eqnarray}
where the implicit constant depends on $v$ and $V$.
\end{thm}
\begin{proof}
Up to performing a Lorentz transformation or a change of variable,
it suffices to show 
\begin{align}
\int\left|\nabla_{x}u\left(x_{1},x_{2},x_{3},vx_{1}\right)\right|^{2}+\left|\partial_{t}u\left(x_{1},x_{2},x_{3},vx_{1}\right)\right|^{2}dx\nonumber \\
\lesssim\int\left|\nabla_{x}u\left(x_{1},x_{2},x_{3},0\right)\right|^{2}+\left|\partial_{t}u\left(x_{1},x_{2},x_{3},0\right)\right|^{2}dx.\label{eq:upperbound}
\end{align}
Set 
\begin{equation}
E_{1}(\mu)=\int\left|\nabla_{x}u\left(x_{1},x_{2},x_{3},\mu x_{1}\right)\right|^{2}dx,\label{eq:E1}
\end{equation}
\begin{equation}
E_{2}(\mu)=\int\left|\partial_{t}u\left(x_{1},x_{2},x_{3},\mu x_{1}\right)\right|^{2}dx.\label{eq:E2}
\end{equation}
In the following computations, for a function $f(x,t)$, we use the
short-hand notation 
\[
\int f\,dx=\int_{\mathbb{R}^{3}}f(x_{1},x_{2},x_{3},\mu x_{1})\,dx.
\]
Then 
\begin{align}
\frac{dE_{1}}{d\mu} & =2\int x_{1}\nabla_{x}u\left(x_{1},x_{2},x_{3},\mu x_{1}\right)\nabla_{x}u_{t}\left(x_{1},x_{2},x_{3},\mu x_{1}\right)\,dx\nonumber \\
 & =2\int x_{1}\nabla_{x}u\nabla_{x}u_{t}\,dx\label{eq:E11}
\end{align}
and 
\begin{align}
\frac{dE_{2}}{d\mu} & =2\int x_{1}\partial_{t}u\left(x_{1},x_{2},x_{3},\mu x_{1}\right)\partial_{tt}u\left(x_{1},x_{2},x_{3},\mu x_{1}\right)\,dx\nonumber \\
 & =2\int x_{1}\partial_{t}u\partial_{tt}u\,dx.\label{eq:E22}
\end{align}
Integration by parts in \eqref{eq:E11} gives 
\begin{align}
\frac{dE_{1}}{d\mu} & =-2\int\partial_{x_{1}}u\cdot u_{t}\,dx-2\int x_{1}\Delta u\cdot u_{t}\,dx-2\mu\int x_{1}\partial_{x_{1}}u_{t}\cdot u_{t}\,dx.\label{eq:E111}
\end{align}
And using the fact that $u$ solves the wave equation implies 
\begin{equation}
\frac{dE_{2}}{d\mu}=2\int x_{1}\partial_{t}u\cdot\Delta u\,dx-2\int x_{1}\partial_{t}u\cdot Vu\,dx.\label{eq:E222}
\end{equation}
Consider the following integral appearing as the last term in \eqref{eq:E111},
\begin{equation}
\int x_{1}\partial_{x_{1}}u_{t}\cdot u_{t}\,dx.\label{eq:Err}
\end{equation}
Integration by parts in $x$, one has 
\[
\int x_{1}\partial_{x_{1}}u_{t}\cdot u_{t}\,dx=-\int\left|u_{t}\right|^{2}dx-\int x_{1}\partial_{x_{1}}u_{t}\cdot u_{t}\,dx-\mu\int x_{1}u_{t}\cdot u_{tt}\,dx.
\]
Therefore, 
\begin{equation}
\int x_{1}\partial_{x_{1}}u_{t}\cdot u_{t}\,dx=-\frac{1}{2}\int\left|u_{t}\right|^{2}dx-\frac{\mu}{4}\frac{dE_{2}}{d\mu}.\label{eq:Err1}
\end{equation}
Combining identities \eqref{eq:E111}, \eqref{eq:E222} and \eqref{eq:Err1}
together, we have
\[
E_{1}^{'}\left(\mu\right)+\left(1-\frac{\mu^{2}}{2}\right)E_{2}^{'}\left(\mu\right)=H\left(\mu\right),
\]
where 
\begin{equation}
H\left(\mu\right)=-2\int\partial_{x_{1}}u\cdot u_{t}\,dx-2\int x_{1}\partial_{t}u\cdot Vu\,dx+\mu\int\left|u_{t}\right|^{2}dx.\label{eq:H}
\end{equation}
By Cauchy-Schwarz and Hardy's inequality,
\[
\left|H(\mu)\right|\lesssim E_{1}(\mu)+E_{2}(\mu),
\]
and hence 
\begin{equation}
\left|E_{1}^{'}\left(\mu\right)+\left(1-\frac{\mu^{2}}{2}\right)E_{2}^{'}\left(\mu\right)\right|\lesssim E_{1}(\mu)+E_{2}(\mu).\label{eq:Gron1}
\end{equation}
Setting 
\[
E_{3}(\mu)=E_{1}\left(\mu\right)+\left(1-\frac{\mu^{2}}{2}\right)E_{2}\left(\mu\right),
\]
one has 
\[
E_{3}^{'}(\mu)=E_{1}^{'}(\mu)+\left(1-\frac{\mu^{2}}{2}\right)E_{2}^{'}(\mu)-\mu E_{2}(\mu)
\]
and 
\[
\left|E_{3}^{'}(\mu)\right|\lesssim E_{1}(\mu)+E_{2}(\mu)+\mu E_{2}(\mu)\lesssim E_{1}(\mu)+E_{2}(\mu)
\]
by \eqref{eq:Gron1}.

Since $0\leq\mu<1$, 
\begin{equation}
E_{1}(\mu)+E_{2}(\mu)\lesssim E_{1}\left(\mu\right)+\left(1-\frac{\mu^{2}}{2}\right)E_{2}\left(\mu\right)=E_{3}(\mu),\label{eq:Gron2}
\end{equation}
so
\begin{equation}
\left|E_{3}^{'}(\mu)\right|\lesssim E_{3}(\mu).\label{eq:Gron3}
\end{equation}
Applying Gr\"onwall's inequality,
\begin{equation}
E_{1}(\mu)+E_{2}(\mu)\lesssim E_{3}(\mu)\lesssim e^{\mu}E_{3}(0)\lesssim E_{1}(0)+E_{2}(0).\label{eq:E1E2E3}
\end{equation}

Therefore, by the definitions of $E_{1}(\mu)$ and $E(\mu)$, we have
\begin{align*}
\int\left|\nabla_{x}u\left(x_{1},x_{2},x_{3},vx_{1}\right)\right|^{2}+\left|\partial_{t}u\left(x_{1},x_{2},x_{3},vx_{1}\right)\right|^{2}dx\\
\lesssim\int\left|\nabla_{x}u\left(x_{1},x_{2},x_{3},0\right)\right|^{2}+\left|\partial_{t}u\left(x_{1},x_{2},x_{3},0\right)\right|^{2}dx.
\end{align*}
The theorem is proved.
\end{proof}
\begin{rem}
The above theorem can be also obtained by local energy conservation
and the control of the energy flux. And this approach will only require
the potential to decay with rate $\left\langle x\right\rangle ^{-2}$.
See \cite{GC2} for more details.
\end{rem}
Applying Theorem \ref{thm:generalC} in the setting of Theorem \ref{thm:local},
we obtain a more general formulation of the local energy decay estimate.
\begin{cor}
\label{thm:local-1} $\forall\epsilon>0\,\left|\vec{\mu}\right|<1$,
 one has 
\begin{equation}
\left\Vert \left(1+\left|x-\vec{\mu}t\right|\right)^{-\frac{1}{2}-\epsilon}e^{it\sqrt{-\Delta}}f\right\Vert _{L_{t,x}^{2}}\lesssim_{\epsilon}\left\Vert f\right\Vert _{L_{x}^{2}}.\label{eq:fullwave-1}
\end{equation}
\end{cor}
As a by product of Theorem \ref{thm:generalC}, we obtain Agmon's
estimates \cite{Agmon} for the decay of eigenfunctions associated
with negative eigenvalues. One can find a detailed proof in \cite{GC2}.

\section{Estimates along Slanted Lines\label{sec:Slanted}}

In order to obtain reversed Strichartz estimates for wave equations
with moving potentials, we need to understand the analogous estimates
along slanted lines. With the results from subsection \ref{sec: Lorentz},
we first consider the estimates along slanted lines for free wave
equations. For the free evolution, the results can be obtained by
explicit calculations with the Kirchhoff formula or the Fourier transforms,
for example see the calculations in \cite{GC2}. In this section,
we will approach those estimates with a viewpoint of Lorentz transformations.
The reason is that this approach will be more consistent with our
construction later on. 

\subsection{Free wave equations}

First of all, we will consider 
\begin{equation}
\partial_{tt}u-\Delta u=0,
\end{equation}
with initial data 
\begin{equation}
u(x,0)=g,\,u_{t}(x,0)=f(x).
\end{equation}
We can write 
\begin{equation}
u(x,t)=\frac{\sin\left(t\sqrt{-\Delta}\right)}{\sqrt{-\Delta}}f+\cos\left(t\sqrt{-\Delta}\right)g.
\end{equation}
By our preliminary discussions in Theorem \ref{thm:EndRStrichF},
we know 
\begin{equation}
\left\Vert u\right\Vert _{L_{x}^{\infty}L_{t}^{2}}\lesssim\left\Vert f\right\Vert _{L_{x}^{2}}+\left\Vert \nabla g\right\Vert _{L_{x}^{2}}.\label{eq:freeRS}
\end{equation}

We consider an analogous estimate to \eqref{eq:freeRS} along slanted
lines. To be more concrete, we integrate $u^{2}$ along slanted lines
\begin{equation}
(x+vt,t)=\left(x_{1}+vt,\,x_{2},\,x_{3},\,t\right)
\end{equation}
Denote 
\begin{equation}
u^{S}(x,t):=u(x+vt,t),
\end{equation}
we estimate 
\begin{equation}
\left\Vert u^{S}\right\Vert _{L_{x}^{\infty}L_{t}^{2}}.
\end{equation}

\begin{lem}
\label{lem:freesl}Let $\left|v\right|<1$ and suppose $u$ solves
\begin{equation}
\partial_{tt}u-\Delta u=0
\end{equation}
with initial data
\begin{equation}
u(0)=g,\,u_{t}(0)=f.
\end{equation}
Then
\begin{equation}
\left\Vert u^{S}\right\Vert _{L_{x}^{\infty}L_{t}^{2}}\lesssim\|f\|_{L^{2}}+\|g\|_{\dot{H}^{1}}.\label{eq:EndSS}
\end{equation}
\end{lem}
\begin{proof}
Recall that performing the Lorentz transformation with respect to
$v$, in the new frame, one has 
\begin{equation}
u_{L}\left(x_{1}',x_{2}',x_{3}',t'\right):=u\left(\gamma\left(x_{1}'+vt'\right),x_{2}',x_{3}',\gamma\left(t'+vx_{1}'\right)\right)\label{eq:newcoord}
\end{equation}
and
\begin{equation}
\partial_{t't'}u_{L}-\Delta_{x'}u_{L}=0.
\end{equation}
Notice that from \eqref{eq:newcoord}, to estimate the $L_{x}^{\infty}L_{t}^{2}$
norm of 
\begin{equation}
u^{S}=u(x+vt,t),
\end{equation}
 is equivalent to integrating of $u_{L}$ along $t'$ up to a multiplication
of an absolute constant only depending on $v$ and $\gamma$. 

Therefore, by the endpoint reversed Strichartz estimate for $u_{L}$,
we have 
\begin{equation}
\left\Vert u^{S}\right\Vert _{L_{x}^{\infty}L_{t}^{2}}\lesssim\left\Vert u_{L}\right\Vert _{L_{x}^{\infty}L_{t}^{2}}\lesssim\|\partial_{t}u_{L}(0)\|_{L^{2}}+\|u_{L}(0)\|_{\dot{H}^{1}}\lesssim\|f\|_{L^{2}}+\|g\|_{\dot{H}^{1}},
\end{equation}
where in the last inequality, we apply Theorem \ref{thm:generalC}
with $V\equiv0$.
\end{proof}

\subsection{Wave equations with stationary potentials}

In this subsection, we consider the perturbed Hamiltonian, 
\begin{equation}
H=-\Delta+V,\label{eql21-1}
\end{equation}
and the wave equation with potential,
\begin{equation}
\partial_{tt}u+Hu=0\label{eq:l21-1}
\end{equation}
with initial data 
\[
u(x,0)=g,\,u_{t}(x,0)=f.
\]
The results in this section can always be obtained by the related
estimates for the free case via the structure formula of wave operators,
cf.~\cite{GC2}. But in order to make our exposition self-contained,
we will prove all estimates independent of the structure formula.

For simplicity, from now on till the end of this section, we will
assume $g=0$. For the other case, the analysis is similar with $L^{2}$
norm replaced by $\dot{H}^{1}$ norm.
\begin{thm}
\label{thm: persl} Let $\left|v\right|<1$ and set 
\begin{equation}
u(x,t)=\frac{\sin\left(t\sqrt{H}\right)}{\sqrt{H}}P_{c}f.
\end{equation}
Denote 
\begin{equation}
u^{S}(x,t):=u(x+vt,t)
\end{equation}
then
\begin{equation}
\left\Vert u^{S}\right\Vert _{L_{x}^{\infty}L_{t}^{2}}\lesssim\left\Vert P_{c}f\right\Vert _{L^{2}}\lesssim\left\Vert f\right\Vert _{L^{2}}.\label{eq:PEndrsl}
\end{equation}
\end{thm}
\begin{proof}
By Duhamel's formula, we write 
\begin{eqnarray}
u(x,t) & = & \frac{\sin\left(t\sqrt{-\Delta}\right)}{\sqrt{-\Delta}}P_{c}f-\int_{0}^{t}\frac{\sin\left((t-s)\sqrt{-\Delta}\right)}{\sqrt{-\Delta}}V\frac{\sin\left(s\sqrt{H}\right)}{\sqrt{H}}P_{c}f\,ds,\nonumber \\
 & =: & A+B
\end{eqnarray}
Now consider the estimate along slanted lines. The estimate for $A$
is known from the free evolution, Lemma \ref{lem:freesl}. For the
second term, we use the explicit representation of the free evolution
$\frac{\sin\left(t\sqrt{-\Delta}\right)}{\sqrt{-\Delta}}$.

Set
\begin{equation}
D(\cdot,t)=\int_{0}^{t}\frac{\sin\left((t-s)\sqrt{-\Delta}\right)}{\sqrt{-\Delta}}F(s)\,ds
\end{equation}
along slanted lines. First of all, by our preliminary results, Theorem
\ref{thm:EndRStrichF}, 
\begin{equation}
\left\Vert D\right\Vert _{L_{x}^{\infty}L_{t}^{2}}\lesssim\left\Vert F\right\Vert _{L_{x}^{\frac{3}{2},1}L_{t}^{2}}.\label{eq:ersIn}
\end{equation}
For the estimate along slanted lines, by Kirchhoff's formula, we know
\begin{equation}
D^{S}(x,t):=D(x+vt,t)=\int_{0}^{t}\int_{\left|x+vt-y\right|=t-s}\frac{F(y,s)}{\left|x+vt-y\right|}\,\sigma\left(dy\right)ds
\end{equation}
and
\begin{eqnarray}
\left\Vert D^{S}(x,\cdot)\right\Vert _{L_{t}^{2}} & = & \left\Vert \int_{0}^{t}\int_{\left|x+vt-y\right|=t-s}\frac{F(y,s)}{\left|x+vt-y\right|}\,\sigma\left(dy\right)ds\right\Vert _{L_{t}^{2}}\nonumber \\
 & = & \left\Vert \int_{\left|y\right|\leq t}\frac{F(x+vt-y,t-\left|y\right|)}{\left|y\right|}\,dy\right\Vert _{L_{t}^{2}}\\
 & \leq & \left\Vert \int_{\mathbb{R}^{3}}\frac{\left|F(x-y,t-\left|y+vt\right|)\right|}{\left|y+vt\right|}\,dy\right\Vert _{L_{t}^{2}}\nonumber \\
 & \leq & \left\Vert \int_{\mathbb{R}^{3}}\frac{\left|F(x-y,t-\left|y+vt\right|)\right|}{\sqrt{y_{2}^{2}+y_{3}^{2}}}\,dy\right\Vert _{L_{t}^{2}},\nonumber 
\end{eqnarray}
where in the third line, we use a change of variable and for the last
inequality and reduce the norm of $y$ to the norm of the component
of $y$ orthogonal to the direction of the motion. 

Finally, 
\begin{equation}
\left\Vert \int_{\mathbb{R}^{3}}\frac{F(x-y,t-\left|y+vt\right|)}{\sqrt{y_{2}^{2}+y_{3}^{2}}}\,dy\right\Vert _{L_{t}^{2}}\leq\int_{\mathbb{R}^{3}}\frac{\left\Vert F(x-y,t-\left|y+vt\right|)\right\Vert _{L_{t}^{2}}}{\sqrt{y_{2}^{2}+y_{3}^{2}}}\,dy
\end{equation}
For fixed $y$, if we apply a change of variable of $t$ here, the
Jacobian is bounded by $1-|v|$ and $1+|v|$, so 
\begin{eqnarray}
\int_{\mathbb{R}^{3}}\frac{\left\Vert F(x-y,t-\left|y+vt\right|)\right\Vert _{L_{t}^{2}}}{\sqrt{y_{2}^{2}+y_{3}^{2}}}\,dy & \lesssim & \int_{\mathbb{R}^{3}}\frac{\left\Vert F(x-y,\cdot)\right\Vert _{L_{t}^{2}}}{\sqrt{y_{2}^{2}+y_{3}^{2}}}dy\nonumber \\
 & \lesssim & \left\Vert F\right\Vert _{L_{x_{1}}^{1}L_{\widehat{x_{1}}}^{2,1}L_{t}^{2}}
\end{eqnarray}
where $\widehat{x_{1}}$ denotes the subspace orthogonal to $x_{1}$
(more generally, the subspace orthogonal to the direction of the motion).
Here $L^{2,1}$ is the Lorentz norm and the last inequality follows
from H\"older's inequality of Lorentz spaces. Therefore, 
\begin{equation}
\left\Vert D^{S}\right\Vert _{L_{x}^{\infty}L_{t}^{2}}\lesssim\left\Vert F\right\Vert _{L_{x_{1}}^{1}L_{\widehat{x_{1}}}^{2,1}L_{t}^{2}}.\label{eq:ersI}
\end{equation}
By a similar discussion as the estimate \eqref{eq:truEnd}, we also
have for $T>0$, 
\begin{equation}
\left\Vert D^{S}\right\Vert _{L_{x}^{\infty}L_{t}^{2}[0,T]}\lesssim\left\Vert F\right\Vert _{L_{x_{1}}^{1}L_{\widehat{x_{1}}}^{2,1}L_{t}^{2}[0,T]}.\label{eq:ersIT}
\end{equation}
With estimate \eqref{eq:ersI}, we know for $u^{S}(x,t):=u(x+vt,t)$,
\begin{eqnarray}
\left\Vert u^{S}\right\Vert _{L_{x}^{\infty}L_{t}^{2}} & \lesssim & \left\Vert P_{c}f\right\Vert _{L^{2}}+\left\Vert V\frac{\sin\left(t\sqrt{H}\right)}{\sqrt{H}}P_{c}f\right\Vert _{L_{x_{1}}^{1}L_{\widehat{x_{1}}}^{2,1}L_{t}^{2}}\nonumber \\
 & \lesssim & \left\Vert P_{c}f\right\Vert _{L^{2}}+\left\Vert V\right\Vert _{L_{x_{1}}^{1}L_{\widehat{x_{1}}}^{2,1}}\left\Vert \frac{\sin\left(t\sqrt{H}\right)}{\sqrt{H}}P_{c}f\right\Vert _{L_{x}^{\infty}L_{t}^{2}}\nonumber \\
 & \lesssim & \left\Vert P_{c}f\right\Vert _{L^{2}}+\left\Vert V\right\Vert _{L_{x_{1}}^{1}L_{\widehat{x_{1}}}^{2,1}}\left\Vert P_{c}f\right\Vert _{L_{x}^{2}}\\
 & \lesssim & \left\Vert f\right\Vert _{L^{2}}.\nonumber 
\end{eqnarray}
where in the third line, we use the endpoint reversed Strichartz estimate
of the wave equation with potentials as Theorem \ref{thm:PStriRStrich}.

Therefore,
\begin{equation}
\left\Vert u\right\Vert _{L_{x}^{\infty}L_{t}^{2}}\lesssim\left\Vert P_{c}f\right\Vert _{L_{x}^{2}}\lesssim\left\Vert f\right\Vert _{L_{x}^{2}}
\end{equation}

\begin{equation}
\left\Vert u^{S}\right\Vert _{L_{x}^{\infty}L_{t}^{2}}\lesssim\left\Vert P_{c}f\right\Vert _{L_{x}^{2}}\lesssim\left\Vert f\right\Vert _{L_{x}^{2}}
\end{equation}
as claimed.
\end{proof}
As a byproduct, we have the following inhomogeneous estimates from
\eqref{eq:ersI} and \eqref{eq:ersIT}.
\begin{cor}
\label{cor:inhomcom}For $\left|v\right|<1$ we have 
\begin{equation}
\left\Vert \int_{0}^{t}\int_{\left|x+vt-y\right|=t-s}\frac{F(y,s)}{\left|x+vt-y\right|}\,\sigma(dy)ds\right\Vert _{L_{x}^{\infty}L_{t}^{2}}\lesssim\left\Vert F\right\Vert _{L_{x_{1}}^{1}L_{\widehat{x_{1}}}^{2,1}L_{t}^{2}},\label{eq:inhomslant}
\end{equation}
and for $T>0$, 
\begin{equation}
\left\Vert \int_{0}^{t}\int_{\left|x+vt-y\right|=t-s}\frac{F(y,s)}{\left|x+vt-y\right|}\,\sigma(dy)ds\right\Vert _{L_{x}^{\infty}L_{t}^{2}[0,T]}\lesssim\left\Vert F\right\Vert _{L_{x_{1}}^{1}L_{\widehat{x_{1}}}^{2,1}L_{t}^{2}[0,T]}.\label{eq:inhomslantT}
\end{equation}
\end{cor}
From the discussion above, we can also obtain the following truncated
versions of inhomogeneous estimates which are crucial in our later
bootstrap arguments.
\begin{cor}
\textup{\label{cor:inhomA}}Suppose\textup{ $A>0$ }and\textup{ $\left|v\right|<1$,}
then\textup{ 
\begin{equation}
\sup_{x}\left\Vert \int_{0}^{t-A}\frac{\sin\left((t-s)\sqrt{-\Delta}\right)}{\sqrt{-\Delta}}\,Fds\right\Vert _{L_{t}^{2}[A,\infty)}\lesssim\frac{1}{A}\left\Vert F\right\Vert _{L_{x}^{1}L_{t}^{2}}.\label{eq:inhomA}
\end{equation}
}and for\textup{ $T>0$, 
\begin{equation}
\sup_{x}\left\Vert \int_{0}^{t-A}\frac{\sin\left((t-s)\sqrt{-\Delta}\right)}{\sqrt{-\Delta}}\,Fds\right\Vert _{L_{t}^{2}[A,T]}\lesssim\frac{1}{A}\left\Vert F\right\Vert _{L_{x}^{1}L_{t}^{2}[0,T]}.\label{eq:inhomAT}
\end{equation}
}Similarly, 
\begin{equation}
\sup_{x}\left\Vert \int_{0}^{t-A}\int_{\left|x+vt-y\right|=t-s}\frac{F(y,s)}{\left|x+vt-y\right|}\,\sigma(dy)ds\right\Vert _{L_{t}^{2}[A,\infty)}\lesssim\frac{1}{A}\left\Vert F\right\Vert _{L_{x}^{1}L_{t}^{2}},\label{eq:inhomASL}
\end{equation}
\emph{and for}\textup{ $T>0$}
\begin{equation}
\sup_{x}\left\Vert \int_{0}^{t-A}\int_{\left|x+vt-y\right|=t-s}\frac{F(y,s)}{\left|x+vt-y\right|}\,\sigma(dy)ds\right\Vert _{L_{t}^{2}[A,T]}\lesssim\frac{1}{A}\left\Vert F\right\Vert _{L_{x}^{1}L_{t}^{2}[0,T]}.\label{eq:inhomASLT}
\end{equation}
\end{cor}
\begin{proof}
By a similar discussion above with Kirchhoff 's formula, 
\begin{eqnarray}
\left\Vert \int_{0}^{t-A}\frac{\sin\left((t-s)\sqrt{-\Delta}\right)}{\sqrt{-\Delta}}F(s)\,ds\right\Vert _{L_{t}^{2}[A,\infty)} & = & \left\Vert \int_{A\leq\left|y\right|\leq t}\frac{F(x-y,t-\left|y\right|)}{\left|y\right|}\,dy\right\Vert _{L_{t}^{2}[A,\infty)}\nonumber \\
 & \lesssim & \int_{A\leq\left|y\right|}\frac{\left\Vert F(x-y,t-\left|y\right|)\right\Vert _{L_{t}^{2}}}{\left|y\right|}\,dy\\
 & \lesssim & \frac{1}{A}\left\Vert F\right\Vert _{L_{x}^{1}L_{t}^{2}}.\nonumber 
\end{eqnarray}
Therefore, 
\begin{equation}
\left\Vert \int_{0}^{t-A}\frac{\sin\left((t-s)\sqrt{-\Delta}\right)}{\sqrt{-\Delta}}F(s)\,ds\right\Vert _{L_{x}^{\infty}L_{t}^{2}[A,\infty)}\lesssim\frac{1}{A}\left\Vert F\right\Vert _{L_{x}^{1}L_{t}^{2}}.
\end{equation}
With the same argument as \eqref{eq:ersIT}, we also have 
\begin{equation}
\left\Vert \int_{0}^{t-A}\frac{\sin\left((t-s)\sqrt{-\Delta}\right)}{\sqrt{-\Delta}}\,F(s)\,ds\right\Vert _{L_{t}^{2}[A,T]}\lesssim\frac{1}{A}\left\Vert F\right\Vert _{L_{x}^{1}L_{t}^{2}[0,T]}.
\end{equation}
Similarly to the way  we derive estimates \eqref{eq:inhomslant} and
\eqref{eq:inhomslantT}, one obtains 
\begin{equation}
\left\Vert \int_{0}^{t-A}\int_{\left|x+vt-y\right|=t-s}\frac{F(y,s)}{\left|x+vt-y\right|}\,\sigma(dy)ds\right\Vert _{L_{x}^{\infty}L_{t}^{2}[A,\infty)}\lesssim\frac{1}{A}\left\Vert F\right\Vert _{L_{x}^{1}L_{t}^{2}},
\end{equation}
\begin{equation}
\left\Vert \int_{0}^{t-A}\int_{\left|x+vt-y\right|=t-s}\frac{F(y,s)}{\left|x+vt-y\right|}\,\sigma(dy)ds\right\Vert _{L_{x}^{\infty}L_{t}^{2}[A,T]}\lesssim\frac{1}{A}\left\Vert F\right\Vert _{L_{x}^{1}L_{t}^{2}[0,T]}.
\end{equation}
We are done.
\end{proof}
Next, we consider estimates in inhomogeneous forms for the perturbed
evolution along slanted lines. In the following proofs, essentially,
we pass the effects caused by the integration along slanted lines
to the free evolution by Duhamel expansion and use the standard case
for the perturbed evolution.

Define 
\begin{equation}
k(\cdot,t):=\int_{0}^{t}\frac{\sin\left((t-s)\sqrt{H}\right)}{\sqrt{H}}P_{c}F(s)\,ds.\label{eq:Pinhom}
\end{equation}
Then from the endpoint reversed Strichartz estimate, Theorem \ref{thm:PStriRStrich},
we have 
\begin{equation}
\left\Vert k\right\Vert _{L_{x}^{\infty}L_{t}^{2}}\lesssim\left\Vert F\right\Vert _{L_{x}^{\frac{3}{2},1}L_{t}^{2}}.\label{eq:endinhom}
\end{equation}

\begin{thm}
\label{thm:perinhomsl}Let $\left|v\right|<1$ and suppose $H=-\Delta+V$
has neither resonances nor eigenfunctions at $0$. Define 
\begin{equation}
k^{S}(x,t)=k(x+vt,t).
\end{equation}
Then we have 
\begin{equation}
\left\Vert k^{S}\right\Vert _{L_{x}^{\infty}L_{t}^{2}}\lesssim\left\Vert F\right\Vert _{L_{x_{1}}^{1}L_{\widehat{x_{1}}}^{2,1}L_{t}^{2}}+\left\Vert F\right\Vert _{L_{x}^{\frac{3}{2},1}L_{t}^{2}},\label{eq:pertinhomsl}
\end{equation}
and for $T>0$, 
\begin{equation}
\left\Vert k^{S}\right\Vert _{L_{x}^{\infty}L_{t}^{2}[0,T]}\lesssim\left\Vert F\right\Vert _{L_{x_{1}}^{1}L_{\widehat{x_{1}}}^{2,1}L_{t}^{2}[0,T]}+\left\Vert F\right\Vert _{L_{x}^{\frac{3}{2},1}L_{t}^{2}[0,T]},\label{eq:pertinhomslT}
\end{equation}
where $\widehat{x_{1}}$ is the subspace orthogonal to to $\vec{e}_{1}$.
\end{thm}
\begin{proof}
By Duhamel's formula, we write 
\begin{eqnarray}
\frac{\sin\left((t-s)\sqrt{H}\right)}{\sqrt{H}}P_{c}F(s) & = & \frac{\sin\left((t-s)\sqrt{-\Delta}\right)}{\sqrt{-\Delta}}P_{c}F(s)\nonumber \\
 &  & -\int_{s}^{t}\frac{\sin\left((t-m)\sqrt{-\Delta}\right)}{\sqrt{-\Delta}}V\frac{\sin\left((m-s)\sqrt{H}\right)}{\sqrt{H}}P_{c}F(s)\,dm.
\end{eqnarray}
Therefore, 
\begin{eqnarray*}
\int_{0}^{t}\frac{\sin\left((t-s)\sqrt{H}\right)}{\sqrt{H}}P_{c}F\,ds & = & \int_{0}^{t}\frac{\sin\left((t-s)\sqrt{-\Delta}\right)}{\sqrt{-\Delta}}F(s)\,ds\\
 &  & -\int_{0}^{t}\int_{s}^{t}\frac{\sin\left((t-m)\sqrt{-\Delta}\right)}{\sqrt{-\Delta}}V\frac{\sin\left((m-s)\sqrt{H}\right)}{\sqrt{H}}P_{c}F(s)\,dmds.\qquad
\end{eqnarray*}
Denote 
\begin{equation}
R(x,t):=\int_{0}^{t}\int_{s}^{t}\frac{\sin\left((t-m)\sqrt{-\Delta}\right)}{\sqrt{-\Delta}}V\frac{\sin\left((m-s)\sqrt{H}\right)}{\sqrt{H}}P_{c}F(s)\,dmds
\end{equation}
and 
\begin{equation}
R^{S}(x,t):=R(x+vt,t).
\end{equation}
Then 
\begin{equation}
\left\Vert k^{S}\right\Vert _{L_{x}^{\infty}L_{t}^{2}}\lesssim\left\Vert D^{S}\right\Vert _{L_{x}^{\infty}L_{t}^{2}}+\left\Vert R^{S}\right\Vert _{L_{x}^{\infty}L_{t}^{2}},
\end{equation}
where 
\begin{equation}
D^{S}(x,t)=D(x+vt,t)=\int_{0}^{t}\int_{\left|x+vt-y\right|=t-s}\frac{F(y,s)}{\left|x+vt-y\right|}\,dyds.
\end{equation}
From Corollary \ref{cor:inhomcom}, we know 
\begin{equation}
\left\Vert D^{S}\right\Vert _{L_{x}^{\infty}L_{t}^{2}}\lesssim\left\Vert F\right\Vert _{L_{x_{1}}^{1}L_{\widehat{x_{1}}}^{2,1}L_{t}^{2}}.
\end{equation}
To estimate 
\begin{equation}
\left\Vert \int_{0}^{t}\int_{s}^{t}\frac{\sin\left((t-k)\sqrt{-\Delta}\right)}{\sqrt{-\Delta}}V\frac{\sin\left((k-s)\sqrt{H}\right)}{\sqrt{H}}P_{c}F(s)\,dkds\right\Vert _{L_{t}^{2}},
\end{equation}
we notice that with an exchange of the order of integration, 
\begin{eqnarray}
R(x,t)=\int_{0}^{t}\int_{s}^{t}\frac{\sin\left((t-k)\sqrt{-\Delta}\right)}{\sqrt{-\Delta}}V\frac{\sin\left((k-s)\sqrt{H}\right)}{\sqrt{H}}P_{c}F(s)\,dkds\quad\quad\quad\nonumber \\
=\int_{0}^{t}\frac{\sin\left((t-k)\sqrt{-\Delta}\right)}{\sqrt{-\Delta}}\left(\int_{0}^{k}V\frac{\sin\left((k-s)\sqrt{H}\right)}{\sqrt{H}}P_{c}F(s)\,ds\right)dk.
\end{eqnarray}
Then applying our estimate for the free evolution estimate in the
inhomogeneous case, Corollary \ref{cor:inhomcom}, we have
\begin{eqnarray}
\left\Vert R^{S}(x,t)\right\Vert _{L_{x}^{\infty}L_{t}^{2}}\lesssim\left\Vert \int_{0}^{t}V\frac{\sin\left((t-s)\sqrt{H}\right)}{\sqrt{H}}P_{c}F(s)\,ds\right\Vert _{L_{x_{1}}^{1}L_{\widehat{x_{1}}}^{2,1}L_{t}^{2}}\nonumber \\
\lesssim\left\Vert V\right\Vert _{L_{x_{1}}^{1}L_{\widehat{x_{1}}}^{2,1}}\left\Vert \int_{0}^{t}\frac{\sin\left((t-s)\sqrt{H}\right)}{\sqrt{H}}P_{c}F(s)\,ds\right\Vert _{L_{x}^{\infty}L_{t}^{2}}\\
\lesssim\left\Vert F\right\Vert _{L_{x}^{\frac{3}{2},1}L_{t}^{2}}\qquad\nonumber 
\end{eqnarray}
where in the third inequality, we use the endpoint reversed Strichartz
estimate \eqref{eq:endinhom}.

Therefore, we conclude that 
\begin{eqnarray}
\left\Vert k^{S}\right\Vert _{L_{x}^{\infty}L_{t}^{2}} & \lesssim & \left\Vert D^{S}\right\Vert _{L_{x}^{\infty}L_{t}^{2}}+\left\Vert R^{S}\right\Vert _{L_{x}^{\infty}L_{T}^{2}}\nonumber \\
 & \lesssim & \left\Vert F\right\Vert _{L_{x_{1}}^{1}L_{\widehat{x_{1}}}^{2,1}L_{t}^{2}}+\left\Vert F\right\Vert _{L_{x}^{\frac{3}{2},1}L_{t}^{2}}.
\end{eqnarray}
When we restrict to $[0,T]$, as above, we can obtain 
\begin{equation}
\left\Vert k^{S}\right\Vert _{L_{x}^{\infty}L_{t}^{2}[0,T]}\lesssim\left\Vert F\right\Vert _{L_{x_{1}}^{1}L_{\widehat{x_{1}}}^{2,1}L_{t}^{2}[0,T]}+\left\Vert F\right\Vert _{L_{x}^{\frac{3}{2},1}L_{t}^{2}[0,T]}.
\end{equation}
The lemma is proved.
\end{proof}
To prepare our bootstrap arguments in the later section, similarly
to the free case, we also consider the truncated versions of the above
estimates.

By the same method we used to estimate 
\begin{equation}
\int_{0}^{t}\frac{\sin\left((t-s)\sqrt{H}\right)}{\sqrt{H}}P_{c}F(s)\,ds
\end{equation}
along slanted lines, we obtain the following:
\begin{cor}
\label{cor:perinhomA}For $\left|v\right|<1$ and $A>0$, suppose
$H=-\Delta+V$ has neither resonances nor eigenfunctions at $0$.
Let 
\begin{equation}
k_{A}(\cdot,t):=\int_{0}^{t-A}\frac{\sin\left((t-s)\sqrt{H}\right)}{\sqrt{H}}P_{c}F(s)\,ds.
\end{equation}
Then
\begin{equation}
\left\Vert k_{A}\right\Vert _{L_{x}^{\infty}L_{t}^{2}[A,\infty)}\lesssim\frac{1}{A}\left(\left\Vert F\right\Vert _{L_{x}^{1}L_{t}^{2}}+\left\Vert F\right\Vert _{L_{x}^{\frac{3}{2},1}L_{t}^{2}}\right),\label{eq:perinhomA}
\end{equation}
and for $T>0$, 
\begin{equation}
\left\Vert k_{A}\right\Vert _{L_{x}^{\infty}L_{t}^{2}[A,T]}\lesssim\frac{1}{A}\left(\left\Vert F\right\Vert _{L_{x}^{1}L_{t}^{2}[0,T]}+\left\Vert F\right\Vert _{L_{x}^{\frac{3}{2},1}L_{t}^{2}[0,T]}\right).\label{eq:perinhomAT}
\end{equation}
Define 
\begin{equation}
k_{A}^{S}(x,t):=k_{A}(x+vt,t).
\end{equation}
then 
\begin{equation}
\left\Vert k_{A}^{S}(x,t)\right\Vert _{L_{x}^{\infty}L_{t}^{2}[A,\infty)}\lesssim\frac{1}{A}\left(\left\Vert F\right\Vert _{L_{x}^{\frac{3}{2},1}L_{t}^{2}}+\left\Vert F\right\Vert _{L_{x}^{1}L_{t}^{2}}\right).\label{eq:perinhomAsl}
\end{equation}
and for $T>0$,
\begin{equation}
\left\Vert k_{A}^{S}(x,t)\right\Vert _{L_{x}^{\infty}L_{t}^{2}[A,T]}\lesssim\frac{1}{A}\left(\left\Vert F\right\Vert _{L_{x}^{\frac{3}{2},1}L_{t}^{2}[0,T]}+\left\Vert F\right\Vert _{L_{x}^{1}L_{t}^{2}[0,T]}\right).\label{eq:perinhomAslT}
\end{equation}
\end{cor}
Finally, in order to handle moving potentials, we consider some estimates
with inhomogeneous terms along slanted lines:

Setting 
\begin{equation}
F^{S}(x,t)=F\left(x+vt,t\right)
\end{equation}
we have the following results.
\begin{lem}
\label{lem:Frot}Let $A>0$ and $\left|\vec{\mu}\right|<1$, $\left|v\right|<1$.
Suppose 
\begin{equation}
D_{A}(x,t):=\int_{0}^{t-A}\frac{\sin\left((t-s)\sqrt{-\Delta}\right)}{\sqrt{-\Delta}}F(s)\,ds,
\end{equation}
\begin{equation}
D_{A}^{S}(x,t)=g_{A}(x+\vec{\mu}t,t).
\end{equation}
We have
\begin{equation}
\left\Vert D_{A}(x,t)\right\Vert _{L_{x}^{\infty}L_{t}^{2}[A,\infty)}\lesssim\frac{1}{A}\left\Vert F^{S}\right\Vert _{L_{x}^{1}L_{t}^{2}},\label{eq:Frot1}
\end{equation}
\begin{equation}
\left\Vert D_{A}^{S}(x,t)\right\Vert _{L_{x}^{\infty}L_{t}^{2}[A,\infty)}\lesssim\frac{1}{A}\left\Vert F^{S}\right\Vert _{L_{x}^{1}L_{t}^{2}},\label{eq:Frot2}
\end{equation}
and for $T>0$, 
\begin{equation}
\left\Vert D_{A}(x,t)\right\Vert _{L_{x}^{\infty}L_{t}^{2}[A,T]}\lesssim\frac{1}{A}\left\Vert F^{S}\right\Vert _{L_{x}^{1}L_{t}^{2}[0,T]},\label{eq:Frot1T}
\end{equation}
\begin{equation}
\left\Vert D_{A}^{S}(x,t)\right\Vert _{L_{x}^{\infty}L_{t}^{2}[A,T]}\lesssim\frac{1}{A}\left\Vert F^{S}\right\Vert _{L_{x}^{1}L_{t}^{2}[0,T]}.\label{eq:Frot2T}
\end{equation}
\end{lem}
\begin{proof}
We know explicitly, 
\begin{equation}
D_{A}(x,t)=\int_{0}^{t-A}\int_{\left|x-y\right|=t-s}\frac{F(y,s)}{\left|x-y\right|}\,dyds.
\end{equation}
Taking $z=y-sv$, we have
\begin{eqnarray}
\left|D_{A}(x,t)\right|=\left|\int_{0}^{t-A}\int_{\left|x-y\right|=t-s}\frac{F(y,s)}{\left|x-y\right|}\,dyds\right|\nonumber \\
=\left|\int_{0}^{t-A}\int_{\left|x-z-vs\right|=t-s}\frac{F^{S}(z,s)}{\left|x-z-vs\right|}\,dzds\right|\nonumber \\
\lesssim\int_{0}^{t-A}\int_{\left|m\right|=t-s}\frac{\left|F^{S}(x-vs-m,s)\right|}{\left|m\right|}\,dmds\\
\lesssim\int_{0}^{t-A}\int_{\left|m\right|=t-s}\frac{\left|F^{S}(x-v\left(t-\left|m\right|\right)-m,t-\left|m\right|)\right|}{\left|m\right|}\,dmds\nonumber \\
\lesssim\frac{1}{A}\int\left|F^{S}(x-v\left(t-\left|m\right|\right)-m,t-\left|m\right|)\,\right|dm\nonumber 
\end{eqnarray}
In the third line above, we apply a change of variable $m=x-z-vs$
and in the fifth line, we again apply a change of variable $v\left|m\right|+m=h$
with bounded Jabobian.

Therefore, if we set $q=v\left(t-\left|m\right|\right)+m$, we have
\begin{eqnarray}
\left\Vert D(x,\cdot)\right\Vert _{L_{t}^{2}[A,\infty)} & \lesssim & \frac{1}{A}\int\left\Vert F^{S}(x-q,\cdot)\right\Vert _{L_{t}^{2}}\,dq\nonumber \\
 & \lesssim & \frac{1}{A}\left\Vert F^{S}\right\Vert _{L_{x}^{1}L_{t}^{2}}
\end{eqnarray}
The estimate for $\left\Vert D_{A}^{S}(x,t)\right\Vert _{L_{x}^{\infty}L_{t}^{2}}$
is the same as we did for Corollary \ref{cor:inhomA}. Hence we obtain
\begin{equation}
\left\Vert D_{A}^{S}(x,t)\right\Vert _{L_{x}^{\infty}L_{t}^{2}[A,\infty)}\lesssim\frac{1}{A}\left\Vert F^{S}\right\Vert _{L_{x}^{1}L_{t}^{2}}.
\end{equation}
The the same as above, when we restrict to $[0,T]$, one has
\begin{equation}
\left\Vert D_{A}(x,t)\right\Vert _{L_{x}^{\infty}L_{t}^{2}[A,T]}\lesssim\frac{1}{A}\left\Vert F^{S}\right\Vert _{L_{x}^{1}L_{t}^{2}[0,T]},
\end{equation}
\begin{equation}
\left\Vert D_{A}^{S}(x,t)\right\Vert _{L_{x}^{\infty}L_{t}^{2}[A,T]}\lesssim\frac{1}{A}\left\Vert F^{S}\right\Vert _{L_{x}^{1}L_{t}^{2}[0,T]}.
\end{equation}
The lemma is proved.
\end{proof}
The above lemma can also be established by a duality argument. For
the sake of completeness, we sketch the argument here. We only focus
on 
\[
\left\Vert D_{A}(x,t)\right\Vert _{L_{x}^{\infty}L_{t}^{2}[A,T]}\lesssim\frac{1}{A}\left\Vert F^{S}\right\Vert _{L_{x}^{1}L_{t}^{2}}.
\]
Testing a function $H(x,t)\in L_{x}^{1}L_{t}^{2}$, one has 
\begin{align*}
\int_{\mathbb{R}^{3}}\int_{A}^{T}H(x,t)D_{A}(x,t)\,dtdx & =\int_{\mathbb{R}^{3}}\int_{A}^{T}H(x,t)\int_{0}^{t-A}\int_{\left|x-y\right|=t-s}\frac{F(y,s)}{\left|x-y\right|}\,\sigma\left(dy\right)\,dsdtdx\\
 & =\int_{\mathbb{R}^{3}}\int_{0}^{T-A}F(y,s)\int_{s+A}^{T}\int_{\left|x-y\right|=t-s}\frac{H(y,t)}{\left|x-z-vs\right|}\sigma\left(dx\right)dtdsdy\\
 & =\int_{\mathbb{R}^{3}}\int_{0}^{T-A}F^{S}(z,s)\int_{s+A}^{T}\int_{\left|x-y\right|=t-s}\frac{H(y,t)}{\left|x-z-vs\right|}\sigma\left(dx\right)dtdsdy.
\end{align*}
Then it suffices to show 
\begin{equation}
\left\Vert \int_{s+A}^{T}\int_{\left|x-y\right|=t-s}\frac{H(y,t)}{\left|x-z-vs\right|}\sigma\left(dx\right)dt\right\Vert _{L_{z}^{\infty}L_{s}^{2}[0,T-A]}\lesssim\frac{1}{A}\left\Vert H\right\Vert _{L_{x}^{1}L_{t}^{2}[0,T]}.\label{eq:dualityA}
\end{equation}
But with an almost identical argument as Corollary  \ref{cor:inhomA},
the estimate \eqref{eq:dualityA} indeed holds, and therefore, our desired
estimate holds too.

By \cite{BecGo} or applying the structure formula of wave operators,
with the calculations in the proof of Lemma \ref{lem:Frot}, Corollary
\ref{cor:perinhomA} and Theorems \ref{thm: persl}, \ref{thm:perinhomsl},
we have the perturbed version of the estimates \eqref{eq:Frot1} and
\eqref{eq:Frot2}. We omit the details here since the calculations are
more or less identical.
\begin{thm}
\label{thm:Prot}For $\left|\vec{\mu}\right|<1,\left|v\right|<1$
and $A>0$, suppose $H=-\Delta+V$ has neither resonances nor eigenfunctions
at $0$. Define 
\begin{equation}
k_{A}(x,t):=\int_{0}^{t-A}\frac{\sin\left((t-s)\sqrt{H}\right)}{\sqrt{H}}P_{c}F(s)\,ds
\end{equation}
\begin{equation}
k_{A}^{S}(x,t)=k_{A}\left(x+\vec{\mu}t,t\right).
\end{equation}
We have
\begin{equation}
\left\Vert k_{A}(x,t)\right\Vert _{L_{x}^{\infty}L_{t}^{2}[A,\infty)}\lesssim\frac{1}{A}\left(\left\Vert F^{S}\right\Vert _{L_{x}^{1}L_{t}^{2}}+\left\Vert F^{S}\right\Vert _{L_{x}^{\frac{3}{2},1}L_{t}^{2}}\right),\label{eq:Prot1}
\end{equation}
\begin{equation}
\left\Vert k_{A}^{S}(x,t)\right\Vert _{L_{x}^{\infty}L_{t}^{2}[A,\infty)}\lesssim\frac{1}{A}\left(\left\Vert F^{S}\right\Vert _{L_{x}^{1}L_{t}^{2}}+\left\Vert F^{S}\right\Vert _{L_{x}^{\frac{3}{2},1}L_{t}^{2}}\right),\label{eq:Prot2}
\end{equation}
and for $T>0$, 
\begin{equation}
\left\Vert k_{A}(x,t)\right\Vert _{L_{x}^{\infty}L_{t}^{2}[A,T]}\lesssim\frac{1}{A}\left(\left\Vert F^{S}\right\Vert _{L_{x}^{1}L_{t}^{2}[0,T]}+\left\Vert F^{S}\right\Vert _{L_{x}^{\frac{3}{2},1}L_{t}^{2}[0,T]}\right),\label{eq:Prot1T}
\end{equation}
\begin{equation}
\left\Vert k_{A}^{S}(x,t)\right\Vert _{L_{x}^{\infty}L_{t}^{2}[A,T]}\lesssim\frac{1}{A}\left(\left\Vert F^{S}\right\Vert _{L_{x}^{1}L_{t}^{2}[0,T]}+\left\Vert F^{S}\right\Vert _{L_{x}^{\frac{3}{2},1}L_{t}^{2}[0,T]}\right).\label{eq:Prot2T}
\end{equation}
\end{thm}
By careful analysis and more complicated computations, one can extend
all the results above to the linear Klein-Gordon equation, cf.~\cite{GC3}.

\section{Endpoint Reversed Strichartz Estimates \label{sec:EndRSChar}}

In this section, we show the endpoint reversed Strichartz estimates
for the wave equation with charge transfer Hamiltonian. More precisely,
we consider
\begin{equation}
\partial_{tt}u-\Delta u+V_{1}(x)u+V_{2}(x-\vec{v}t)u=0\label{eq:chargeeq}
\end{equation}
with initial data
\[
u(x,0)=g(x),\,u_{t}(x,0)=f(x).
\]
Throughout this section, for simplicity, we furthermore assume $V_{i}$
is compactly supported. With a little bit more careful analysis, one
can easily obtain the same results for general case, see Remark \ref{rem:smalltail}.

Recall that after we apply the associated Lorentz transformation $L$,
under the new coordinate, we denote 
\begin{equation}
u_{L}\left(x_{1}',x_{2}',x_{3}',t'\right)=u\left(\gamma\left(x_{1}'+vt'\right),x_{2}',x_{3}',\gamma\left(t'+vx_{1}'\right)\right),\label{eq:newcorch}
\end{equation}
and with the inverse transformation $L^{-1}$ 
\begin{equation}
u(x,t)=u_{L}\left(\gamma\left(x_{1}-vt\right),x_{2},x_{3},\gamma\left(t-vx_{1}\right)\right).\label{eq:newcor2ch}
\end{equation}
Under the above setting, we state the main result of this section.
\begin{thm}
\label{thm:EndRSChWOB}Let $\left|v\right|<1$. Suppose $u$ is a
scattering state in the sense of Definition \ref{AO} and solves 
\begin{equation}
\partial_{tt}u-\Delta u+V_{1}(x)u+V_{2}(x-\vec{v}t)u=0
\end{equation}
with initial data
\begin{equation}
u(x,0)=g(x),\,u_{t}(x,0)=f(x).
\end{equation}
Then 
\begin{equation}
\sup_{x\in\mathbb{R}^{3}}\int_{0}^{\infty}\left|u(x,t)\right|^{2}dt\lesssim\left(\|f\|_{L^{2}}+\|g\|_{\dot{H}^{1}}\right)^{2}.\label{eq:EndRSChWOB}
\end{equation}
Furthermore, if we denote 
\begin{equation}
u^{S}(x,t):=u(x+vt,t),
\end{equation}
then 
\begin{equation}
\sup_{x\in\mathbb{R}^{3}}\int_{0}^{\infty}\left|u^{S}(x,t)\right|^{2}dt\lesssim\left(\|f\|_{L^{2}}+\|g\|_{\dot{H}^{1}}\right)^{2}.\label{eq:EndRSChWOBSL}
\end{equation}
\end{thm}
To show Theorem \ref{thm:EndRSChWOB}, we will apply a bootstrap process
and decomposition into channels in the spirit of \cite{RSS}. If there
are no bound states, the bootstrap arguments simply work for the entire
evolution. But in the presence of bound states, a more careful analysis
is necessary. We will construct a truncated evolution and show that
the estimates we obtain are independent of the truncation. Finally,
we pass our estimates to the entire evolution.

\subsection{Bootstrap argument}

We set up the bootstrap argument and prove the initial assumptions
for the bootstrap argument hold for big $T$ with some positive constants.

By Duhamel's formula, 
\begin{equation}
u(x,t)=\frac{\sin\left(t\sqrt{-\Delta}\right)}{\sqrt{-\Delta}}f+\cos\left(t\sqrt{-\Delta}\right)g-\int_{0}^{t}\frac{\sin\left((t-s)\sqrt{-\Delta}\right)}{\sqrt{-\Delta}}\left(V_{1}+V_{2}(\cdot-vs)\right)u(s)\,ds.\label{eq:express}
\end{equation}
By Gr\"onwall's inequality, the endpoint reversed Strichartz estimates
and the estimate along slanted lines for the free evolution, we have
the following estimates as bootstrap assumptions.
\begin{lem}
\label{lem:bootstrap}For $T>0$ large, there exist constants $C_{1}(T)$
and $C_{2}(T)$ such that 
\begin{equation}
\sup_{x\in\mathbb{R}^{3}}\int_{0}^{T}\left|u(x,t)\right|^{2}dt\leq C_{1}(T)\left(\|f\|_{L^{2}}+\|g\|_{\dot{H}^{1}}\right)^{2}\label{eq:boot1}
\end{equation}
and if we denote 
\[
u^{S}(x,t)=u(x+vt,t),
\]
then 
\begin{equation}
\sup_{x\in\mathbb{R}^{3}}\int_{0}^{T}\left|u^{S}\left(x,t\right)\right|^{2}dt\leq C_{2}(T)\left(\|f\|_{L^{2}}+\|g\|_{\dot{H}^{1}}\right)^{2}.\label{eq:boot2}
\end{equation}
\end{lem}
\begin{proof}
To establish the bootstrap assumptions, we first notice that by the
expression \eqref{eq:express} and Gr\"onwall inequality, we have 
\begin{equation}
\int_{\mathbb{R}^{3}}\left|\nabla u(x,t)\right|^{2}+\left|\partial_{t}u(x,t)\right|^{2}\,dx\lesssim e^{C\left|t\right|}\left(\|f\|_{L^{2}}+\|g\|_{\dot{H}^{1}}\right)^{2}.\label{eq:enegrowth}
\end{equation}
Clearly, estimates \eqref{eq:boot1} and \eqref{eq:boot2} hold for $T=0$.
Next, we note that for arbitrary $T_{0}>0$, from Theorem\ \ref{thm:EndRStrichF},
\begin{eqnarray}
\sup_{x\in\mathbb{R}^{3}}\int_{0}^{T_{0}}\left|u(x,t)\right|^{2}dt\lesssim\left(\|f\|_{L^{2}}+\|g\|_{\dot{H}^{1}}\right)^{2}\qquad\qquad\qquad\qquad\qquad\qquad\qquad\nonumber \\
\qquad\qquad+C(T_{0})\left(\sup_{x\in\mathbb{R}^{3}}\int_{0}^{T_{0}}\left|u(x,t)\right|^{2}dt+\sup_{x\in\mathbb{R}^{3}}\int_{0}^{T_{0}}\left|u^{S}(x,t)\right|^{2}dt\right)\label{eq:boot1att}
\end{eqnarray}
where $C(T_{0})$ can be computed explicitly, see Theorem \ref{thm:EndRStrichF}
and duality argument as Lemma \ref{lem:Frot}:
\[
C(T_{0})=\sup_{x\in\mathbb{R}^{3}}\int_{\left|x-y\right|\leq T_{0}}\frac{1}{\left|x-y\right|}\left|V_{1}\right|dy+\int_{\left|\hat{y}_{1}-\hat{x}_{1}\right|\leq T_{0}}\int\left|V_{2}\right|dy_{1}d\hat{y}_{1}.
\]
We can perform a similar estimate for $\sup_{x\in\mathbb{R}^{3}}\int_{0}^{T_{0}}\left|u^{S}(x,t)\right|^{2}dt.$

Therefore, for $T_{0}$ small enough, 
\begin{equation}
\sup_{x\in\mathbb{R}^{3}}\int_{0}^{T_{0}}\left|u(x,t)\right|^{2}dt\leq C(T_{0})\left(\|f\|_{L^{2}}+\|g\|_{\dot{H}^{1}}\right)^{2}.
\end{equation}
\begin{equation}
\sup_{x\in\mathbb{R}^{3}}\int_{0}^{T_{0}}\left|u^{S}(x,t)\right|^{2}dt\leq C(T_{0})\left(\|f\|_{L^{2}}+\|g\|_{\dot{H}^{1}}\right)^{2}.
\end{equation}
Iterating the above construction with the energy growth estimate \eqref{eq:enegrowth},
we can obtain that for $T>0$ large, there exists constant $C_{1}(T)$,
$C_{2}(T)$ such that 

\begin{equation}
\sup_{x\in\mathbb{R}^{3}}\int_{0}^{T}\left|u(x,t)\right|^{2}dt\leq C_{1}(T)\left(\|f\|_{L^{2}}+\|g\|_{\dot{H}^{1}}\right)^{2},
\end{equation}

\begin{equation}
\sup_{x\in\mathbb{R}^{3}}\int_{0}^{T}\left|u^{S}\left(x,t\right)\right|^{2}dt\leq C_{2}(T)\left(\|f\|_{L^{2}}+\|g\|_{\dot{H}^{1}}\right)^{2},
\end{equation}
as claimed.
\end{proof}
Based on estimates \eqref{eq:boot1}, \eqref{eq:boot2}, we will run a
bootstrap argument to improve these two estimate and reduce to estimates
with constants independent of $T$. 

We also have a perturbed version of Lemma \ref{lem:bootstrap} with
the same constants in estimates \eqref{eq:boot1} and \eqref{eq:boot2}
up to multiplication of a constant only depending on the potentials.
Let 
\[
H_{i}=-\Delta+V_{i},\,i=1,2
\]
and $P_{c}\left(H_{i}\right)$ to be the projection onto the continuous
spectrum of $H_{i}$. 
\begin{lem}
\label{lem:Pbootstrap}For $T>0$ large, there exist constants $C_{1}(T)$
and $C_{2}(T)$ such that 
\begin{equation}
\sup_{x\in\mathbb{R}^{3}}\int_{0}^{T}\left|P_{c}(H_{1})u(x,t)\right|^{2}dt\leq_{V_{1}}C_{1}(T)\left(\|f\|_{L^{2}}+\|g\|_{\dot{H}^{1}}\right)^{2},\label{eq:P1boot1}
\end{equation}
\begin{equation}
\sup_{x\in\mathbb{R}^{3}}\int_{0}^{T}\left|P_{c}\left(H_{2}\right)u_{L}(x,t)\right|^{2}dt\leq_{V_{2}}C_{1}(T)\left(\|f\|_{L^{2}}+\|g\|_{\dot{H}^{1}}\right)^{2}\label{eq:P2boot1}
\end{equation}
\end{lem}
\begin{rem}
By symmetry, with $C_{1}(T)$ and $C_{2}(T)$, we also have with $T>0$,
\begin{equation}
\sup_{x\in\mathbb{R}^{3}}\int_{-T}^{0}\left|u(x,t)\right|^{2}dt\leq C_{1}(T)\left(\|f\|_{L^{2}}+\|g\|_{\dot{H}^{1}}\right)^{2}\label{eq:boot3}
\end{equation}
and
\begin{equation}
\sup_{x\in\mathbb{R}^{3}}\int_{-T}^{0}\left|u^{S}\left(x,t\right)\right|^{2}dt\leq C_{2}(T)\left(\|f\|_{L^{2}}+\|g\|_{\dot{H}^{1}}\right)^{2}.\label{eq:boot4}
\end{equation}
\end{rem}

\subsection{Bound states\label{subsec:Boundstates}}

Before we start the bootstrap analysis, it is necessary to understand
the evolution of bound states.

In the following, for simplicity, we assume $H_{i}=-\Delta+V_{i},\,i=1,2$
has only one negative eigenvalue. With $\lambda>0,\,\,\mu>0,$
\begin{equation}
H_{1}w=-\lambda^{2}w,\,\,H_{2}m=-\mu^{2}m.
\end{equation}
$w$ and $m$ decay exponentially by Agmon's estimate. The analysis
can be easily adapted to the most general situation. 

Set $U(t,s)$ as evolution from $s$ to $t$ associated to the initial
velocity and formally, we use $\dot{U}(t,s)$ to denote the evolution
associated the other initial data.

Suppose $u(x,t)$ is a scattering state. We decompose the evolution
as following, 
\begin{equation}
u(x,t)=U(t,0)f+\dot{U}(t,0)g=a(t)w(x)+b\left(\gamma(t-vx_{1})\right)m_{v}\left(x,t\right)+r(x,t)\label{eq:evolution}
\end{equation}
where 
\[
m_{v}(x,t)=m\left(\gamma\left(x_{1}-vt\right),x_{2},x_{3}\right).
\]
With our decomposition, we know 
\begin{equation}
P_{c}\left(H_{1}\right)r=r
\end{equation}
 and 
\begin{equation}
P_{c}\left(H_{2}\right)r_{L}=r_{L}
\end{equation}
where the Lorentz transformation $L$ makes $V_{2}$ stationary. 

Surely, since $u(x,t)$ is asymptotically orthogonal to the bound
states of $H_{1}$ and $H_{2}$, it forces $a(t)$ to go $0$ and
$b(t)$ go to $0$. Following the above construction, we do some preliminary
calculations.

Plugging the evolution \eqref{eq:evolution} into the equation \eqref{eq:chargeeq}
and taking inner product with $w$, we get 
\begin{eqnarray}
\ddot{a}(t)-\lambda^{2}a(t)+a(t)\left\langle V_{2}\left(x-\vec{v}t\right)w,w\right\rangle \qquad\qquad\qquad\qquad\qquad\qquad\nonumber \\
\qquad\qquad\qquad+\left\langle V_{2}\left(x-\vec{v}t\right)\left(b\left(\gamma(t-vx_{1})\right)m_{v}\left(x,t\right)+r(x,t)\right),w\right\rangle =0.
\end{eqnarray}
One can write 
\begin{equation}
\ddot{a}(t)-\lambda^{2}a(t)+a(t)c(t)+h(t)=0,\label{eq:aode}
\end{equation}
where 
\begin{equation}
c(t):=\left\langle V_{2}\left(x-\vec{v}t\right)w,w\right\rangle 
\end{equation}
and 
\begin{equation}
h(t):=\left\langle V_{2}\left(x-\vec{v}t\right)\left(b\left(\gamma(t-vx_{1})\right)m_{v}\left(x,t\right)+r(x,t)\right),w\right\rangle .
\end{equation}
Since $w$ is exponentially localized by Agmon's estimate, we know
\begin{equation}
\left|c(t)\right|\lesssim e^{-\alpha\left|t\right|}.
\end{equation}
The existence of the solution to the ODE \eqref{eq:aode} is clear.
We study the long-time behavior of the solution. Write the equation
as 

\begin{equation}
\ddot{a}(t)-\lambda^{2}a(t)=-\left[a(t)c(t)+h(t)\right],
\end{equation}
and denote 
\begin{equation}
N(t):=-\left[a(t)c(t)+h(t)\right].
\end{equation}
Then
\begin{equation}
a(t)=\frac{e^{\lambda t}}{2}\left[a(0)+\frac{1}{\lambda}\dot{a}(0)+\frac{1}{\lambda}\int_{0}^{t}e^{-\lambda s}N(s)\,ds\right]+R(t)
\end{equation}
where 
\begin{equation}
\left|R(t)\right|\lesssim e^{-\beta t},
\end{equation}
for some positive constant $\beta>0$. Therefore, the stability condition
forces 
\begin{equation}
a(0)+\frac{1}{\lambda}\dot{a}(0)+\frac{1}{\lambda}\int_{0}^{\infty}e^{-\lambda s}N(s)\,ds=0.\label{eq:stability}
\end{equation}
Then under the stability condition \eqref{eq:stability}, 
\begin{equation}
a(t)=e^{-\lambda t}\left[a(0)+\frac{1}{2\lambda}\int_{0}^{\infty}e^{-\lambda s}N(s)ds\right]+\frac{1}{2\lambda}\int_{0}^{\infty}e^{-\lambda\left|t-s\right|}N(s)\,ds.
\end{equation}
We notice that in order to estimate $a(t)$ and $b(t)$, we need a
non-local term 
\begin{equation}
\int_{0}^{\infty}e^{-\lambda s}N(s)\,ds,
\end{equation}
and in all estimates, a global estimate for 
\begin{equation}
\left\Vert b\left(\gamma(t-vx_{1})\right)m_{v}\left(x,t\right)+r(x,t)\right\Vert _{L_{x}^{\infty}L_{t}^{2}[0,\infty)}
\end{equation}
is involved. But for the general charge transfer model, a-priori,
we do not have any global estimates. Therefore, we will consider a
truncated version of the above construction restricted to interval
$t\in[0,T]$ for large positive $T$. Then one can run the bootstrap
procedure for our truncated evolution. 

For $t\in[0,T]$, we construct the following truncated version of
the evolution: 
\begin{equation}
u_{T}(x,t)=U(t,0)f+\dot{U}(t,0)g=a_{T}(t)w(x)+b_{T}\left(\gamma(t-vx_{1})\right)m_{v}\left(x,t\right)+r_{T}(x,t).
\end{equation}
For $a_{T}(t)$, we analyze the same ODE for $a(t)$ again but restricted
to $[0,T]$ and instead of the stability condition 
\begin{equation}
a(0)+\frac{1}{\lambda}\dot{a}(0)+\frac{1}{\lambda}\int_{0}^{\infty}e^{-\lambda s}N(s)\,ds=0
\end{equation}
we impose the condition that
\begin{equation}
a_{T}(0)+\frac{1}{\lambda}\dot{a}_{T}(0)+\frac{1}{\lambda}\int_{0}^{T}e^{-\lambda s}N(s)\,ds=0.
\end{equation}
The same construction can be applied to \textbf{$b_{T}$.}
\begin{lem}
From the construction above, we have the following estimates: for
$0\ll A\ll T$,
\begin{equation}
\left\Vert a_{T}\right\Vert _{L^{\infty}[0,T]}\lesssim\left(C(A,\lambda)+\frac{1}{\lambda A}C_{1}(T)\right)\left(\|f\|_{L^{2}}+\|g\|_{\dot{H}^{1}}\right),\label{eq:abootL0}
\end{equation}
\begin{equation}
\left\Vert a_{T}\right\Vert _{L^{1}[0,T]}\lesssim\left(C(A,\lambda)+\frac{1}{\lambda A}C_{1}(T)\right)\left(\|f\|_{L^{2}}+\|g\|_{\dot{H}^{1}}\right),\label{eq:abootL1}
\end{equation}
\begin{equation}
\left\Vert b_{T}\right\Vert _{L^{\infty}[0,T]}\lesssim\left(C(A,\mu)+\frac{1}{\mu A}C_{1}(T)\right)\left(\|f\|_{L^{2}}+\|g\|_{\dot{H}^{1}}\right),\label{eq:bbootL0}
\end{equation}
and 
\begin{equation}
\left\Vert b_{T}\right\Vert _{L^{1}[0,T]}\lesssim\left(C(A,\mu)+\frac{1}{\mu A}C_{1}(T)\right)\left(\|f\|_{L^{2}}+\|g\|_{\dot{H}^{1}}\right).\label{eq:bbootL1}
\end{equation}
\end{lem}
\begin{proof}
First of all, by the bootstrap assumption \eqref{eq:P1boot1}, 
\begin{equation}
\left\Vert b_{T}\left(\gamma(t-vx_{1})\right)m_{v}\left(x,t\right)+r_{T}(x,t)\right\Vert _{L_{x}^{\infty}L_{t}^{2}[0,T]}\leq C_{1}(T)\left(\|f\|_{L^{2}}+\|g\|_{\dot{H}^{1}}\right).
\end{equation}
For $a_{T}(t)$, we know that 
\begin{eqnarray}
\ddot{a}_{T}(t)-\lambda^{2}a_{T}(t)+a_{T}(t)\left\langle V_{2}\left(x-\vec{v}t\right)w,w\right\rangle \qquad\qquad\qquad\qquad\qquad\qquad\nonumber \\
\qquad\qquad\qquad+\left\langle V_{2}\left(x-\vec{v}t\right)\left(b_{T}\left(\gamma(t-vx_{1})\right)m_{v}\left(x,t\right)+r_{T}(x,t)\right),w\right\rangle =0.
\end{eqnarray}
We obtain
\begin{equation}
a_{T}(t)=\frac{e^{\lambda t}}{2}\left[a_{T}(0)+\frac{1}{\lambda}\dot{a}_{T}(0)+\frac{1}{\lambda}\int_{0}^{t}e^{-\lambda s}N(s)ds\right]+R(t)
\end{equation}
where 
\begin{equation}
\left|R(t)\right|\lesssim e^{-\beta t},
\end{equation}
With notations introduced above, we consider the truncated version
of the stability condition, 
\begin{equation}
a_{T}(0)+\frac{1}{\lambda}\dot{a}_{T}(0)+\frac{1}{\lambda}\int_{0}^{T}e^{-\lambda s}N(s)\,ds=0.
\end{equation}
So 
\begin{equation}
a_{T}(t)=e^{-\lambda t}\left[a_{T}(0)+\frac{1}{2\lambda}\int_{0}^{T}e^{-\lambda s}N(s)ds\right]+\frac{1}{2\lambda}\int_{0}^{T}e^{-\lambda\left|t-s\right|}N(s)\,ds.
\end{equation}
where 
\begin{equation}
N(t)=-\left[a_{T}(t)c(t)+h(t)\right]
\end{equation}
with 
\begin{equation}
\left|c(t)\right|\lesssim e^{-\alpha\left|t\right|}
\end{equation}
\begin{equation}
h(t):=\left\langle V_{2}\left(x-\vec{v}t\right)\left[b_{T}(t-vx_{1})m_{v}\left(x,t\right)+r_{T}(x,t)\right],w\right\rangle .
\end{equation}
For $0\ll A\ll T$ fixed, we can always bound the $L^{\infty}$ norm
of $a_{T}$ on the interval $[0,A]$ by Gr\"onwall's inequality.
Therefore, it suffices to estimate the $L^{\infty}$ norm of $a_{T}$
from $A$ to $T$. Note that $\left|c(t)\right|\lesssim e^{-\alpha\left|t\right|}$,
for $A$ large, one can always absorb the effects from $\int_{A}^{T}$$a_{T}(t)c(t)\,dt$
into the left-hand side. Hence it reduces to estimate the $L_{t}^{1}$
norm of $h(t)$ restricted to $[A,T]$. 

Consider the integral

\[
\int_{A}^{T}\left|h(t)\right|\,dt=\int_{A}^{T}\left|\left\langle V_{2}\left(x-\vec{v}t\right)\left[b_{T}\left(\gamma(t-vx_{1})\right)m_{v}\left(x,t\right)+r_{T}(x,t)\right],w\right\rangle \right|\,dt.
\]
Clearly,
\begin{eqnarray*}
\int_{A}^{T}\left|V_{2}\left(x-\vec{v}t\right)\left[b_{T}\left(\gamma(t-vx_{1})\right)m_{v}\left(x,t\right)+r_{T}(x,t)\right]\right|dt\lesssim\qquad\qquad\\
\left(\int_{A}^{T}\left|\left(b_{T}\left(\gamma(t-vx_{1})\right)m_{v}\left(x,t\right)+r_{T}(x,t)\right)\right|^{2}dt\right)^{\frac{1}{2}}\left(\int_{A}^{T}\left|V_{2}\left(x-\vec{v}t\right)\right|^{2}dt\right)^{\frac{1}{2}}.
\end{eqnarray*}
Note that 
\begin{equation}
\left|\left\langle \left(\int_{A}^{T}\left|V_{2}\left(\cdot-vt\right)\right|^{2}dt\right)^{\frac{1}{2}},\,w\right\rangle \right|\lesssim\frac{1}{A}.
\end{equation}
By the preliminary calculations above, we can estimate the $L^{\infty}$
norm of $a_{T}(t)$, 
\begin{eqnarray}
\left\Vert a_{T}\right\Vert _{L^{\infty}[0,T]} & \lesssim & C(A,\lambda)\left(\|f\|_{L^{2}}+\|g\|_{\dot{H}^{1}}\right)+\frac{1}{\lambda}\int_{A}^{T}\left|h(t)\right|dt\nonumber \\
 & \lesssim & C(A,\lambda)\left(\|f\|_{L^{2}}+\|g\|_{\dot{H}^{1}}\right)\nonumber \\
 &  & +\frac{1}{\lambda A}\left(\int_{A}^{T}\left|\left(b_{T}\left(\gamma(t-vx_{1})\right)m_{v}\left(x,t\right)+r_{T}(x,t)\right)\right|^{2}dt\right)^{\frac{1}{2}}\\
 & \lesssim & \left(C(A,\lambda)+\frac{1}{\lambda A}C_{1}(T)\right)\left(\|f\|_{L^{2}}+\|g\|_{\dot{H}^{1}}\right).\nonumber 
\end{eqnarray}
Similarly, for the $L^{1}$ norm of $a_{T}(t)$,  
\begin{equation}
\left\Vert a_{T}\right\Vert _{L^{1}[0,T]}\lesssim\left(C(A,\lambda)+\frac{1}{\lambda A}C_{1}(T)\right)\left(\|f\|_{L^{2}}+\|g\|_{\dot{H}^{1}}\right).
\end{equation}
After applying a Lorentz transformation, we have analogous estimates
for $b_{T}(t)$:
\begin{equation}
\left\Vert b_{T}\right\Vert _{L^{\infty}[0,T]}\lesssim\left(C(A,\mu)+\frac{1}{\mu A}C_{1}(T)\right)\left(\|f\|_{L^{2}}+\|g\|_{\dot{H}^{1}}\right),
\end{equation}
\begin{equation}
\left\Vert b_{T}\right\Vert _{L^{1}[0,T]}\lesssim\left(C(A,\mu)+\frac{1}{\mu A}C_{1}(T)\right)\left(\|f\|_{L^{2}}+\|g\|_{\dot{H}^{1}}\right).
\end{equation}
The lemma is proved.
\end{proof}
In the following subsections, we will show estimates with constants
independent of $T$, 
\begin{equation}
\sup_{x\in\mathbb{R}^{3}}\int_{0}^{T}\left|u_{T}(x,t)\right|^{2}dt\leq C_{1}\left(\|f\|_{L^{2}}+\|g\|_{\dot{H}^{1}}\right)^{2}
\end{equation}
\begin{equation}
\sup_{x\in\mathbb{R}^{3}}\int_{0}^{T}\left|u_{T}^{S}\left(x,t\right)\right|^{2}dt\leq C_{2}\left(\|f\|_{L^{2}}+\|g\|_{\dot{H}^{1}}\right)^{2}.
\end{equation}
Then we know our construction of $u_{T}$ has estimates independent
of $T$. As $T\rightarrow\infty$, the stability condition \eqref{eq:stability}
will be recovered from 
\begin{equation}
a_{T}(0)+\frac{1}{\lambda}\dot{a}_{T}(0)+\frac{1}{\lambda}\int_{0}^{T}e^{-\lambda s}N(s)\,ds=0.
\end{equation}
and 
\begin{equation}
a_{T}(t)=e^{-\lambda t}\left[a_{T}(0)+\frac{1}{2\lambda}\int_{0}^{T}e^{-\lambda s}N(s)ds\right]+\frac{1}{2\lambda}\int_{0}^{T}e^{-\lambda\left|t-s\right|}N(s)\,ds.
\end{equation}
Therefore, from the estimates for $u_{T}$, we can obtain the desired
estimates for a scattering state $u(x,t)$,
\begin{equation}
\sup_{x\in\mathbb{R}^{3}}\int_{0}^{\infty}\left|u(x,t)\right|^{2}dt\leq C_{1}\left(\|f\|_{L^{2}}+\|g\|_{\dot{H}^{1}}\right)^{2}
\end{equation}
\begin{equation}
\sup_{x\in\mathbb{R}^{3}}\int_{0}^{\infty}\left|u^{S}\left(x,t\right)\right|^{2}dt\leq C_{2}\left(\|f\|_{L^{2}}+\|g\|_{\dot{H}^{1}}\right)^{2}.
\end{equation}
Therefore in the remaining part of this section, we will analyze the
bootstrap process for $u_{T}(x,t)$ carefully.

\subsection{Decomposition into channels}

Following the notations above, for $t\in[0,T]$, consider 
\[
u_{T}(x,t)=U(t,0)f+\dot{U}(t,0)g=a_{T}(t)w(x)+b_{T}\left(\gamma(t-vx_{1})\right)m_{v}\left(x,t\right)+r_{T}(x,t).
\]
There exist constants $C_{1}(T)$ and $C_{2}(T)$ such that 

\begin{equation}
\sup_{x\in\mathbb{R}^{3}}\int_{0}^{T}\left|u_{T}(x,t)\right|^{2}dt\leq C_{1}(T)\left(\|f\|_{L^{2}}+\|g\|_{\dot{H}^{1}}\right)^{2}
\end{equation}
and
\begin{equation}
\sup_{x\in\mathbb{R}^{3}}\int_{0}^{T}\left|u_{T}^{S}\left(x,t\right)\right|^{2}dt\leq C_{2}(T)\left(\|f\|_{L^{2}}+\|g\|_{\dot{H}^{1}}\right)^{2}.
\end{equation}
We run our bootstrap argument for $u_{T}(x,t)$. Notice that since
$V_{i},\,i=1,2$ is a short-range potential and $V_{2}(x-\vec{v}t)$ moves
away from $V_{1}$, intuitively, $u(x,t)$ will have different dominant
behaviors in different regions in $\mathbb{R}^{3}$. To make this
heuristic rigorous, we perform a decomposition of channels based on
it. For some fixed small $\delta>0$, we introduce a partition of
unity associated with the sets 
\begin{equation}
B_{\delta t}(0)=\left\{ x:\,\left|x\right|\leq\delta t\right\} ,\qquad B_{\delta t}(tv)=\left\{ x:\,\left|x-\left(tv,0,0\right)\right|\leq\delta t\right\} 
\end{equation}
and
\begin{equation}
\mathbb{R}^{3}\backslash\left(B_{\delta t}(0)\cup B_{\delta t}(tv)\right).
\end{equation}
To be more precisely, let $\chi_{1}(x,t)$ be a smooth cutoff function
such that 
\begin{equation}
\chi_{1}(x,t)=1,\;\forall x\in B_{\delta t}(0),\qquad\chi_{1}(x,t)=0,\;\forall x\in\mathbb{R}^{3}\backslash B_{2\delta t}(0).
\end{equation}
One might assume $t\geq t_{0}$ for some large $t_{0}$. We also define
\begin{equation}
\chi_{2}(x,t)=\chi_{1}(x-\vec{v}t,t),\qquad\chi_{3}=1-\chi_{1}-\chi_{2}.
\end{equation}
Note that we only consider the estimates for large $t$, so one might
also assume the support of $\chi_{1}(x,t)$ contains the support of
$V_{1}\left(x\right)$ and support of $\chi_{2}(x,t)$ contains the
support of $V_{2}\left(\cdot-vt\right)$.

With the partition above, we rewrite the evolution as
\begin{equation}
u_{T}(x,t)=\chi_{1}(x,t)u_{T}(x,t)+\chi_{2}(x,t)u_{T}(x,t)+\chi_{3}(x,t)u_{T}(x,t).
\end{equation}
We will discuss $\chi_{i}(x,t)u_{T}(x,t),\,i=1,2,3$, separately.

Based on Duhamel's formula, we will compare $u$ to different evolution
groups on different ``channels''.

For 
\begin{equation}
\chi_{1}(x,t)u_{T}(x,t),
\end{equation}
we will compare it to 
\begin{equation}
W_{1}(t)f+\dot{W}_{1}(t)g
\end{equation}
 where 
\begin{equation}
W_{1}(t):=\frac{\sin\left(t\sqrt{H_{1}}\right)}{\sqrt{H_{1}}}.
\end{equation}

As to 
\begin{equation}
\chi_{2}(x,t)u_{T}(x,t),
\end{equation}
it will be compared to 
\begin{equation}
W_{2}(t)f+\dot{W}_{2}(t)g
\end{equation}
 where $W_{2}\left(t,s\right)$ denotes the evolution associated with
the Hamiltonian $-\Delta+V_{2}(x-\vec{v}t)$ and initial velocity $f$.
starting from $s$ to $t$. And formally, $\dot{W}_{2}(t,s)$ is used
to denote the evolution associated with $g$ from $s$ to $t$. Here
the dot in $\dot{W}_{2}$ is not the time derivative but simply a
notation. These evolution can be obtained from the entries of the
solution map if we write the wave equation $\partial_{tt}u-\Delta u+V_{2}(x-\vec{v}t)u=0$
using the Hamiltonian structure. We also use the short-hand notation
$W_{2}(t)$ and $\dot{W}_{2}(t)$ to denote the evolution starting
at $s=0$.

Finally 
\begin{equation}
\chi_{3}(x,t)u_{T}(x,t)
\end{equation}
 is compared with 
\begin{equation}
W_{0}(t)f+\dot{W}_{0}(t)g
\end{equation}
 where 
\begin{equation}
W_{0}(t):=\frac{\sin\left(t\sqrt{-\Delta}\right)}{\sqrt{-\Delta}}.
\end{equation}

To be more explicit, we write 
\begin{eqnarray}
\chi_{1}(x,t)u_{T}(x,t) & = & \chi_{1}(x,t)W_{1}(t)f+\chi_{1}(x,t)\dot{W}_{1}(t)g\nonumber \\
 &  & -\chi_{1}(x,t)\int_{0}^{t}W_{1}(t-s)V_{2}(\cdot-sv)u_{T}(s)\,ds,\label{eq:first}
\end{eqnarray}
\begin{eqnarray}
\chi_{2}(x,t)u_{T}(x,t) & = & \chi_{2}(x,t)W_{2}(t)f+\chi_{2}(x,t)\dot{W}_{2}(t)g\nonumber \\
 &  & -\chi_{2}(x,t)\int_{0}^{t}W_{2}(t,s)V_{1}u_{T}(s)\,ds\label{eq:second}
\end{eqnarray}
and 
\begin{eqnarray}
\chi_{3}(x,t)u_{T}(x,t) & = & \chi_{3}(x,t)W_{0}(t)f+\chi_{3}(x,t)\dot{W}_{0}(t)g\nonumber \\
 &  & -\chi_{3}(x,t)\int_{0}^{t}W_{0}(t-s)\left(V_{1}+V_{2}(\cdot-vs)\right)u_{T}(s)\,ds.\label{eq:third}
\end{eqnarray}

\subsection{Analysis of the three channels}

We will use the notations 
\begin{eqnarray}
u_{T}(x,t) & = & a_{T}(t)w(x)+b_{T}\left(\gamma(t-vx_{1})\right)m_{v}\left(x,t\right)+r_{T}(x,t)\nonumber \\
 & =: & a_{T}(t)w(x)+u_{T,1}\left(x,t\right)\label{eq:decomposition}\\
 & =: & b_{T}\left(\gamma(t-vx_{1})\right)m_{v}\left(x,t\right)+u_{T,2}(x,t).\nonumber 
\end{eqnarray}
Note that 
\begin{equation}
P_{c}\left(H_{1}\right)\left(u_{T,1}\right)=u_{T,1}
\end{equation}
and 
\begin{equation}
P_{c}\left(H_{2}\right)\left(u_{T,2}\right)_{L}=\left(u_{T,2}\right)_{L}.
\end{equation}
The free channel and the channel associated with $H_{1}$ are easy
to analyze with the endpoint reversed Strichartz estimate and results
for estimates along slanted lines, Theorems \ref{thm:EndRStrichF},
\ref{thm:PStriRStrich}, \ref{thm: persl} and Lemma \ref{lem:freesl}. 

\subsubsection{Analysis of $\chi_{1}(x,t)u_{T}(x,t)$:}

We consider 
\begin{eqnarray}
\chi_{1}(x,t)u_{T}(x,t) & = & \chi_{1}(x,t)W_{1}(t)f+\chi_{1}(x,t)\dot{W}_{1}(t)g\nonumber \\
 &  & -\chi_{1}(x,t)\int_{0}^{t}W_{1}(t-s)V_{2}(\cdot-sv)u_{T}(s)\,ds.
\end{eqnarray}
Given $B$ fixed and $0\ll B\ll T$, one can always bound the integrals
restricted to $[0,B]$,
\[
\int_{0}^{B}\left|\chi_{1}(x,t)u_{T}(x,t)\right|^{2}dt,\quad\int_{0}^{B}\left|\chi_{1}(x,t)u_{T}^{S}(x,t)\right|^{2}dt
\]
by a prescribed constant by Gr\"onwall's inequality as Lemma \ref{lem:bootstrap}.
Therefore, it suffices to consider the integrals over $[B,T]$. If
we fixed $0\ll A\ll T$ large, one can always find a big constant
$B$ such that $A\ll\frac{\left(v-2\delta\right)}{1+v}B$. Then when
we consider the integrals from $B$ to $T$, by the finite speed of
propagation and the fact that $V_{2}$ is compactly supported, we
can further reduce 
\begin{eqnarray}
\chi_{1}(x,t)u_{T}(x,t) & = & \chi_{1}(x,t)W_{1}(t)f+\chi_{1}(x,t)\dot{W}_{1}(t)g\nonumber \\
 &  & -\chi_{1}(x,t)\int_{0}^{t-A}W_{1}(t-s)V_{2}(\cdot-sv)u_{T}(s)\,ds.
\end{eqnarray}
For $s>t-A$, the center of $V_{2}$ is of distance at least $\left|\left(t-A\right)v\right|$
away from the center of the support of $\chi_{1}$. Meanwhile, $t-s$
is at most $A$. So the effects caused by $W_{1}(t-s)V_{2}(\cdot-sv)u_{T}(s)$
will not influence the points in the support of $\chi_{1}$.

First, we consider the endpoint reversed Strichartz estimate \eqref{eq:boot1},
\begin{eqnarray*}
\int_{0}^{T}\left|\chi_{1}(x,t)u_{T,1}(x,t)\right|^{2}dt & \lesssim & \int_{0}^{T}\left|\chi_{1}(x,t)W_{1}(t)P_{c}\left(H_{1}\right)f+\chi_{1}(x,t)\dot{W}_{1}(t)P_{c}\left(H_{1}\right)g\right|^{2}dt\\
 &  & +\int_{0}^{T}\left|\chi_{1}(x,t)\int_{0}^{t-A}W_{1}(t-s)P_{c}\left(H_{1}\right)V_{2}(\cdot-sv)u_{T}(s)\,ds\right|^{2}dt\\
 & \lesssim & \left(\|f\|_{L^{2}}+\|g\|_{\dot{H}^{1}}\right)^{2}\\
 &  & +\int_{0}^{B}\left|\chi_{1}(x,t)\int_{0}^{t}W_{1}(t-s)P_{c}\left(H_{1}\right)V_{2}(\cdot-sv)u_{T}(s)\,ds\right|^{2}dt\\
 &  & +\int_{B}^{T}\left|\chi_{1}(x,t)\int_{0}^{t-A}W_{1}(t-s)P_{c}\left(H_{1}\right)V_{2}(\cdot-sv)u_{T}(s)\,ds\right|^{2}\\
 & \lesssim & \left(\|f\|_{L^{2}}+\|g\|_{\dot{H}^{1}}\right)^{2}+C(B)\left(\|f\|_{L^{2}}+\|g\|_{\dot{H}^{1}}\right)^{2}\\
 &  & +\int_{B}^{T}\left|\chi_{1}(x,t)\int_{0}^{t-A}W_{1}(t-s)P_{c}\left(H_{1}\right)V_{2}(\cdot-sv)u_{T}(s)\,ds\right|^{2}dt\\
 & \lesssim & \left(\|f\|_{L^{2}}+\|g\|_{\dot{H}^{1}}\right)^{2}+C(B)\left(\|f\|_{L^{2}}+\|g\|_{\dot{H}^{1}}\right)^{2}\\
 &  & +\frac{1}{A}C_{2}(T)\left(\|f\|_{L^{2}}+\|g\|_{\dot{H}^{1}}\right)^{2}.
\end{eqnarray*}
In the above calculations, for the second inequality, we applied the
endpoint Strichartz estimate for perturbed wave equations, cf.~Theorem
\ref{thm:PStriRStrich}:
\[
\int_{0}^{T}\left|\chi_{1}(x,t)W_{1}(t)P_{c}\left(H_{1}\right)f+\chi_{1}(x,t)\dot{W}_{1}(t)P_{c}\left(H_{1}\right)g\right|^{2}dt\lesssim\left(\|f\|_{L^{2}}+\|g\|_{\dot{H}^{1}}\right)^{2}.
\]
For the third inequality, we used the fact that $B$ is a fixed big
constant, one can always find $C(B)$ independent of $T$ to ensure
the inequality holds as we did in Lemma \ref{lem:Pbootstrap}:
\[
\int_{0}^{B}\left|\chi_{1}(x,t)\int_{0}^{t}W_{1}(t-s)P_{c}\left(H_{1}\right)V_{2}(\cdot-sv)u_{T}(s)\,ds\right|^{2}dt\lesssim C(B)\left(\|f\|_{L^{2}}+\|g\|_{\dot{H}^{1}}\right)^{2}.
\]
For the last inequality, we used the bootstrap assumption \eqref{eq:boot2}
and the results from the section on estimates along slanted lines,
Theorem \ref{thm:Prot} and Corollary \ref{cor:perinhomA}. By Theorem
\ref{thm:Prot}, 
\begin{align*}
\int_{B}^{T}\left|\chi_{1}(x,t)\int_{0}^{t-A}W_{1}(t-s)P_{c}\left(H_{1}\right)V_{2}(\cdot-sv)u_{T}(s)\,ds\right|^{2}dt & \lesssim\frac{1}{A^{2}}\left\Vert V_{2}\right\Vert _{L_{x}^{1}}^{2}\sup_{x}\int_{0}^{T}\left|u_{T}(t)\right|^{2}\,dt\\
 & \lesssim\frac{1}{A}C_{2}(T)\left(\|f\|_{L^{2}}+\|g\|_{\dot{H}^{1}}\right).^{2}
\end{align*}
Therefore, 
\begin{equation}
\int_{0}^{T}\left|\chi_{1}(x,t)u_{T,1}(x,t)\right|^{2}dt\lesssim\left(C_{0}+C(A)+\frac{1}{A}C_{2}(T)\right)\left(\|f\|_{L^{2}}+\|g\|_{\dot{H}^{1}}\right)^{2}.\label{eq:firstEnd}
\end{equation}
For the remaining piece, by estimates \eqref{eq:abootL0}, \eqref{eq:abootL1}
and Agmon's estimate,
\begin{equation}
\int_{0}^{T}\left|\chi_{1}(x,t)a_{T}(t)w(x)\right|^{2}dt\lesssim\left(C(A,\lambda)+\frac{1}{\lambda A}C_{1}(T)\right)\left(\|f\|_{L^{2}}+\|g\|_{\dot{H}^{1}}\right)^{2}.\label{eq:firstEndB}
\end{equation}
Therefore, with estimates \eqref{eq:firstEnd} and \eqref{eq:firstEndB},
for the endpoint reversed estimate, we obtain
\begin{equation}
C_{1}(T)\lesssim C_{0}+C(A,B)+\frac{1}{A}C_{2}(T)\label{eq:boot1first}
\end{equation}
in the first channel. So for $A$ large, in this channel, we have
the condition for the bootstrap argument.

Next we consider the estimate along the slanted line $(x+vt,t)$. 

Denoting 
\begin{equation}
u_{T,1}^{S}(x,t)=\chi_{1}(x+vt,t)u_{T,1}(x+vt,t),
\end{equation}
we want to estimate 
\begin{equation}
\int_{0}^{T}\left|\chi_{1}(x+vt,t)u_{T,1}(x+vt,t)\right|^{2}dt=\int_{0}^{T}\left|u_{T,1}^{S}(x,t)\right|^{2}dt.
\end{equation}
Furthermore, we introduce
\begin{equation}
D_{1}^{S}(x,t):=D_{1}\left(x+vt,t\right)
\end{equation}
where 
\begin{equation}
D_{1}(x,t):=\chi_{1}(x,t)W_{1}(t)P_{c}\left(H_{1}\right)f+\chi_{1}(x,t)\dot{W_{1}}(t)P_{c}\left(H_{1}\right)g;
\end{equation}
 
\begin{equation}
k_{1}^{S}(x,t):=k_{1}\left(x+vt,t\right)
\end{equation}
where 
\begin{equation}
k_{1}(x,t):=\chi_{1}(x,t)\int_{0}^{t}W_{1}(t-s)P_{c}\left(H_{1}\right)V_{2}(\cdot-sv)u_{T}(s)\,ds;
\end{equation}

\begin{equation}
E_{1}^{S}(x,t):=E_{1}\left(x+vt,t\right)
\end{equation}
where 
\begin{equation}
E_{1}(x,t):=\chi_{1}(x,t)\int_{0}^{t-A}W_{1}(t-s)P_{c}\left(H_{1}\right)V_{2}(\cdot-sv)u_{T}(s)\,ds.
\end{equation}
Then we can conclude 
\begin{eqnarray}
\int_{0}^{T}\left|u_{T,1}^{S}\right|^{2}dt & \lesssim & \int_{0}^{T}\left|D_{1}^{S}\right|^{2}dt+\int_{0}^{B}\left|k_{1}^{S}\right|^{2}dt+\int_{B}^{T}\left|E_{1}^{S}\right|^{2}dt\nonumber \\
 & \lesssim & \left(\|f\|_{L^{2}}+\|g\|_{\dot{H}^{1}}\right)^{2}+C(B)\left(\|f\|_{L^{2}}+\|g\|_{\dot{H}^{1}}\right)^{2}\nonumber \\
 &  & +\frac{1}{A}C_{2}(T)\left(\|f\|_{L^{2}}+\|g\|_{\dot{H}^{1}}\right)^{2}\label{eq:firstSE}
\end{eqnarray}
 similar to the analysis of estimate \eqref{eq:firstEnd} via Theorems
\ref{thm: persl}, \ref{thm:Prot} and Corollary \ref{cor:perinhomA}. 

For the piece with bound states, by estimate \eqref{eq:abootL1} and
Agmon's estimate, 
\begin{eqnarray}
\int_{0}^{T}\left|\chi_{1}(x+vt,t)a_{T}(t)w(x+vt)\right|^{2}dt\label{eq:firstSEB}\\
\lesssim\left(C(A,\lambda)+\frac{1}{\lambda A}C_{1}(T)\right)\left(\|f\|_{L^{2}}+\|g\|_{\dot{H}^{1}}\right)^{2}.\nonumber 
\end{eqnarray}
Therefore, with estimates \eqref{eq:firstSE} and \eqref{eq:firstSEB},
we obtain
\begin{equation}
C_{2}(T)\lesssim C_{0}+C(A,B)+\frac{1}{A}C_{2}(T)\label{eq:boot2first}
\end{equation}
in the first channel. So for $A$ large, in this channel, we obtain
the desired reduction for the bootstrap argument.

\subsubsection{Analysis of $\chi_{2}(x,t)u_{T}(x,t)$:}

Now we consider the most delicate channel which is the channel associated
to the moving potential.
\begin{eqnarray}
\chi_{2}(x,t)u_{T}(x,t) & = & \chi_{2}(x,t)W_{2}(t)f+\chi_{2}(x,t)\dot{W}_{2}(t)g\nonumber \\
 &  & -\chi_{2}(x,t)\int_{0}^{t}W_{2}(t,s)V_{1}u_{T}(s)\,ds.
\end{eqnarray}
Again, by the finite speed of propagation, it suffices to consider
\begin{eqnarray}
\chi_{2}(x,t)u_{T}(x,t) & = & \chi_{2}(x,t)W_{2}(t)f+\chi_{2}(x,t)\dot{W}_{2}(t)g\nonumber \\
 &  & -\chi_{2}(x,t)\int_{0}^{t-A}W_{2}(t,s)V_{1}u_{T}(s)\,ds.
\end{eqnarray}
Note that with the Lorentz transformation associated with $V_{2}(x-\vec{v}t)$,
we have 
\begin{equation}
\left(u_{T}\right)_{L}\left(x_{1}',x_{2}',x_{3}',t'\right)=u_{T}\left(\gamma\left(x_{1}'+vt'\right),x_{2}',x_{3}',\gamma\left(t'+vx_{1}'\right)\right)
\end{equation}
and 
\begin{equation}
u_{T}(x,t)=\left(u_{T}\right)_{L}\left(\gamma\left(x_{1}-vt\right),x_{2},x_{3},\gamma\left(t-vx_{1}\right)\right).
\end{equation}
The endpoint reversed Strichartz estimate \eqref{eq:boot1} for this
channel is equivalent to the estimate along the slanted line $(x-\vec{v}t,t)$
under the new frame. Meanwhile, the estimate along the slanted line
$(x+vt,t)$, see \eqref{eq:boot2}, for this channel is equivalent to
the endpoint reversed Strichartz estimate with respect to the new
frame. 

Denote $\widetilde{g}$ and $\widetilde{f}$ to denote the initial
data with respect to this new frame under which $V_{2}$ is stationary
and $V_{1}$ is moving. We use $W_{2}^{L}(t)$ and $\dot{W}_{2}^{L}(t)$
to denote the evolutions associated to $\tilde{f}$ and $\widetilde{g}$
respectively in the new frame. By construction, in the new frame,
$W_{2}^{L}(t)$ is the sine evolution with respect to $H_{2}$. By
Theorem \ref{thm:generalC}, we know 
\begin{equation}
\left(\|\widetilde{f}\|_{L^{2}}+\|\widetilde{g}\|_{\dot{H}^{1}}\right)\simeq\left(\|f\|_{L^{2}}+\|g\|_{\dot{H}^{1}}\right).
\end{equation}
Denote 
\begin{equation}
D_{2}^{S}(x,t):=D_{2}\left(x-\vec{v}t,t\right)
\end{equation}
where 
\begin{equation}
D_{2}(x,t):=W_{2}^{L}(t)P_{c}\left(H_{2}\right)\widetilde{f}+\dot{W}_{2}^{L}(t)P_{c}\left(H_{2}\right)\widetilde{g};
\end{equation}
\begin{equation}
k_{2}^{S}(x,t):=k_{2}\left(x-\vec{v}t,t\right)
\end{equation}
where 
\begin{equation}
k_{2}(x,t):=\int_{0}^{t}W_{2}^{L}(t-s)P_{c}\left(H_{2}\right)V_{1}(s)u_{T}(s)\,ds;
\end{equation}

\begin{equation}
E_{2}^{S}(x,t):=E_{2}\left(x-\vec{v}t,t\right)
\end{equation}
where 
\begin{equation}
E_{2}(x,t)=\int_{0}^{t-A}W_{2}^{L}(t-s)P_{c}\left(H_{2}\right)V_{1}(s)u_{T}(s)\,ds.
\end{equation}
With the estimates along the slanted line $\left(x-\vec{v}t,t\right)$ for
$W_{2}^{L}(t)$, Theorem \ref{thm: persl}, we know
\begin{eqnarray}
\int_{0}^{T}\left|\chi_{2}(x,t)u_{T,2}(x,t)\right|^{2}dt & = & \int_{0}^{T}\left|\left(u_{T,2}\right)_{L}\left(\gamma\left(x_{1}-vt\right),x_{2},x_{3},\gamma\left(t-vx_{1}\right)\right)\right|^{2}dt\nonumber \\
 & \lesssim & \int_{0}^{T}\left|D_{2}^{S}\right|^{2}dt+\int_{0}^{B}\left|k_{2}^{S}\right|^{2}dt+\int_{B}^{T}\left|E_{2}^{S}\right|^{2}dt\nonumber \\
 & \lesssim & \left(\|f\|_{L^{2}}+\|g\|_{\dot{H}^{1}}\right)^{2}+C(B)\left(\|f\|_{L^{2}}+\|g\|_{\dot{H}^{1}}\right)^{2}\nonumber \\
 &  & +\frac{1}{A}\left(\left\Vert V_{1}u_{T}\right\Vert _{L_{x}^{1}L_{t}^{2}[0,T]}\right)^{2}\label{eq:SecondEnd}\\
 & \lesssim & \left(\|f\|_{L^{2}}+\|g\|_{\dot{H}^{1}}\right)^{2}+C(B)\left(\|f\|_{L^{2}}+\|g\|_{\dot{H}^{1}}\right)^{2}\nonumber \\
 &  & +\frac{1}{A}\left(\left\Vert u_{T}\right\Vert _{L_{x}^{\infty}L_{t}^{2}[0,T]}\right)^{2}\nonumber \\
 & \lesssim & \left(\|f\|_{L^{2}}+\|g\|_{\dot{H}^{1}}\right)^{2}+C(B)\left(\|f\|_{L^{2}}+\|g\|_{\dot{H}^{1}}\right)^{2}\nonumber \\
 &  & +\frac{1}{A}C_{1}(T)\left(\|f\|_{L^{2}}+\|g\|_{\dot{H}^{1}}\right)^{2}\nonumber 
\end{eqnarray}
by the bootstrap assumption \eqref{eq:boot1}, Theorem \ref{thm:Prot}
and Corollary \ref{cor:perinhomA}. For the third inequality, we also
use the fact $A$ is a fixed big constant, we can always find $C(A)$
independent of $T$ to ensure the inequality holds.

For the piece with bound states, by estimate \eqref{eq:bbootL0} and
Agmon's estimate, one has
\begin{eqnarray}
\int_{0}^{T}\left|\chi_{2}(x,t)b_{T}\left(\gamma(t-vx_{1})\right)m_{v}\left(x,t\right)\right|^{2}dt\label{eq:secondEB}\\
\lesssim\left(C(A,\mu)+\frac{1}{\mu A}C_{1}(T)\right)\left(\|f\|_{L^{2}}+\|g\|_{\dot{H}^{1}}\right)^{2}.\nonumber 
\end{eqnarray}
Hence in this channel, with estimates \eqref{eq:SecondEnd} and \eqref{eq:secondEB},
\begin{equation}
C_{1}(T)\lesssim C_{0}+C(A,B)+\frac{1}{A}C_{1}(T).\label{eq:boot1second}
\end{equation}
So for $A$ large, in this channel, we achieve the condition for the
bootstrap argument.

Now we analyze the estimate along $(x+vt,t)$. The argument here is
similar to the analysis for the first channel. 

Denote 
\begin{equation}
u_{T,2}^{S}(x,t):=\chi_{2}(x+vt,t)u_{T,2}(x+vt,t).
\end{equation}
Then
\begin{eqnarray}
\int_{0}^{T}\left|u_{T,2}^{S}(x,t)\right|^{2}dt & \lesssim & \int_{0}^{T}\left|\left(u_{T,2}\right)_{L}(x,t)\right|^{2}dt\nonumber \\
 & \lesssim & \int_{0}^{T}\left|D_{2}(x,t)\right|^{2}dt+\int_{0}^{T}\left|k_{2}(x,t)\right|^{2}dt+\int_{0}^{T}\left|E_{2}(x,t)\right|^{2}dt\nonumber \\
 & \lesssim & \left(\|f\|_{L^{2}}+\|g\|_{\dot{H}^{1}}\right)^{2}+C(B)\left(\|f\|_{L^{2}}+\|g\|_{\dot{H}^{1}}\right)^{2}\nonumber \\
 &  & +\frac{1}{A}\left(\left\Vert V_{1}u_{T}\right\Vert _{L_{x}^{1}L_{t}^{2}\left(0,T\right)}\right)^{2}\\
 & \lesssim & \left(\|f\|_{L^{2}}+\|g\|_{\dot{H}^{1}}\right)^{2}+C(B)\left(\|f\|_{L^{2}}+\|g\|_{\dot{H}^{1}}\right)^{2}\nonumber \\
 &  & +\frac{1}{A}C_{1}(T)\left(\|f\|_{L^{2}}+\|g\|_{\dot{H}^{1}}\right)^{2}\nonumber 
\end{eqnarray}
with the bootstrap assumption \eqref{eq:boot1} and Corollary \ref{cor:perinhomA}. 

For the remaining piece with bound states, by a similar argument to
estimate \eqref{eq:firstEndB}, we have
\begin{eqnarray}
\int_{0}^{T}\left|\chi_{2}(x+vt,t)\left(b_{T}\left(\gamma\left(\left(1-v^{2}\right)t-vx_{1}\right)\right)m_{v}\left(x+vt,t\right)\right)\right|^{2}dt\nonumber \\
\lesssim\left(C(A,\mu)+\frac{1}{\mu A}C_{1}(T)\right)\left(\|f\|_{L^{2}}+\|g\|_{\dot{H}^{1}}\right)^{2}.\label{eq:secondSEB}
\end{eqnarray}
 Therefore, in this channel, we obtain
\begin{equation}
C_{2}(T)\lesssim C_{0}+C(A,B)+\frac{1}{A}C_{1}(T).\label{eq:boot2second}
\end{equation}
For $A$ large, in this channel, we recapture the condition for the
bootstrap argument.

\subsubsection{Analysis of $\chi_{3}(x,t)u_{T}(x,t)$:}

Finally, we consider the free channel $\chi_{3}(x,t)u_{T}(x,t)$.
In this channel, we can estimate all pieces together since the dominant
evolution is the free ones.

We know 
\begin{eqnarray}
\chi_{3}(x,t)u_{T}(x,t) & = & \chi_{3}(x,t)W_{0}(t)f+\chi_{3}(x,t)\dot{W}_{0}(t)g\nonumber \\
 &  & -\chi_{3}(x,t)\int_{0}^{t}W_{0}(t-s)\left(V_{1}+V_{2}(\cdot-vs)\right)u_{T}(s)\,ds.
\end{eqnarray}
By the finite speed of propagation as above, it suffices to consider
\begin{eqnarray}
\chi_{3}(x,t)u_{T}(x,t) & = & \chi_{3}(x,t)W_{0}(t)f+\chi_{3}(x,t)\dot{W}_{0}(t)g\nonumber \\
 &  & -\chi_{3}(x,t)\int_{0}^{t-A}W_{0}(t-s)\left(V_{1}+V_{2}(\cdot-vs)\right)u_{T}(s)\,ds.
\end{eqnarray}
Consider the endpoint reversed Strichartz estimate,
\begin{eqnarray}
\int_{0}^{T}\left|\chi_{3}(x,t)u_{T}(x,t)\right|^{2}dt & \lesssim & \int_{0}^{T}\left|\chi_{3}(x,t)W_{0}(t)f+\chi_{3}(x,t)\dot{W_{0}}(t)g\right|^{2}dt\nonumber \\
 &  & +\int_{0}^{T}\left|\chi_{3}(x,t)\int_{0}^{t-A}W_{0}(t-s)V_{1}u_{T}(s)\,ds\right|^{2}dt\\
 &  & +\int_{0}^{T}\left|\chi_{3}(x,t)\int_{0}^{t-A}W_{0}(t-s)V_{2}(\cdot-sv)u_{T}(s)\,ds\right|^{2}dt\nonumber \\
 & \lesssim & \left(\|f\|_{L^{2}}+\|g\|_{\dot{H}^{1}}\right)^{2}+C(B)\left(\|f\|_{L^{2}}+\|g\|_{\dot{H}^{1}}\right)^{2}\nonumber \\
 &  & +\frac{1}{A}\left(C_{1}(T)+C_{2}(T)\right)\left(\|f\|_{L^{2}}+\|g\|_{\dot{H}^{1}}\right)^{2}\nonumber 
\end{eqnarray}
For the last inequality, we apply the bootstrap assumptions \eqref{eq:boot1}
and \eqref{eq:boot2}. 
\begin{equation}
C_{1}(T)\lesssim C_{0}+C(A,B)+\frac{1}{A}\left(C_{1}(T)+C_{2}(T)\right).\label{eq:boot1third}
\end{equation}
So for $A$ large, in this channel, we obtain the condition for bootstrap
argument.

Next we consider the estimate along slanted line $(x+vt,t)$. 

Denote 
\begin{equation}
u_{T}^{S}(x,t)=\chi_{3}(x+vt,t)u_{T}(x+vt,t),
\end{equation}
\begin{equation}
u_{T,3}^{S}(x,t):=u_{T,3}\left(x+vt,t\right)
\end{equation}
where 
\begin{equation}
u_{T,3}(x,t):=\chi_{3}(x,t)W_{0}(t)f+\chi_{3}(x,t)\dot{W}_{0}(t)g;
\end{equation}
\begin{equation}
k_{3}^{S}(x,t):=k_{3}\left(x+vt,t\right)
\end{equation}
where 
\begin{equation}
k_{3}(x,t):=\chi_{3}(x,t)\int_{0}^{t}W_{0}(t-s)\left(V_{1}+V_{2}(\cdot-sv)\right)u_{T}(s)\,ds;
\end{equation}

\begin{equation}
E_{3}^{S}(x,t):=E_{3}\left(x+vt,t\right)
\end{equation}
where 
\begin{equation}
E_{3}(x,t):=\chi_{3}(x,t)\int_{0}^{t-A}W_{0}(t-s)\left(V_{1}+V_{2}(\cdot-sv)\right)u_{T}(s)\,ds.
\end{equation}
Then
\begin{eqnarray}
\int_{0}^{T}\left|u_{T}^{S}(x,t)\right|^{2}dt & \lesssim & \int_{0}^{T}\left|u_{T,3}^{S}(x,t)\right|^{2}dt+\int_{0}^{B}\left|k_{3}^{S}(x,t)\right|^{2}dt+\int_{B}^{T}\left|E_{3}^{S}(x,t)\right|^{2}dt\nonumber \\
 & \lesssim & \left(\|f\|_{L^{2}}+\|g\|_{\dot{H}^{1}}\right)^{2}+C(B)\left(\|f\|_{L^{2}}+\|g\|_{\dot{H}^{1}}\right)^{2}\nonumber \\
 &  & +\frac{1}{A}\left(C_{1}(T)+C_{2}(T)\right)\left(\|f\|_{L^{2}}+\|g\|_{\dot{H}^{1}}\right)^{2}
\end{eqnarray}
Therefore, we obtain
\begin{equation}
C_{2}(T)\lesssim C_{0}+C(A,B)+\frac{1}{A}\left(C_{1}(T)+C_{2}(T)\right)\label{eq:boot2third}
\end{equation}
along the free channel. So for $A$ large, we recapture the condition
for the bootstrap argument.

\subsection{Conclusion}

Finally, by the results from the analysis of three channels above,
we conclude 
\begin{equation}
C_{1}(T)\lesssim C_{0}+\frac{1}{A}C_{1}(T)+\frac{1}{A}C_{2}(T)+C(A,B)
\end{equation}
\begin{equation}
C_{2}(T)\lesssim C_{0}+\frac{1}{A}C_{1}(T)+\frac{1}{A}C_{2}(T)+C(A,B)
\end{equation}
where $C_{1}(T)$ is the constant appearing for the bootstrap assumption
\eqref{eq:boot1} for the endpoint reversed Strichartz estimate and
$C_{2}(T)$ is the constant for the bootstrap assumption \eqref{eq:boot2}
for the estimate along $(x+vt,t)$. 

We apply the bootstrap argument for these two estimates simultaneously.
We conclude that $C_{1}(T)$ and $C_{2}(T)$ are independent of $T$.
In other words, one has

\begin{equation}
\sup_{x\in\mathbb{R}^{3}}\int_{0}^{T}\left|u_{T}(x,t)\right|^{2}dt\leq C_{1}\left(\|f\|_{L^{2}}+\|g\|_{\dot{H}^{1}}\right)^{2}
\end{equation}
\begin{equation}
\sup_{x\in\mathbb{R}^{3}}\int_{0}^{T}\left|u_{T}^{S}\left(x,t\right)\right|^{2}dt\leq C_{2}\left(\|f\|_{L^{2}}+\|g\|_{\dot{H}^{1}}\right)^{2}.
\end{equation}
Finally, as we discussed above, passing $T$ to $\infty$, we will
recover those two estimates for a scattering state $u(x,t)$:

\begin{equation}
\sup_{x\in\mathbb{R}^{3}}\int_{0}^{\infty}\left|u(x,t)\right|^{2}dt\lesssim\left(\|f\|_{L^{2}}+\|g\|_{\dot{H}^{1}}\right)^{2}
\end{equation}
and 
\begin{equation}
\sup_{x\in\mathbb{R}^{3}}\int_{0}^{\infty}\left|u^{S}(x,t)\right|^{2}dt\lesssim\left(\|f\|_{L^{2}}+\|g\|_{\dot{H}^{1}}\right)^{2}.
\end{equation}

\begin{rem}
\label{rem:smalltail}In the above analysis, we assumed $V_{1}$ and
$V_{2}$ are compactly supported. With more careful calculations,
it is easy to extend the above results to $V_{1}$ and $V_{2}$ decay
as we assume in the Definition \ref{def: Charge}. In this case, instead
of vanishing, the smallness conditions of our bootstrap procedure
are from the smallness of tails of $V_{1}$ and $V_{2}$ in $L_{x}^{1}$
and $L_{x}^{\frac{3}{2},1}$ in the estimates for the following terms 

\begin{equation}
\chi_{1}\int_{t-A}^{t}W_{1}(t-s)V_{2}(\cdot-sv)u_{T}(s)\,ds,
\end{equation}
\begin{equation}
\chi_{2}\int_{t-A}^{t}W_{2}^{L}(t-s)V_{1}(s)u_{T}(s)\,ds
\end{equation}
and
\begin{equation}
\chi_{3}\int_{t-A}^{t}W_{0}(t-s)\left(V_{1}+V_{2}(\cdot-vs)\right)u_{T}(s)\,ds.
\end{equation}
To demonstrate, we compute a concrete example below. 
\begin{align}
\left\Vert \int_{t-A}^{t}W_{0}(t-s)V_{2}(\cdot-sv)u_{T}(s)\,ds\right\Vert _{L_{t}^{2}[A,T]}|_{x=0}\qquad\qquad\nonumber \\
=\left\Vert \int_{\left|y\right|\leq A}\frac{1}{\left|y\right|}V_{2}(y-v\left(t-\left|y\right|\right))u_{T}(\left(t-\left|y\right|\right))\,\,dy\right\Vert _{L_{t}^{2}[B,T]}\nonumber \\
\lesssim\left(\frac{A^{2}}{\left\langle A\right\rangle ^{\alpha}}\right)\sup_{x}\left\Vert u_{T}\right\Vert _{L_{t}^{2}[0,T]}\\
\lesssim\frac{1}{A}\sup_{x}\left\Vert u_{T}\right\Vert _{L_{t}^{2}[0,T]}.\nonumber 
\end{align}
All other terms can be estimated by a similar way.
\end{rem}

\section{Strichartz Estimates and Energy Bound \label{sec:StrichWOB}}

We know from the introduction that weighted estimates play important
roles in building Strichartz estimates. In this section, we establish
weighted estimates for a scattering state to the wave equation with
charge transfer Hamiltonian. Just for the sake of convenience, we
will restate our main theorems in this section. 

Throughout this subsection, we will use the short-hand notation 
\begin{equation}
L_{t}^{p}L_{x}^{q}:=L_{t}^{p}\left([0,\infty),\,L_{x}^{q}\right).\label{eq:notation-1}
\end{equation}

\begin{cor}
\label{cor:weightChargeWOB}Let $\left|v\right|<1$. Suppose $u$
is a scattering state in the sense of Definition \ref{AO} and that
it solves 
\begin{equation}
\partial_{tt}u-\Delta u+V_{1}(x)u+V_{2}(x-\vec{v}t)u=0
\end{equation}
with initial data
\begin{equation}
u(x,0)=g(x),\,u_{t}(x,0)=f(x)
\end{equation}
Then for $\alpha>3$,
\begin{equation}
\int_{\mathbb{R}^{+}}\int_{\mathbb{R}^{3}}\frac{1}{\left\langle x\right\rangle ^{\alpha}}u^{2}(x,t)\,dxdt\lesssim\left(\|f\|_{L^{2}}+\|g\|_{\dot{H}^{1}}\right)^{2}\label{eq:WeightCharWOB1}
\end{equation}
and

\textup{
\begin{equation}
\int_{\mathbb{R}^{+}}\int_{\mathbb{R}^{3}}\frac{1}{\left\langle x-\vec{v}t\right\rangle ^{\alpha}}u^{2}(x,t)\,dxdt\lesssim\left(\|f\|_{L^{2}}+\|g\|_{\dot{H}^{1}}\right)^{2}\label{eq:WeightCharWOB2}
\end{equation}
}
\end{cor}
\begin{proof}
The two weighted estimates above follow easily from Theorem \ref{thm:EndRSChWOB}. 

For the first one, 
\begin{eqnarray}
\int_{\mathbb{R}^{+}}\int_{\mathbb{R}^{3}}\frac{1}{\left\langle x\right\rangle ^{\alpha}}u^{2}(x,t)\,dxdt & \lesssim & \left(\int_{\mathbb{R}^{3}}\frac{1}{\left\langle x\right\rangle ^{\alpha}}dx\right)\sup_{x}\int_{\mathbb{R}^{+}}u^{2}(x,t)\,dt\nonumber \\
 & \lesssim & \left(\|f\|_{L^{2}}+\|g\|_{\dot{H}^{1}}\right)^{2}
\end{eqnarray}
by the endpoint reversed Strichartz estimate \eqref{eq:EndRSChWOB}
for $u$.

For the second one, one has
\begin{eqnarray}
\int_{\mathbb{R}^{+}}\int_{\mathbb{R}^{3}}\frac{1}{\left\langle x-\vec{v}t\right\rangle ^{\alpha}}u^{2}(x,t)\,dxdt & \lesssim & \int_{\mathbb{R}}\int_{\mathbb{R}^{3}}\frac{1}{\left\langle y\right\rangle ^{\alpha}}u^{2}(y+vt,t)\,dydt\nonumber \\
 & \lesssim & \left(\int_{\mathbb{R}^{3}}\frac{1}{\left\langle y\right\rangle ^{\alpha}}\,dy\right)\sup_{x}\int_{\mathbb{R}^{+}}\left|u^{S}(x,t)\right|^{2}dt\nonumber \\
 & \lesssim & \left(\|f\|_{L^{2}}+\|g\|_{\dot{H}^{1}}\right)^{2}
\end{eqnarray}
by our estimate \eqref{eq:EndRSChWOBSL} along the slanted line $(x+vt,t)$.

We are done.
\end{proof}
\begin{thm}
\label{thm:StrichaWOB}Let $\left|v\right|<1$. Suppose $u$ is a
scattering state in the sense of Definition \ref{AO} which solves
\begin{equation}
\partial_{tt}u-\Delta u+V_{1}(x)u+V_{2}(x-\vec{v}t)u=0\label{eq:StriTwoE}
\end{equation}
with initial data
\begin{equation}
u(x,0)=g(x),\,u_{t}(x,0)=f(x).
\end{equation}
Then for $p>2$, and $(p,q)$ satisfying 
\begin{equation}
\frac{1}{2}=\frac{1}{p}+\frac{3}{q}
\end{equation}
we have
\begin{equation}
\|u\|_{L_{t}^{p}\left([0,\infty),\,L_{x}^{q}\right)}\lesssim\|f\|_{L^{2}}+\|g\|_{\dot{H}^{1}}\label{eq:StriCharWOB}
\end{equation}
\end{thm}
\begin{proof}
Following \cite{LSch}, we set $A=\sqrt{-\Delta}$ and notice that
\begin{equation}
\left\Vert Af\right\Vert _{L^{2}}\simeq\left\Vert f\right\Vert _{\dot{H}^{1}},\,\,\forall f\in C^{\infty}\left(\mathbb{R}^{3}\right).\label{eq:nabla-1}
\end{equation}
For real-valued $u=\left(u_{1},u_{2}\right)\in\mathcal{H}=\dot{H}^{1}\left(\mathbb{R}^{3}\right)\times L^{2}\left(\mathbb{\mathbb{R}}^{3}\right)$,
we write 
\begin{equation}
U:=Au_{1}+iu_{2}.
\end{equation}
From \eqref{eq:nabla-1}, we know 
\begin{equation}
\left\Vert U\right\Vert _{L^{2}}\simeq\left\Vert \left(u_{1},u_{2}\right)\right\Vert _{\mathcal{H}}.
\end{equation}
We also notice that $u$ solves \eqref{eq:StriTwoE} if and only if
\begin{equation}
U:=Au+i\partial_{t}u
\end{equation}
satisfies 
\begin{equation}
i\partial_{t}U=AU+V_{1}u+V_{2}\left(x-\vec{v}t\right)u,
\end{equation}
\begin{equation}
U(0)=Ag+if\in L^{2}\left(\mathbb{R}^{3}\right).
\end{equation}
By Duhamel's formula, 
\begin{equation}
U(t)=e^{itA}U(0)-i\int_{0}^{t}e^{-i\left(t-s\right)A}\left(V_{1}u(s)+V_{2}\left(\cdot-vs\right)u(s)\right)\,ds.
\end{equation}
Let $P:=A^{-1}\Re$, then from Strichartz estimates for the free evolution,
\begin{equation}
\left\Vert Pe^{itA}U(0)\right\Vert _{L_{t}^{p}L_{x}^{q}}\lesssim\left\Vert U(0)\right\Vert _{L^{2}}.\label{eq:Sfirst-1}
\end{equation}
Writing $V_{1}=V_{3}V_{4}$, $V_{2}=V_{5}V_{6}$ , since $V_{1}$
and $V_{2}$ decay like $\left\langle x\right\rangle ^{-\alpha}$
with $\alpha>3$, we can make $V_{3}$ and $V_{5}$ satisfy the weight
condition in Theorem \ref{thm:local}. Also $V_{4}^{2},\,V_{6}^{2}$
decay with rate $\left\langle x\right\rangle ^{-\alpha}$. By the
Christ-Kiselev lemma, cf.~Lemma \ref{lem:Christ-Kiselev}, it suffices
to bound
\begin{equation}
\left\Vert P\int_{0}^{\infty}e^{-i\left(t-s\right)A}V_{3}V_{4}u(s)\,ds\right\Vert _{L_{t}^{p}L_{x}^{q}},
\end{equation}
and 
\begin{equation}
\left\Vert P\int_{0}^{\infty}e^{-i\left(t-s\right)A}V_{5}V_{6}(\cdot-vs)u(s)\,ds\right\Vert _{L_{t}^{p}L_{x}^{q}}.
\end{equation}
It is clear that 
\begin{equation}
\left\Vert P\int_{0}^{\infty}e^{-i\left(t-s\right)A}V_{3}V_{4}u(s)\,ds\right\Vert _{L_{t}^{p}L_{x}^{q}}\leq\left\Vert K\right\Vert _{L_{t,x}^{2}\rightarrow L_{t}^{p}L_{x}^{q}}\left\Vert V_{4}u\right\Vert _{L_{t,x}^{2}},\label{eq:KV-1}
\end{equation}
where 
\begin{equation}
\left(KF\right)(t):=P\int_{0}^{\infty}e^{-i\left(t-s\right)A}V_{3}F(s)\,ds.
\end{equation}
Similarly, 
\begin{equation}
\left\Vert P\int_{0}^{\infty}e^{-i\left(t-s\right)A}V_{5}V_{6}(\cdot-vs)u(s)\,ds\right\Vert _{L_{t}^{p}L_{x}^{q}}\leq\left\Vert \widetilde{K}\right\Vert _{L_{t,x}^{2}\rightarrow L_{t}^{p}L_{x}^{q}}\left\Vert V_{6}(x-\vec{v}t)u\right\Vert _{L_{t,x}^{2}},
\end{equation}
where 
\begin{equation}
\left(\widetilde{K}F\right)(t):=P\int_{0}^{\infty}e^{-i\left(t-s\right)A}V_{5}(\cdot-vs)F(s)\,ds.
\end{equation}
We need to estimate
\begin{equation}
\left\Vert K\right\Vert _{L_{t,x}^{2}\rightarrow L_{t}^{p}L_{x}^{q}},\,\left\Vert \widetilde{K}\right\Vert _{L_{t,x}^{2}\rightarrow L_{t}^{p}L_{x}^{q}}.
\end{equation}
Testing against $F\in L_{t,x}^{2}$, clearly, 
\begin{equation}
\left\Vert KF\right\Vert _{L_{t}^{p}L_{x}^{q}}\leq\left\Vert Pe^{-itA}\right\Vert _{L^{2}\rightarrow L_{t}^{p}L_{x}^{q}}\left\Vert \int_{0}^{\infty}e^{isA}V_{3}F(s)\,ds\right\Vert _{L^{2}}.\label{eq:KF-1}
\end{equation}
\begin{equation}
\left\Vert \widetilde{K}F\right\Vert _{L_{t}^{p}L_{x}^{q}}\leq\left\Vert Pe^{-itA}\right\Vert _{L^{2}\rightarrow L_{t}^{p}L_{x}^{q}}\left\Vert \int_{0}^{\infty}e^{isA}V_{5}(\cdot-vs)F(s)\,ds\right\Vert _{L^{2}}.\label{eq:TKF-1}
\end{equation}
The first factors on the right-hand side of \eqref{eq:KF-1} and \eqref{eq:TKF-1}
is bounded by Strichartz estimates for the free evolution. Consider
the second factors, by duality, it suffices to show 
\begin{equation}
\left\Vert V_{3}e^{-itA}\phi\right\Vert _{L_{t,x}^{2}}\lesssim\left\Vert \phi\right\Vert _{L^{2}},\,\forall\phi\in L^{2}\left(\mathbb{R}^{3}\right)
\end{equation}
\begin{equation}
\left\Vert V_{5}(x-\vec{v}t)e^{-itA}\phi\right\Vert _{L_{t,x}^{2}}\lesssim\left\Vert \phi\right\Vert _{L^{2}},\,\forall\phi\in L^{2}\left(\mathbb{R}^{3}\right).
\end{equation}
which holds by Theorem \ref{thm:local} and Corollary \ref{thm:local-1}.

Hence 
\begin{equation}
\left\Vert \int_{0}^{\infty}e^{isA}V_{3}F(s)\,ds\right\Vert _{L^{2}}\lesssim\left\Vert F\right\Vert _{L_{t,x}^{2}},
\end{equation}
\begin{equation}
\left\Vert \int_{0}^{\infty}e^{isA}V_{5}(\cdot-vs)F(s)\,ds\right\Vert _{L^{2}}\lesssim\left\Vert F\right\Vert _{L_{t,x}^{2}}.
\end{equation}
Therefore, indeed, we have 
\begin{equation}
\left\Vert K\right\Vert _{L_{t,x}^{2}\rightarrow L_{t}^{p}L_{x}^{q}}\leq C,\,\left\Vert \widetilde{K}\right\Vert _{L_{t,x}^{2}\rightarrow L_{t}^{p}L_{x}^{q}}\leq C
\end{equation}
and from \eqref{eq:KV-1}, it follows that 
\begin{equation}
\left\Vert P\int_{0}^{\infty}e^{-i\left(t-s\right)A}V_{3}V_{4}u(s)\,ds\right\Vert _{L_{t}^{p}L_{x}^{q}}\lesssim\left\Vert V_{4}u\right\Vert _{L_{t,x}^{2}},
\end{equation}
\begin{equation}
\left\Vert P\int_{0}^{\infty}e^{-i\left(t-s\right)A}V_{5}V_{6}(\cdot-vs)u(s)\,ds\right\Vert _{L_{t}^{p}L_{x}^{q}}\lesssim\left\Vert V_{6}(x-\vec{v}t)u\right\Vert _{L_{t,x}^{2}}.
\end{equation}
By estimates \eqref{eq:WeightCharWOB1} and \eqref{eq:WeightCharWOB2}
from Corollary \ref{cor:weightChargeWOB}, 
\begin{equation}
\left\Vert V_{4}u\right\Vert _{L_{t,x}^{2}}\lesssim\left(\int_{\mathbb{R}^{+}}\int_{\mathbb{R}^{3}}\frac{1}{\left\langle x\right\rangle ^{\alpha}}\left|u(x,t)\right|^{2}dxdt\right)^{\frac{1}{2}}\lesssim\|f\|_{L^{2}}+\|g\|_{\dot{H}^{1}},
\end{equation}
\begin{equation}
\left\Vert V_{6}(x-\vec{v}t)u\right\Vert _{L_{t,x}^{2}}\lesssim\left(\int_{\mathbb{R}^{+}}\int_{\mathbb{R}^{3}}\frac{1}{\left\langle x-\vec{v}t\right\rangle ^{\alpha}}\left|u(x,t)\right|^{2}dxdt\right)^{\frac{1}{2}}\lesssim\|f\|_{L^{2}}+\|g\|_{\dot{H}^{1}}.
\end{equation}
They follows that 
\begin{equation}
\left\Vert P\int_{0}^{\infty}e^{-i\left(t-s\right)A}V_{3}V_{4}u(s)\,ds\right\Vert _{L_{t}^{p}L_{x}^{q}}\lesssim\|f\|_{L^{2}}+\|g\|_{\dot{H}^{1}}.\label{eq:Ssecond-1}
\end{equation}
\begin{equation}
\left\Vert P\int_{0}^{\infty}e^{-i\left(t-s\right)A}V_{5}V_{6}(\cdot-vs)u(s)\,ds\right\Vert _{L_{t}^{p}L_{x}^{q}}\lesssim\|f\|_{L^{2}}+\|g\|_{\dot{H}^{1}}.\label{eq:Ssecond-2}
\end{equation}
Therefore, by estimates \eqref{eq:Sfirst-1}, \eqref{eq:Ssecond-1} and
\eqref{eq:Ssecond-2}, for $p>2$, and 
\[
\frac{1}{2}=\frac{1}{p}+\frac{3}{q}
\]
we have
\begin{equation}
\|u\|_{L_{t}^{p}\left([0,\infty),\,L_{x}^{q}\right)}\lesssim\|f\|_{L^{2}}+\|g\|_{\dot{H}^{1}}.
\end{equation}
as claimed.
\end{proof}
Taking the case $p=q$ in the regular Strichartz estimate \eqref{eq:StriCharWOB}
and interpolating it with the endpoint reversed Strichartz estimate
\eqref{eq:EndRSChWOB}, we obtain more reversed Strichartz estimates.
\begin{cor}
\label{cor: RevCharB}Let $\left|v\right|<1$. Suppose $u$ is a scattering
state in the sense of Definition \ref{AO} which solves\textup{
\begin{equation}
\partial_{tt}u-\Delta u+V_{1}(x)u+V_{2}(x-\vec{v}t)u=0
\end{equation}
}with initial data 
\begin{equation}
u(x,0)=g(x),\,u_{t}(x,0)=f(x).
\end{equation}
Then for $(p,q)$ satisfying
\begin{equation}
\frac{1}{2}=\frac{1}{p}+\frac{3}{q}
\end{equation}
with 
\begin{equation}
2\leq p\leq8,
\end{equation}
we have
\begin{equation}
\|u\|_{L_{x}^{q}\left(\mathbb{R}^{3},\,L_{t}^{p}[0,\infty)\right)}\lesssim\|f\|_{L^{2}}+\|g\|_{\dot{H}^{1}}.\label{eq:RevCharB}
\end{equation}
\end{cor}
\begin{thm}
\label{thm:EnergyCharge}Let $\left|v\right|<1$. Suppose $u$ is
a scattering state in the sense of Definition \ref{AO} which solves
\begin{equation}
\partial_{tt}u-\Delta u+V_{1}(x)u+V_{2}(x-\vec{v}t)u=0\label{eq:StriTwoE-2}
\end{equation}
with initial data
\begin{equation}
u(x,0)=g(x),\,u_{t}(x,0)=f(x).
\end{equation}
Then we have
\begin{equation}
\sup_{t\geq0}\left(\|\nabla u(t)\|_{L^{2}}+\|u_{t}(t)\|_{L^{2}}\right)\lesssim\|f\|_{L^{2}}+\|g\|_{\dot{H}^{1}}.\label{eq:StriCharWOB-2}
\end{equation}
\end{thm}
\begin{proof}
The proof is similar to Theorem \ref{thm:StrichaWOB}. We still use
the notations from the above proof of Theorem \ref{thm:StrichaWOB}.

Set 
\begin{equation}
U:=Au+i\partial_{t}u,
\end{equation}
then by Duhamel's formula, 
\begin{equation}
U(t)=e^{itA}U(0)-i\int_{0}^{t}e^{-i\left(t-s\right)A}\left(V_{1}u(s)+V_{2}\left(\cdot-vs\right)u(s)\right)\,ds.
\end{equation}
It suffices to estimate the $L^{2}$ norm of $U(t)$.

From the energy estimate for the free evolution,
\begin{equation}
\sup_{t\geq0}\left\Vert e^{itA}U(0)\right\Vert _{L_{x}^{2}}\lesssim\left\Vert U(0)\right\Vert _{L^{2}}.\label{eq:Sfirst-1-2}
\end{equation}
Writing $V_{1}=V_{3}V_{4}$, $V_{2}=V_{5}V_{6}$ as above, it suffices
to bound
\begin{equation}
\sup_{t\geq0}\left\Vert \int_{0}^{\infty}e^{-i\left(t-s\right)A}V_{3}V_{4}u(s)\,ds\right\Vert _{L_{x}^{2}},
\end{equation}
and 
\begin{equation}
\sup_{t\geq0}\left\Vert \int_{0}^{\infty}e^{-i\left(t-s\right)A}V_{5}V_{6}(\cdot-vs)u(s)\,ds\right\Vert _{L_{x}^{2}}.
\end{equation}
It is clear that 
\begin{equation}
\sup_{t\geq0}\left\Vert \int_{0}^{\infty}e^{-i\left(t-s\right)A}V_{3}V_{4}u(s)\,ds\right\Vert _{L_{x}^{2}}\leq\left\Vert K\right\Vert _{L_{t}^{2}L_{x}^{2}\rightarrow L_{t}^{\infty}L_{x}^{2}}\left\Vert V_{4}u\right\Vert _{L_{t,x}^{2}},\label{eq:KV-1-2}
\end{equation}
where 
\begin{equation}
\left(KF\right)(t):=\int_{0}^{\infty}e^{-i\left(t-s\right)A}V_{3}F(s)\,ds.
\end{equation}
Similarly, 
\begin{equation}
\left\Vert \int_{0}^{\infty}e^{-i\left(t-s\right)A}V_{5}V_{6}(\cdot-vs)u(s)\,ds\right\Vert _{L_{t}^{\infty}L_{x}^{2}}\leq\left\Vert \widetilde{K}\right\Vert _{L_{t}^{2}L_{x}^{2}\rightarrow L_{t}^{\infty}L_{x}^{2}}\left\Vert V_{6}(x-\vec{v}t)u\right\Vert _{L_{t}^{2}L_{x}^{2}},
\end{equation}
where 
\begin{equation}
\left(\widetilde{K}F\right)(t):=\int_{0}^{\infty}e^{-i\left(t-s\right)A}V_{5}(\cdot-vs)F(s)\,ds.
\end{equation}
We need to estimate
\begin{equation}
\left\Vert K\right\Vert _{L_{t}^{2}L_{x}^{2}\rightarrow L_{t}^{\infty}L_{x}^{2}},\,\left\Vert \widetilde{K}\right\Vert _{L_{t}^{2}L_{x}^{2}\rightarrow L_{t}^{\infty}L_{x}^{2}}.
\end{equation}
Testing against $F\in L_{t}^{2}\left([0,\infty),\,L_{x}^{2}\right)$,
clearly, 
\begin{equation}
\left\Vert KF\right\Vert _{L_{t}^{\infty}L_{x}^{2}}\leq\left\Vert e^{-itA}\right\Vert _{L^{2}\rightarrow L_{t}^{\infty}L_{x}^{2}}\left\Vert \int_{0}^{\infty}e^{isA}V_{3}F(s)\,ds\right\Vert _{L^{2}}.\label{eq:KF-1-2}
\end{equation}
\begin{equation}
\left\Vert \widetilde{K}F\right\Vert _{L_{t}^{\infty}L_{x}^{2}}\leq\left\Vert e^{-itA}\right\Vert _{L^{2}\rightarrow L_{t}^{\infty}L_{x}^{2}}\left\Vert \int_{0}^{\infty}e^{isA}V_{5}(\cdot-vs)F(s)\,ds\right\Vert _{L^{2}}.\label{eq:TKF-1-2}
\end{equation}
The first factors on the right-hand side of \eqref{eq:KF-1-2} and \eqref{eq:TKF-1-2}
is bounded by the energy estimates for the free evolution. The second
factors are estimated in the same manner as for \eqref{eq:KF-1} and
\eqref{eq:TKF-1}. 

Therefore, we have 
\begin{equation}
\left\Vert K\right\Vert _{L_{t}^{2}L_{x}^{2}\rightarrow L_{t}^{\infty}L_{x}^{2}}\leq C,\,\left\Vert \widetilde{K}\right\Vert _{L_{t}^{2}L_{x}^{2}\rightarrow L_{t}^{\infty}L_{x}^{2}}\leq C
\end{equation}
and from \eqref{eq:KV-1-2}, 
\begin{equation}
\sup_{t\geq0}\left\Vert \int_{0}^{\infty}e^{-i\left(t-s\right)A}V_{3}V_{4}u(s)\,ds\right\Vert _{L_{x}^{2}}\lesssim\left\Vert V_{4}u\right\Vert _{L_{t,x}^{2}},
\end{equation}
\begin{equation}
\sup_{t\geq0}\left\Vert \int_{0}^{\infty}e^{-i\left(t-s\right)A}V_{5}V_{6}(\cdot-vs)u(s)\,ds\right\Vert _{L_{x}^{2}}\lesssim\left\Vert V_{6}(x-\vec{v}t)u\right\Vert _{L_{t,x}^{2}}.
\end{equation}
From Corollary \ref{cor:weightChargeWOB}, 
\begin{equation}
\left\Vert V_{4}u\right\Vert _{L_{t}^{2}L_{x}^{2}}\lesssim\left(\int_{\mathbb{R}^{+}}\int_{\mathbb{R}^{3}}\frac{1}{\left\langle x\right\rangle ^{\alpha}}\left|u(x,t)\right|^{2}dxdt\right)^{\frac{1}{2}}\lesssim\|f\|_{L^{2}}+\|g\|_{\dot{H}^{1}},
\end{equation}
\begin{equation}
\left\Vert V_{6}(x-\vec{v}t)u\right\Vert _{L_{t}^{2}L_{x}^{2}}\lesssim\left(\int_{\mathbb{R}^{+}}\int_{\mathbb{R}^{3}}\frac{1}{\left\langle x-\vec{v}t\right\rangle ^{\alpha}}\left|u(x,t)\right|^{2}dxdt\right)^{\frac{1}{2}}\lesssim\|f\|_{L^{2}}+\|g\|_{\dot{H}^{1}}.
\end{equation}
They imply 
\begin{equation}
\sup_{t\geq0}\left\Vert \int_{0}^{\infty}e^{-i\left(t-s\right)A}V_{3}V_{4}u(s)\,ds\right\Vert _{L_{x}^{2}}\lesssim\|f\|_{L^{2}}+\|g\|_{\dot{H}^{1}}.\label{eq:Ssecond-1-1}
\end{equation}
\begin{equation}
\sup_{t\geq0}\left\Vert \int_{0}^{\infty}e^{-i\left(t-s\right)A}V_{5}V_{6}(\cdot-vs)u(s)\,ds\right\Vert _{L_{x}^{2}}\lesssim\|f\|_{L^{2}}+\|g\|_{\dot{H}^{1}}.\label{eq:Ssecond-2-1}
\end{equation}
Therefore, with estimates \eqref{eq:Sfirst-1-2}, \eqref{eq:Ssecond-1-1}
and \eqref{eq:Ssecond-2-1}, we have 
\begin{equation}
\sup_{t\geq0}\left(\|\nabla u(t)\|_{L^{2}}+\|u_{t}(t)\|_{L^{2}}\right)\lesssim\|f\|_{L^{2}}+\|g\|_{\dot{H}^{1}}.
\end{equation}
as claimed.
\end{proof}
Similarly, one can also obtain the local energy decay estimate:
\begin{thm}
\label{thm:LEnergyCharge}Let $\left|v\right|<1$. Suppose $u$ is
a scattering state in the sense of Definition \ref{AO} which solves
\begin{equation}
\partial_{tt}u-\Delta u+V_{1}(x)u+V_{2}(x-\vec{v}t)u=0\label{eq:StriTwoE-1-2}
\end{equation}
with initial data
\begin{equation}
u(x,0)=g(x),\,u_{t}(x,0)=f(x).
\end{equation}
Then for $\forall\epsilon>0,\,\left|\mu\right|<1$, we have 
\begin{equation}
\left\Vert \left(1+\left|x-\mu t\right|\right)^{-\frac{1}{2}-\epsilon}\left(\left|\nabla u\right|+\left|u_{t}\right|\right)\right\Vert _{L^{2}\left([0,\infty),\,L_{x}^{2}\right)}\lesssim_{\mu,\epsilon}\|f\|_{L^{2}}+\|g\|_{\dot{H}^{1}}.
\end{equation}
\end{thm}
\begin{proof}
The proof is the same as above with the energy estimate for the free
wave equation replaced by the local energy decay estimate for the
free wave equation. 
\begin{equation}
\left\Vert \left(1+\left|x-\mu t\right|\right)^{-\frac{1}{2}-\epsilon}e^{it\sqrt{-\Delta}}f\right\Vert _{L^{2}\left([0,\infty),\,L_{x}^{2}\right)}\lesssim_{\mu,\epsilon}\left\Vert f\right\Vert _{L_{x}^{2}}.
\end{equation}
The claim follows easily.
\end{proof}
Finally, we consider the boundedness of the total energy. We denote
the total energy by
\begin{equation}
E(t)=\int\left|\nabla_{x}u\right|^{2}+\left|\partial_{t}u\right|^{2}+V_{1}\left|u\right|^{2}+V_{2}(x-\vec{v}t)\left|u\right|^{2}dx.\label{eq:Energy}
\end{equation}

\begin{cor}
\label{cor:wholeEneWOB}Let $\left|v\right|<1$ and suppose $u$ is
a scattering state in the sense of Definition \ref{AO} and solves
\begin{equation}
\partial_{tt}u-\Delta u+V_{1}(x)u+V_{2}(x-\vec{v}t)u=0
\end{equation}
with initial data
\begin{equation}
u(x,0)=g(x),\,u_{t}(x,0)=f(x).
\end{equation}
Assume 
\begin{equation}
\left\Vert \nabla V_{2}\right\Vert _{L^{1}}<\infty,
\end{equation}
then $E(t)$ is bounded by the initial energy independently of $t$,
\[
\sup_{t\geq0}E(t)\lesssim\left\Vert \left(g,f\right)\right\Vert _{\dot{H}^{1}\times L^{2}}^{2}.
\]
\end{cor}
\begin{proof}
We might assume $u$ is smooth. Taking time derivative of $E(t)$,
with the fact $u$ solves 
\begin{equation}
\partial_{tt}u-\Delta u+V_{1}(x)u+V_{2}(x-\vec{v}t)u=0,
\end{equation}
one obtains

\begin{equation}
\partial_{t}E(t)=\int_{\mathbb{R}^{3}}\partial_{t}V_{2}(x-\vec{v}t)\left|u\right|^{2}(x,t)dx=-v\int_{\mathbb{R}^{3}}\partial_{x}V_{2}(x)\left|u^{S}(x,t)\right|^{2}dx
\end{equation}
by a simple change of variable. 

Note that 
\begin{eqnarray}
\int_{0}^{\infty}\left|\partial_{t}E(t)\right|dt & \lesssim & \int_{0}^{\infty}\int_{\mathbb{R}^{3}}\left|\partial_{x}V_{2}(x)\right|\left|u^{S}(x,t)\right|^{2}dxdt\nonumber \\
 & = & \left\Vert \partial_{x}V_{2}\right\Vert _{L_{x}^{1}}\left\Vert u^{S}\right\Vert _{L_{x}^{\infty}L_{t}^{2}[0,\infty)}^{2}\\
 & \lesssim & \left\Vert \left(g,f\right)\right\Vert _{\dot{H}^{1}\times L^{2}}^{2}.\nonumber 
\end{eqnarray}
For arbitrary $t\in\mathbb{R}$, we have
\begin{equation}
E(t)-E(0)\leq\int_{\mathbb{R}^{+}}\left|\partial_{t}E(t)\right|dt\lesssim\left\Vert \left(g,f\right)\right\Vert _{\dot{H}^{1}\times L^{2}}^{2}
\end{equation}
which implies 
\begin{equation}
\sup_{t\geq0}E(t)\lesssim\left\Vert \left(g,f\right)\right\Vert _{\dot{H}^{1}\times L^{2}}^{2}
\end{equation}
as claimed.

To finish this section, we prove a version of endpoint Strichartz
estimate which is inhomogeneous with respect to angular and radial
variables. 
\end{proof}
\begin{thm}
\label{thm:EndStri}Let $\left|v\right|<1$. Suppose $u$ is a scattering
state in the sense of Definition \ref{AO} which solves 
\begin{equation}
\partial_{tt}u-\Delta u+V_{1}(x)u+V_{2}(x-\vec{v}t)u=0\label{eq:StriTwoE-3}
\end{equation}
with initial data
\begin{equation}
u(x,0)=g(x),\,u_{t}(x,0)=f(x).
\end{equation}
Then for $1\leq p<\infty$, 
\begin{equation}
\|u\|_{L_{t}^{2}\left([0,\infty),\,L_{r}^{\infty}L_{\omega}^{p}\right)}\lesssim\|f\|_{L^{2}}+\|g\|_{\dot{H}^{1}}\label{eq:EndStri}
\end{equation}
\end{thm}
\begin{proof}
First of all, we consider a auxiliary function given by 
\begin{align}
v(x,t) & =\frac{\sin\left(t\sqrt{-\Delta}\right)}{\sqrt{-\Delta}}f+\cos\left(t\sqrt{-\Delta}\right)\\
 & +\int_{0}^{\infty}\frac{\sin\left(\left(t-s\right)\sqrt{-\Delta}\right)}{\sqrt{-\Delta}}\left(\left|V_{1}u(s)\right|+\left|V_{2}\left(\cdot-vs\right)u(s)\right|\right)ds.\nonumber 
\end{align}
Since in $\mathbb{R}^{3}$ $\frac{\sin\left(t\sqrt{-\Delta}\right)}{\sqrt{-\Delta}}$
is a positive operator, we know 
\begin{align}
\left|\int_{0}^{t}\frac{\sin\left(\left(t-s\right)\sqrt{-\Delta}\right)}{\sqrt{-\Delta}}\left(V_{1}u(s)+V_{2}\left(\cdot-vs\right)u(s)\right)ds\right|\\
\lesssim\int_{0}^{\infty}\frac{\sin\left(\left(t-s\right)\sqrt{-\Delta}\right)}{\sqrt{-\Delta}}\left(\left|V_{1}u(s)\right|+\left|V_{2}\left(\cdot-vs\right)u(s)\right|\right)ds\nonumber 
\end{align}
and it follows 
\begin{equation}
u(x,t)\leq v(x,t).
\end{equation}
We need to estimate the Strichartz norm of 
\begin{equation}
\int_{0}^{\infty}\frac{\sin\left(\left(t-s\right)\sqrt{-\Delta}\right)}{\sqrt{-\Delta}}\left(\left|V_{1}u(s)\right|+\left|V_{2}\left(\cdot-vs\right)u(s)\right|\right)ds.
\end{equation}
Clearly, 
\begin{align}
\int_{0}^{\infty}\frac{\sin\left(\left(t-s\right)\sqrt{-\Delta}\right)}{\sqrt{-\Delta}}\left(\left|V_{1}u(s)\right|+\left|V_{2}\left(\cdot-vs\right)u(s)\right|\right)ds\\
=P\int_{0}^{\infty}e^{-i\left(t-s\right)A}\left(\left|V_{1}u(s)\right|+\left|V_{2}\left(\cdot-vs\right)u(s)\right|\right)ds.\nonumber 
\end{align}
So now we can follow the same scheme as before to consider the following
two estimates 

\begin{equation}
\left\Vert P\int_{0}^{\infty}e^{-i\left(t-s\right)A}\left|V_{1}u(s)\right|\,ds\right\Vert _{L_{t}^{2}L_{r}^{\infty}L_{\omega}^{p}},
\end{equation}
\begin{equation}
\left\Vert P\int_{0}^{\infty}e^{-i\left(t-s\right)A}\left|V_{2}\left(\cdot-vs\right)u(s)\right|\,ds\right\Vert _{L_{t}^{2}L_{r}^{\infty}L_{\omega}^{p}}.
\end{equation}
As above, we write $V_{1}=V_{3}V_{4}$, $V_{2}=V_{5}V_{6}$ . It is
clear that 
\begin{equation}
\left\Vert P\int_{0}^{\infty}e^{-i\left(t-s\right)A}\left|V_{1}u(s)\right|\,ds\right\Vert _{L_{t}^{2}L_{r}^{\infty}L_{\omega}^{p}}\leq\left\Vert K\right\Vert _{L_{t,x}^{2}\rightarrow L_{t}^{2}L_{r}^{\infty}L_{\omega}^{p}}\left\Vert V_{4}u\right\Vert _{L_{t,x}^{2}},\label{eq:KV-1-1}
\end{equation}
where 
\begin{equation}
\left(KF\right)(t):=P\int_{0}^{\infty}e^{-i\left(t-s\right)A}\left|V_{3}\right|F(s)\,ds.
\end{equation}
Similarly, 
\begin{equation}
\left\Vert P\int_{0}^{\infty}e^{-i\left(t-s\right)A}\left|V_{2}\left(\cdot-vs\right)u(s)\right|\,ds\right\Vert _{L_{t}^{2}L_{r}^{\infty}L_{\omega}^{p}}.\leq\left\Vert \widetilde{K}\right\Vert _{L_{t,x}^{2}\rightarrow L_{t}^{2}L_{r}^{\infty}L_{\omega}^{p}}\left\Vert V_{6}(x-\vec{v}t)u\right\Vert _{L_{t,x}^{2}},
\end{equation}
where 
\begin{equation}
\left(\widetilde{K}F\right)(t):=P\int_{0}^{\infty}e^{-i\left(t-s\right)A}V_{5}(\cdot-vs)F(s)\,ds.
\end{equation}
We need to estimate
\begin{equation}
\left\Vert K\right\Vert _{L_{t,x}^{2}\rightarrow L_{t}^{2}L_{r}^{\infty}L_{\omega}^{p}},\,\left\Vert \widetilde{K}\right\Vert _{L_{t,x}^{2}\rightarrow L_{t}^{2}L_{r}^{\infty}L_{\omega}^{p}}.
\end{equation}
Testing against $F\in L_{t,x}^{2}$, clearly, 
\begin{equation}
\left\Vert KF\right\Vert _{L_{t}^{p}L_{x}^{q}}\leq\left\Vert Pe^{-itA}\right\Vert _{L^{2}\rightarrow L_{t}^{2}L_{r}^{\infty}L_{\omega}^{p}}\left\Vert \int_{0}^{\infty}e^{isA}V_{3}F(s)\,ds\right\Vert _{L^{2}}.\label{eq:KF-1-1}
\end{equation}
\begin{equation}
\left\Vert \widetilde{K}F\right\Vert _{L_{t}^{p}L_{x}^{q}}\leq\left\Vert Pe^{-itA}\right\Vert _{L^{2}\rightarrow L_{t}^{2}L_{r}^{\infty}L_{\omega}^{p}}\left\Vert \int_{0}^{\infty}e^{isA}V_{5}(\cdot-vs)F(s)\,ds\right\Vert _{L^{2}}.\label{eq:TKF-1-1}
\end{equation}
The first factors on the right-hand side of \eqref{eq:KF-1-1} and \eqref{eq:TKF-1-1}
is bounded by the endpoint Strichartz estimates for the free evolution.
For the second factors, we can bound them as previous proofs. 

Therefore, indeed, we have 
\begin{equation}
\left\Vert K\right\Vert _{L_{t,x}^{2}\rightarrow L_{t}^{2}L_{r}^{\infty}L_{\omega}^{p}}\leq C,\,\left\Vert \widetilde{K}\right\Vert _{L_{t,x}^{2}\rightarrow L_{t}^{2}L_{r}^{\infty}L_{\omega}^{p}}\leq C
\end{equation}
and from \eqref{eq:KV-1-1}, it follows that 
\begin{equation}
\left\Vert P\int_{0}^{\infty}e^{-i\left(t-s\right)A}\left|V_{1}u(s)\right|\,ds\right\Vert _{L_{t}^{2}L_{r}^{\infty}L_{\omega}^{p}}\lesssim\left\Vert V_{4}u\right\Vert _{L_{t,x}^{2}},
\end{equation}
\begin{equation}
\left\Vert P\int_{0}^{\infty}e^{-i\left(t-s\right)A}\left|V_{2}\left(\cdot-vs\right)u(s)\right|\,ds\right\Vert _{L_{t}^{2}L_{r}^{\infty}L_{\omega}^{p}}\lesssim\left\Vert V_{6}(x-\vec{v}t)u\right\Vert _{L_{t,x}^{2}}.
\end{equation}
They follows that 
\begin{equation}
\left\Vert P\int_{0}^{\infty}e^{-i\left(t-s\right)A}\left|V_{1}u(s)\right|\,ds\right\Vert _{L_{t}^{2}L_{r}^{\infty}L_{\omega}^{p}}\lesssim\|f\|_{L^{2}}+\|g\|_{\dot{H}^{1}}.\label{eq:Ssecond-1-2}
\end{equation}
\begin{equation}
\left\Vert P\int_{0}^{\infty}e^{-i\left(t-s\right)A}\left|V_{2}\left(\cdot-vs\right)u(s)\right|\,ds\right\Vert _{L_{t}^{2}L_{r}^{\infty}L_{\omega}^{p}}\lesssim\|f\|_{L^{2}}+\|g\|_{\dot{H}^{1}}.\label{eq:Ssecond-2-2}
\end{equation}
Therefore,
\begin{align}
\left\Vert \int_{0}^{t}\frac{\sin\left(\left(t-s\right)\sqrt{-\Delta}\right)}{\sqrt{-\Delta}}\left(V_{1}u(s)+V_{2}\left(\cdot-vs\right)u(s)\right)ds\right\Vert _{L_{t}^{2}L_{r}^{\infty}L_{\omega}^{p}}\\
\lesssim\left\Vert \int_{0}^{\infty}\frac{\sin\left(\left(t-s\right)\sqrt{-\Delta}\right)}{\sqrt{-\Delta}}\left(\left|V_{1}u(s)\right|+\left|V_{2}\left(\cdot-vs\right)u(s)\right|\right)ds\right\Vert _{L_{t}^{2}L_{r}^{\infty}L_{\omega}^{p}}\nonumber \\
\lesssim\|f\|_{L^{2}}+\|g\|_{\dot{H}^{1}}.\nonumber 
\end{align}
And hence
\begin{equation}
\|u\|_{L_{t}^{2}\left([0,\infty),\,L_{r}^{\infty}L_{\omega}^{p}\right)}\lesssim\|f\|_{L^{2}}+\|g\|_{\dot{H}^{1}}
\end{equation}
as claimed.
\end{proof}

\section{Inhomogeneous Estimates\label{sec:Inhom}}

When we consider nonlinear applications, it is useful to have estimates
for inhomogeneous equations. Again, for simplicity we consider the
case of two potentials. 

\subsection{Scattering states}

We start with revisiting scattering states. 

Recall that if $u$ solves 

\begin{equation}
\partial_{tt}u-\Delta u+V_{1}(x)u+V_{2}(x-\vec{v}t)u=0
\end{equation}
and $u$ satisfies 
\begin{equation}
\left\Vert P_{b}\left(H_{1}\right)u(t)\right\Vert _{L_{x}^{2}}\rightarrow0,\,\,\left\Vert P_{b}\left(H_{2}\right)u_{L}(t')\right\Vert _{L_{x'}^{2}}\rightarrow0\,\,\,t,t'\rightarrow\infty,
\end{equation}
then we call it a scattering state.

Clearly, the set of $\left(g,\,f\right)\in H^{1}\left(\mathbb{R}^{3}\right)\times L^{2}\left(\mathbb{R}^{3}\right)$
which produce a scattering state forms a subspace of $H^{1}\left(\mathbb{R}^{3}\right)\times L^{2}\left(\mathbb{R}^{3}\right)$.
In order to study this more precisely, we reformulate the wave equation
as a Hamiltonian system, 
\begin{equation}
\partial_{t}\left(\begin{array}{c}
u\\
\partial_{t}u
\end{array}\right)-\left(\begin{array}{cc}
0 & 1\\
-1 & 0
\end{array}\right)\left(\begin{array}{cc}
-\Delta+V_{1}(x)+V_{2}(x-\vec{v}t) & 0\\
0 & 1
\end{array}\right)\left(\begin{array}{c}
u\\
\partial_{t}u
\end{array}\right)=0.
\end{equation}
Setting
\begin{equation}
U:=\left(\begin{array}{c}
u\\
\partial_{t}u
\end{array}\right),\,J:=\left(\begin{array}{cc}
0 & 1\\
-1 & 0
\end{array}\right),
\end{equation}
\begin{equation}
H(t):=\left(\begin{array}{cc}
-\Delta+V_{1}(x)+V_{2}(x-\vec{v}t) & 0\\
0 & 1
\end{array}\right),
\end{equation}
and defining
\begin{equation}
P_{1}(U):=u,\label{eq:Pfirst}
\end{equation}
we can rewrite the wave equation with charge transfer Hamiltonian
as 
\begin{equation}
\dot{U}-JH(t)U=0,
\end{equation}
\begin{equation}
U(0)=\left(\begin{array}{c}
g\\
f
\end{array}\right).
\end{equation}
With the above notations, we define the solution operator starting
from $\tau$ to $t$ as $S(t,\tau)$. In particular, one can write
\begin{equation}
U(t)=S(t,0)U(0).
\end{equation}

As pointed out above, the set of $\left(g,\,f\right)\in H^{1}\left(\mathbb{R}^{3}\right)\times L^{2}\left(\mathbb{R}^{3}\right)$
which produce a scattering state in the sense of Definition \ref{AO}
forms a subspace 
\[
\mathcal{H}_{s}(0)\mbox{\ensuremath{\subset}}H^{1}\left(\mathbb{R}^{3}\right)\times L^{2}\left(\mathbb{R}^{3}\right).
\]
We can do a more general time-dependent construction. One considers
the evolution from $\tau$ to $t$, i.e., $S(t,\tau)$. Similar as
our original construction there is a subspace 
\[
\mathcal{H}_{s}(\tau)\mbox{\ensuremath{\subset}}H^{1}\left(\mathbb{R}^{3}\right)\times L^{2}\left(\mathbb{R}^{3}\right)
\]
 such that for $\Phi\in\mathcal{H}_{s}(\tau)$,
\begin{equation}
\left\Vert P_{b}\left(H_{1}\right)S(t,\tau)\Phi\right\Vert _{L_{x}^{2}}\rightarrow0,\,\,\left\Vert P_{b}\left(H_{2}\right)\left(S(\cdot,\tau)\Phi\right)_{L_{\tau}}(t')\right\Vert _{L_{x'}^{2}}\rightarrow0\,\,\,t,t'\rightarrow\infty.
\end{equation}
It is important to notice a fundamental property of $\mathcal{H}_{s}\mbox{(\ensuremath{\tau})}$. 
\begin{lem}
\label{lem:projS}Denote $P_{s}(\tau)$ as the projection onto $\mathcal{H}_{s}(\tau)$.
Then $\forall s,\,\tau\in\mathbb{R}$, 
\begin{equation}
P_{s}(s)S(s,\tau)=S(s,\tau)P_{s}(\tau).\label{eq:commute}
\end{equation}
\end{lem}
\begin{proof}
Notice that for $\Phi\in\mathcal{H}_{s}(\tau)$, then $S(s,\tau)\Phi\in\mathcal{H}_{s}(s)$.
Since
\begin{equation}
\left\Vert P_{b}\left(H_{1}\right)S(t,s)S(s,\tau)\Phi\right\Vert _{L^{2}}=\left\Vert P_{b}\left(H_{1}\right)S(t,\tau)\Phi\right\Vert _{L^{2}}\rightarrow0,
\end{equation}
\begin{equation}
\left\Vert P_{b}\left(H_{2}\right)\left(S(\cdot,s)S(s,\tau)\Phi\right)_{L_{s}}(t')\right\Vert _{L_{x'}^{2}}=\left\Vert P_{b}\left(H_{2}\right)\left(S(\cdot,\tau)\Phi\right)_{L_{\tau}}(t')\right\Vert _{L_{x'}^{2}}\rightarrow0
\end{equation}
as $t,\,t'\rightarrow\infty$ by the definition of $\mathcal{H}_{s}(\tau).$
Then again by the definition of $\mathcal{H}_{s}(s)$, it is clear
$S(s,\tau)\Phi\in\mathcal{H}_{s}(s)$. Conversely, by symmetry, for
$\Phi\in\mathcal{H}_{s}(s)$, then $S(\tau,s)\Phi\in\mathcal{H}_{s}(\tau)$.
Therefore, we have that the scattering spaces are invariant under
the flow $S(s,\tau)$, 
\begin{equation}
\mathcal{H}_{s}(s)=S(s,\tau)\mathcal{H}_{s}(\tau).
\end{equation}
Let $\Phi\in H^{1}\left(\mathbb{R}^{3}\right)\times L^{2}\left(\mathbb{R}^{3}\right)$,
then $S(s,\tau)P_{s}(\tau)\Phi\in\mathcal{H}_{s}(s)$ by construction.
So 
\begin{eqnarray}
S(s,\tau)P_{s}(\tau)\Phi & = & \left(1-P_{s}(s)\right)S(s,\tau)P_{s}(\tau)\Phi+P_{s}(s)S(s,\tau)P_{s}(\tau)\Phi\nonumber \\
 & = & P_{s}(s)S(s,\tau)P_{s}(\tau)\Phi.
\end{eqnarray}
Similarly, 
\begin{equation}
P_{s}(s)S(s,\tau)\Phi=P_{s}(s)S(s,\tau)P_{s}(\tau)\Phi.
\end{equation}
Hence 
\begin{equation}
P_{s}(s)S(s,\tau)=S(s,\tau)P_{s}(\tau),
\end{equation}
as claimed.
\end{proof}
For wave equations, it is always necessary to exchange the scalar
formulation and the Hamiltonian formulation. Here we introduce some
notations which are useful in our later analysis. We define $P_{s}(\tau)$
via the Hamiltonian formulation above. Now consider a scalar function
$v(x,t)\in C\left(\mathbb{R},\,H^{1}\left(\mathbb{R}^{3}\right)\right)\cap C^{1}\left(\mathbb{R},\,L^{2}\left(\mathbb{R}^{3}\right)\right)$,
it can give the data $\left(v,\,v_{t}\right)$ for the charge transfer
model. We define 
\begin{equation}
P_{s}^{\mathbf{S}}\left(\tau\right)v:=P_{1}P_{s}(\tau)\left(v,\,v_{t}\right),\label{eq:ProjScalar}
\end{equation}
where $P_{1}$ is the projection onto the first component as in \eqref{eq:Pfirst}.
For a vector-valued function $V=\left(\begin{array}{c}
v\\
v_{t}
\end{array}\right)\in H^{1}\left(\mathbb{R}^{3}\right)\times L^{2}\left(\mathbb{R}^{3}\right)$, 
\begin{equation}
P_{s}^{\mathbf{V}}\left(\tau\right)V:=P_{1}P_{s}\left(\tau\right)V,\label{eq:ProjVector}
\end{equation}

Given data $\left(g,\,f\right)\in H^{1}\left(\mathbb{R}^{3}\right)\times L^{2}\left(\mathbb{R}^{3}\right)$,
formally, we can define the evolution from $\tau$ to $t$ associated
with $f$ as 
\begin{equation}
U(t,\tau)f
\end{equation}
 and the evolution associated with $g$ as 
\begin{equation}
\dot{U}(t,\tau)g.
\end{equation}
Here $\dot{U}$ is just a formal notation.

Finally, we consider two special cases. 

Setting $g=0$, then the set of $f\in L^{2}\left(\mathbb{R}^{3}\right)$
such that $\left(0,\,f\right)\in\mathcal{H}_{s}(\tau)$ forms a subspace
of $L^{2}\left(\mathbb{R}^{3}\right)$. We use $L_{s}^{2}\left(\tau\right)$
to denote this subspace and let $P_{s}^{L}(\tau)$ to be the associated
projection. 

Setting $f=0$, then the set of $g\in H^{1}\left(\mathbb{R}^{3}\right)$
such that $\left(g,\,0\right)\in\mathcal{H}_{s}(\tau)$ forms a subspace
of $H^{1}\left(\mathbb{R}^{3}\right)$. We use $H_{s}^{1}\left(\tau\right)$
to denote this subspace and let $P_{s}^{H}(\tau)$ to be the associated
projection.

\subsection{Inhomogeneous local decay estimate and Strichartz estimates}

Throughout this subsection, we will use the short-hand notation 
\begin{equation}
L_{t}^{p}L_{x}^{q}:=L_{t}^{p}\left([0,\infty),\,L_{x}^{q}\right).\label{eq:notation}
\end{equation}

Let $u$ solve 
\begin{equation}
\partial_{tt}u-\Delta u+V_{1}(x)u+V_{2}(x-\vec{v}t)u=0
\end{equation}
with initial data
\begin{equation}
u(x,0)=g(x),\,u_{t}(x,0)=f(x).
\end{equation}
Denote the evolution as
\begin{equation}
u(x,t)=U(t,0)f+\dot{U}(t,0)g.
\end{equation}

From the endpoint reversed Strichartz estimate, Theorem \ref{thm:EndRSChWOB},
with the notations introduced above, we know 
\begin{equation}
\sup_{x}\int_{0}^{\infty}\left|P_{s}^{\mathbf{S}}\left(t\right)u(x,t)\right|^{2}dt\lesssim\left(\|f\|_{L^{2}}+\|g\|_{\dot{H}^{1}}\right)^{2}
\end{equation}
and 
\begin{equation}
\sup_{x}\int_{0}^{\infty}\left|\left(P_{s}^{\mathbf{S}}\left(t\right)u(x,t)\right)^{S}\right|^{2}dt\lesssim\left(\|f\|_{L^{2}}+\|g\|_{\dot{H}^{1}}\right)^{2}.
\end{equation}
Furthermore, we have the following corollary as particular situations:
\begin{cor}
\label{cor:SineEvo} For the evolution $U(t,\tau)$ and the projections
$P_{s}^{\mathbf{S}}\left(t\right),\,P_{s}^{L}(\tau)$  defined above,
one has 
\begin{equation}
\sup_{x}\int_{0}^{\infty}\left|P_{s}^{\mathbf{S}}\left(t\right)U(t,\tau)f\right|^{2}dt=\sup_{x}\int_{0}^{\infty}\left|U(t,\tau)P_{s}^{L}(\tau)f\right|^{2}dt\lesssim\|f\|_{L^{2}}^{2},\label{eq:SineEvo1}
\end{equation}
\begin{equation}
\sup_{x}\int_{0}^{\infty}\left|\left(P_{s}^{\mathbf{S}}\left(t\right)U(t,\tau)f\right)^{S}\right|^{2}dt=\sup_{x}\int_{0}^{\infty}\left|U^{S}(t,\tau)P_{s}^{L}(\tau)f\right|^{2}dt\lesssim\|f\|_{L^{2}}^{2},\label{eq:SineEvo2}
\end{equation}
where $U^{S}$ denotes the integration along the slanted line $(x+vt,t)$. 
\end{cor}
\begin{proof}
This is just the particular cases of what we have discussed above.
\end{proof}
By Corollary \ref{cor:SineEvo}, we have the weighted estimates for
the inhomogeneous evolution.
\begin{lem}
\label{lem:WeiInhom}For $\alpha>3$, with $U(t,\tau)$ and projections
$P_{s}^{\mathcal{S}}\left(t\right),\,P_{s}^{L}(\tau)$ defined above,
we have
\begin{equation}
\left\Vert \left\langle x\right\rangle ^{-\frac{\alpha}{2}}\int_{0}^{t}U(t,\tau)P_{s}^{L}(\tau)H(\tau)\,d\tau\right\Vert _{L_{t}^{2}L_{x}^{2}}=\left\Vert \left\langle x\right\rangle ^{-\frac{\alpha}{2}}\int_{0}^{t}P_{s}^{\mathbf{S}}\left(t\right)U(t,\tau)H(\tau)\,d\tau\right\Vert _{L_{t}^{2}L_{x}^{2}},\label{eq:equiwei1}
\end{equation}
\begin{equation}
\left\Vert \left\langle x\right\rangle ^{-\frac{\alpha}{2}}\int_{0}^{t}U(t,\tau)P_{s}^{L}(\tau)H(\tau)\,d\tau\right\Vert _{L_{t}^{2}L_{x}^{2}}\lesssim\left\Vert H(t)\right\Vert _{L_{t}^{1}L_{x}^{2}},\label{eq:WeiInh1}
\end{equation}
\begin{equation}
\left\Vert \left\langle x\right\rangle ^{-\frac{\alpha}{2}}\int_{0}^{t}U^{S}(t,\tau)P_{s}^{L}(\tau)H(\tau)\,d\tau\right\Vert _{L_{t}^{2}L_{x}^{2}}=\left\Vert \left\langle x\right\rangle ^{-\frac{\alpha}{2}}\int_{0}^{t}\left(P_{s}^{\mathbf{S}}\left(t\right)\left(t\right)U(t,\tau)\right)^{S}H(\tau)\,d\tau\right\Vert _{L_{t}^{2}L_{x}^{2}},\label{eq:equiwei2}
\end{equation}
\begin{equation}
\left\Vert \left\langle x-\vec{v}t\right\rangle ^{-\frac{\alpha}{2}}\int_{0}^{t}U(t,\tau)P_{s}^{L}(\tau)H(\tau)\,d\tau\right\Vert _{L_{t}^{2}L_{x}^{2}}\lesssim\left\Vert H(t)\right\Vert _{L_{t}^{1}L_{x}^{2}}.\label{eq:WeiInh2}
\end{equation}
\end{lem}
\begin{proof}
By the definition of projections, we have 

\[
\left\Vert \left\langle x\right\rangle ^{-\frac{\alpha}{2}}\int_{0}^{t}U(t,\tau)P_{s}^{L}(\tau)H(\tau)\,d\tau\right\Vert _{L_{t}^{2}L_{x}^{2}}=\left\Vert \left\langle x\right\rangle ^{-\frac{\alpha}{2}}\int_{0}^{t}P_{s}^{\mathbf{S}}\left(t\right)U(t,\tau)H(\tau)\,d\tau\right\Vert _{L_{t}^{2}L_{x}^{2}},
\]
\[
\left\Vert \left\langle x\right\rangle ^{-\frac{\alpha}{2}}\int_{0}^{t}U^{S}(t,\tau)P_{s}^{L}(\tau)H(\tau)\,d\tau\right\Vert _{L_{t}^{2}L_{x}^{2}}=\left\Vert \left\langle x\right\rangle ^{-\frac{\alpha}{2}}\int_{0}^{t}\left(P_{s}^{\mathbf{S}}\left(t\right)U(t,\tau)\right)^{S}H(\tau)\,d\tau\right\Vert _{L_{t}^{2}L_{x}^{2}}.
\]
Applying Minkowski's inequality and Corollary \ref{cor:SineEvo},
we have 
\begin{eqnarray*}
\left\Vert \left\langle x\right\rangle ^{-\alpha}\int_{0}^{t}U(t,\tau)P_{s}^{L}(\tau)H(\tau)\,d\tau\right\Vert _{L_{t}^{2}L_{x}^{2}} & \lesssim & \left\Vert \left\langle x\right\rangle ^{-\alpha}\int_{0}^{t}\left|U(t,\tau)P_{s}^{L}(\tau)H(\tau)\right|\,d\tau\right\Vert _{L_{t}^{2}L_{x}^{2}}\\
 & \lesssim & \left\Vert \left\langle x\right\rangle ^{-\alpha}\int_{0}^{\infty}\left|U(t,\tau)P_{s}^{L}(\tau)H(\tau)\right|\,d\tau\right\Vert _{L_{t}^{2}L_{x}^{2}}\\
 & \lesssim & \int_{0}^{\infty}\left\Vert U(t,\tau)P_{s}^{L}(\tau)H(\tau)\right\Vert _{L_{x}^{\infty}L_{t}^{2}}d\tau\\
 & \lesssim & \left\Vert H(t)\right\Vert _{L_{t}^{1}L_{x}^{2}}
\end{eqnarray*}
and similarly, 
\begin{eqnarray*}
\left\Vert \left\langle x-\vec{v}t\right\rangle ^{-\alpha}\int_{0}^{t}U(t,\tau)P_{s}^{L}(\tau)H(\tau)\,d\tau\right\Vert _{L_{t}^{2}L_{x}^{2}} & \lesssim & \int_{0}^{\infty}\left\Vert U^{S}(t,\tau)P_{s}^{L}(\tau)H(\tau)\right\Vert _{L_{x}^{\infty}L_{t}^{2}}d\tau\\
 & \lesssim & \left\Vert H(t)\right\Vert _{L_{t}^{1}L_{x}^{2}}.
\end{eqnarray*}
The lemma is proved.
\end{proof}
With the preparations above, we are ready to proceed to the analysis
of inhomogeneous Strichartz estimates. As one can observe from previous
sections on the homogeneous Strichartz estimates that it suffices
to establish certain local decay estimates. 

Now we set 
\begin{equation}
\partial_{tt}u-\Delta u+V_{1}(x)u+V_{2}(x-\vec{v}t)u=F
\end{equation}
with initial data
\begin{equation}
u(x,0)=g(x),\,u_{t}(x,0)=f(x)
\end{equation}

\begin{lem}
\label{lem:LDF}Suppose $u$ solves\textup{
\begin{equation}
\partial_{tt}u-\Delta u+V_{1}(x)u+V_{2}(x-\vec{v}t)u=F
\end{equation}
}with initial data
\begin{equation}
u(x,0)=g(x),\,u_{t}(x,0)=f(x).
\end{equation}
Then for $\alpha>3$ $\left|v\right|<1$, we have
\begin{equation}
\left\Vert \left\langle x\right\rangle ^{-\frac{\alpha}{2}}P_{s}^{\mathbf{S}}\left(t\right)u\right\Vert _{L_{t}^{2}\left([0,\infty),\,L_{x}^{2}\right)}\lesssim\left\Vert \nabla g\right\Vert _{L_{x}^{2}}+\left\Vert f\right\Vert _{L_{x}^{2}}+\left\Vert F\right\Vert _{L_{t}^{1}\left([0,\infty),\,L_{x}^{2}\right)},\label{eq:inhomweight1}
\end{equation}
and 
\begin{equation}
\left\Vert \left\langle x-\vec{v}t\right\rangle ^{-\frac{\alpha}{2}}P_{s}^{\mathbf{S}}\left(t\right)u\right\Vert _{L_{t}^{2}\left([0,\infty),\,L_{x}^{2}\right)}\lesssim\left\Vert \nabla g\right\Vert _{L_{x}^{2}}+\left\Vert f\right\Vert _{L_{x}^{2}}+\left\Vert F\right\Vert _{L_{t}^{1}\left([0,\infty),\,L_{x}^{2}\right)}.\label{eq:inhomweight2}
\end{equation}
\end{lem}
\begin{proof}
By Duhamel's formula, we write 
\begin{equation}
u(x,t)=U(t,0)f+\dot{U}(t,0)g+\int_{0}^{t}U(t,s)F(s)\,ds.
\end{equation}
\begin{eqnarray}
P_{s}^{\mathbf{S}}\left(t\right)u(x,t) & = & P_{s}^{\mathbf{S}}\left(t\right)\left(U(t,0)f+\dot{U}(t,0)g\right)+\int_{0}^{t}P_{s}^{\mathbf{S}}\left(t\right)U(t,s)F(s)\,ds\nonumber \\
 & = & P_{s}^{\mathbf{S}}\left(t\right)\left(U(t,0)f+\dot{U}(t,0)g\right)+\int_{0}^{t}U(t,s)P_{s}^{L}(s)F(s)\,ds.
\end{eqnarray}
Applying the weighted norms, for the homogeneous part, we know 
\begin{equation}
\left\Vert \left\langle x\right\rangle ^{-\alpha}P_{s}^{\mathbf{S}}\left(t\right)\left(U(t,0)f+\dot{U}(t,0)g\right)\right\Vert _{L_{t}^{2}L_{x}^{2}}\lesssim\left\Vert \nabla g\right\Vert _{L_{x}^{2}}+\left\Vert f\right\Vert _{L_{x}^{2}}
\end{equation}
and 
\begin{equation}
\left\Vert \left\langle x-\vec{v}t\right\rangle ^{-\alpha}P_{s}^{\mathbf{S}}\left(t\right)\left(U(t,0)f+\dot{U}(t,0)g\right)\right\Vert _{L_{t}^{2}L_{x}^{2}}\lesssim\left\Vert \nabla g\right\Vert _{L_{x}^{2}}+\left\Vert f\right\Vert _{L_{x}^{2}}.
\end{equation}
For the inhomogeneous part, by our discussion above, one has 
\begin{equation}
\left\Vert \left\langle x\right\rangle ^{-\alpha}\int_{0}^{t}U(t,s)P_{s}^{L}(s)F(s)\,ds\right\Vert _{L_{t}^{2}L_{x}^{2}}\lesssim\left\Vert F\right\Vert _{L_{t}^{1}L_{x}^{2}},
\end{equation}
and 
\begin{equation}
\left\Vert \left\langle x-\vec{v}t\right\rangle ^{-\alpha}\int_{0}^{t}U(t,s)P_{s}^{L}(s)F(s)\,ds\right\Vert _{L_{t}^{2}L_{x}^{2}}\lesssim\left\Vert F\right\Vert _{L_{t}^{1}L_{x}^{2}}.
\end{equation}
Therefore, one can conclude that 
\begin{equation}
\left\Vert \left\langle x\right\rangle ^{-\frac{\alpha}{2}}P_{s}^{\mathbf{S}}\left(t\right)u\right\Vert _{L_{t}^{2}\left([0,\infty),\,L_{x}^{2}\right)}\lesssim\left\Vert \nabla g\right\Vert _{L_{x}^{2}}+\left\Vert f\right\Vert _{L_{x}^{2}}+\left\Vert F\right\Vert _{L_{t}^{1}\left([0,\infty),\,L_{x}^{2}\right)},
\end{equation}
\begin{equation}
\left\Vert \left\langle x-\vec{v}t\right\rangle ^{-\frac{\alpha}{2}}P_{s}^{\mathbf{S}}\left(t\right)u\right\Vert _{L_{t}^{2}\left([0,\infty),\,L_{x}^{2}\right)}\lesssim\left\Vert \nabla g\right\Vert _{L_{x}^{2}}+\left\Vert f\right\Vert _{L_{x}^{2}}+\left\Vert F\right\Vert _{L_{t}^{1}\left([0,\infty),\,L_{x}^{2}\right)}.
\end{equation}
The lemma is proved.
\end{proof}
With the decay estimate Lemma \ref{lem:LDF}, we can establish Strichartz
estimates using almost identical procedures as for the homogeneous
Strichartz estimates.
\begin{thm}
\label{thm:inhomStric}Let $\left|v\right|<1$ and suppose $u$ solves
\begin{equation}
\partial_{tt}u-\Delta u+V_{1}(x)u+V_{2}(x-\vec{v}t)u=F
\end{equation}
with initial data 
\begin{equation}
u(x,0)=g(x),\,u_{t}(x,0)=f(x).
\end{equation}
Then for $p,\,\tilde{p}>2$, and 
\begin{equation}
\frac{1}{2}=\frac{1}{p}+\frac{3}{q},\,\frac{1}{2}=\frac{1}{\tilde{p}}+\frac{3}{\tilde{q}}
\end{equation}
we have
\begin{equation}
\|P_{s}^{\mathbf{S}}\left(t\right)u\|_{L_{t}^{p}\left([0,\infty),\,L_{x}^{q}\right)}\lesssim\|f\|_{L^{2}}+\|g\|_{\dot{H}^{1}}+\left\Vert F\right\Vert _{L_{t}^{\tilde{p}'}\left([0,\,\infty),\,L_{x}^{\tilde{q}'}\right)\cap L_{t}^{1}\left([0,\infty),\,L_{x}^{2}\right)}\label{eq:inhomst}
\end{equation}
where $\tilde{p}',\,\tilde{q}'$ are H\"older conjugate of $\tilde{p},\,\tilde{q}$.
\end{thm}
\begin{proof}
The proof is almost identical to the one for Theorem \ref{thm:StrichaWOB}.
But we need some preliminary calculations. By Lemma \ref{lem:projS},
we know 
\begin{equation}
P_{s}(s)S(s,\tau)=S(s,\tau)P_{s}(\tau).\label{eq:commt}
\end{equation}
Differentiating \eqref{eq:commt} with respect to $s$ and then setting
both $\tau=s=t$, we have 
\begin{equation}
\dot{P}_{s}(t)=-JH(t)P_{s}(t)+P_{s}(t)JH(t).\label{eq:diffP}
\end{equation}
Just as we discussed about projections, we write 
\begin{equation}
\partial_{tt}u-\Delta u+V_{1}(x)u+V_{2}(x-\vec{v}t)u=F
\end{equation}
as a system:
\begin{equation}
\partial_{t}\left(\begin{array}{c}
u\\
\partial_{t}u
\end{array}\right)-\left(\begin{array}{cc}
0 & 1\\
-1 & 0
\end{array}\right)\left(\begin{array}{cc}
-\Delta+V_{1}(x)+V_{2}(x-\vec{v}t) & 0\\
0 & 1
\end{array}\right)\left(\begin{array}{c}
u\\
\partial_{t}u
\end{array}\right)=\left(\begin{array}{c}
0\\
F(t)
\end{array}\right).
\end{equation}
Then 
\begin{equation}
P_{s}(t)\partial_{t}\left(\begin{array}{c}
u\\
\partial_{t}u
\end{array}\right)-P_{s}(t)\left(\begin{array}{cc}
0 & 1\\
-1 & 0
\end{array}\right)\left(\begin{array}{cc}
-\Delta+V_{1}(x)+V_{2}(x-\vec{v}t) & 0\\
0 & 1
\end{array}\right)=P_{s}(t)\left(\begin{array}{c}
0\\
F(t)
\end{array}\right)
\end{equation}
which is 
\begin{equation}
P_{s}(t)\dot{U}(t)-P_{s}(t)JH(t)=P_{s}(t)F(t).\label{eq:projeq}
\end{equation}
By equations \eqref{eq:commt} and \eqref{eq:projeq}, one has 
\begin{equation}
\frac{d}{dt}\left(P_{s}(t)U(t)\right)-JH(t)P_{s}(t)U(t)=P_{s}(t)F(t).
\end{equation}
Hence returning to our scalar setting, we have 
\begin{equation}
\partial_{tt}\left(P_{s}^{\mathbf{S}}\left(t\right)u\right)+\left(-\Delta+V_{1}(x)+V_{2}(x-\vec{v}t)\right)P_{s}^{\mathbf{S}}\left(t\right)u=P_{s}^{\mathbf{S}}\left(t\right)F(t).
\end{equation}
Now we are ready to proceed to the Strichartz estimates argument similar
to the case in Theorem \ref{thm:StrichaWOB}.

Again, following \cite{LSch}, setting $A=\sqrt{-\Delta}$ and taking
\[
U(t)=AP_{s}^{\mathbf{S}}\left(t\right)u(t)+i\partial_{t}\left(P_{s}^{\mathbf{S}}\left(t\right)u(t)\right),
\]
then $U$ satisfies 
\begin{equation}
i\partial_{t}U=AU+V_{1}P_{s}^{\mathbf{S}}\left(t\right)u(t)+V_{2}\left(x-\vec{v}t\right)P_{s}^{\mathbf{S}}\left(t\right)u(t)+P_{s}^{\mathbf{S}}\left(t\right)F,
\end{equation}
By Duhamel's formula, 
\begin{equation}
U(t)=e^{itA}U(0)-i\int_{0}^{t}e^{-i\left(t-s\right)A}\left(V_{1}P_{s}^{\mathbf{S}}\left(s\right)u(s)+V_{2}\left(\cdot-vs\right)P_{s}^{\mathbf{S}}\left(s\right)u(s)+P_{s}^{\mathbf{S}}\left(s\right)F(s)\right)\,ds.
\end{equation}
Let $P:=A^{-1}\Re$, then from Strichartz estimates for the free evolution,
\begin{equation}
\left\Vert Pe^{itA}U(0)\right\Vert \lesssim\left\Vert U(0)\right\Vert _{L^{2}},\label{eq:Sfirst-1-1}
\end{equation}
and 
\begin{equation}
\left\Vert \int_{0}^{t}e^{-i\left(t-s\right)A}P_{s}^{\mathbf{S}}\left(s\right)F(s)\,ds\right\Vert _{L_{t}^{p}L_{x}^{q}}\lesssim\left\Vert F\right\Vert _{L_{t}^{\tilde{p}'}L_{x}^{\tilde{q}'}}.\label{eq:secondIn}
\end{equation}
As in the proof of Theorem \ref{thm:StrichaWOB}, writing $V_{1}=V_{3}V_{4}$,
$V_{2}=V_{5}V_{6}$ , it suffices to bound
\begin{equation}
\left\Vert P\int_{0}^{\infty}e^{-i\left(t-s\right)A}V_{3}V_{4}P_{s}^{\mathbf{S}}\left(s\right)u(s)\,ds\right\Vert _{L_{t}^{p}L_{x}^{q}},
\end{equation}
and 
\begin{equation}
\left\Vert P\int_{0}^{\infty}e^{-i\left(t-s\right)A}V_{5}V_{6}(\cdot-vs)P_{s}^{\mathbf{S}}\left(s\right)u(s)\,ds\right\Vert _{L_{t}^{p}L_{x}^{q}}.
\end{equation}
In the same manner as we did in the proof of Theorem \ref{thm:StrichaWOB},
one has

\begin{equation}
\left\Vert P\int_{0}^{\infty}e^{-i\left(t-s\right)A}V_{3}V_{4}P_{s}^{\mathbf{S}}\left(s\right)u(s)\,ds\right\Vert _{L_{t}^{p}L_{x}^{q}}\lesssim\left\Vert V_{4}P_{s}^{\mathbf{S}}\left(t\right)u\right\Vert _{L_{t}^{2}L_{x}^{2}},
\end{equation}
\begin{equation}
\left\Vert P\int_{0}^{\infty}e^{-i\left(t-s\right)A}V_{5}V_{6}(\cdot-vs)P_{s}^{\mathbf{S}}\left(s\right)u(s)\,ds\right\Vert _{L_{t}^{p}L_{x}^{q}}\lesssim\left\Vert V_{6}(x-\vec{v}t)P_{s}^{\mathbf{S}}\left(t\right)u\right\Vert _{L_{t}^{2}L_{x}^{2}}.
\end{equation}
By estimates \eqref{eq:inhomweight1} and \eqref{eq:inhomweight2} from
Lemma \ref{lem:LDF}, 
\begin{equation}
\left\Vert V_{4}P_{s}^{\mathbf{S}}\left(t\right)u\right\Vert _{L_{t}^{2}L_{x}^{2}}\lesssim\left\Vert \nabla g\right\Vert _{L_{x}^{2}}+\left\Vert f\right\Vert _{L_{x}^{2}}+\left\Vert F\right\Vert _{L_{t}^{1}L_{x}^{2}},
\end{equation}
\begin{equation}
\left\Vert V_{6}(x-\vec{v}t)P_{s}^{\mathbf{S}}\left(t\right)u\right\Vert _{L_{t}^{2}L_{x}^{2}}\lesssim\left\Vert \nabla g\right\Vert _{L_{x}^{2}}+\left\Vert f\right\Vert _{L_{x}^{2}}+\left\Vert F\right\Vert _{L_{t}^{1}L_{x}^{2}}.
\end{equation}
Therefore, by the same argument as for the homogeneous Strichartz
estimates, we have
\begin{equation}
\|P_{s}^{\mathbf{S}}\left(t\right)u\|_{L_{t}^{p}\left([0,\infty),\,L_{x}^{q}\right)}\lesssim\|f\|_{L^{2}}+\|g\|_{\dot{H}^{1}}+\left\Vert F\right\Vert _{L_{t}^{\tilde{p}'}\left([0,\,\infty),\,L_{x}^{\tilde{q}'}\right)\cap L_{t}^{1}\left([0,\infty),\,L_{x}^{2}\right)}.
\end{equation}
as claimed.
\end{proof}
From the discussions above, we can also conclude the endpoint reversed
Strichartz estimate.
\begin{thm}
\label{thm:endpointInhom}Let $\left|v\right|<1$ and suppose $u$
solves
\begin{equation}
\partial_{tt}u-\Delta u+V_{1}(x)u+V_{2}(x-\vec{v}t)u=F
\end{equation}
with initial data 
\begin{equation}
u(x,0)=g(x),\,u_{t}(x,0)=f(x).
\end{equation}
Then we have
\begin{equation}
\sup_{x}\int_{0}^{\infty}\left|P_{s}^{\mathbf{S}}\left(t\right)u\right|^{2}dt\lesssim\left(\|f\|_{L^{2}}+\|g\|_{\dot{H}^{1}}+\left\Vert F\right\Vert _{L_{t}^{1}\left([0,\infty),\,L_{x}^{2}\right)}.\right)^{2}\label{eq:inhomreverse}
\end{equation}
\end{thm}
Taking the case $p=q$ in the regular Strichartz estimate and interpolating
it with the endpoint reversed Strichartz estimate \eqref{eq:inhomreverse},
we obtain more reversed Strichartz estimates.
\begin{cor}
\label{cor:moreStrichInhom}Let $\left|v\right|<1$ and suppose $u$
solves
\begin{equation}
\partial_{tt}u-\Delta u+V_{1}(x)u+V_{2}(x-\vec{v}t)u=F
\end{equation}
with initial data 
\[
u(x,0)=g(x),\,u_{t}(x,0)=f(x).
\]
Then for 
\begin{equation}
2\leq p,\,\tilde{p}\leq8
\end{equation}
and
\begin{equation}
\frac{1}{2}=\frac{1}{p}+\frac{3}{q},\,\frac{1}{2}=\frac{1}{\tilde{p}}+\frac{3}{\tilde{q}}
\end{equation}
we have
\begin{equation}
\left\Vert P_{s}^{\mathbf{S}}\left(t\right)u\right\Vert _{L_{x}^{q}\left(\mathbb{R}^{3},\,L_{t}^{p}[0,\infty)\right)}\lesssim\|f\|_{L^{2}}+\|g\|_{\dot{H}^{1}}+\left\Vert F\right\Vert _{L_{x}^{\tilde{q}'}\left(\mathbb{R}^{3},\,L_{t}^{\tilde{p}'}[0,\infty)\right)\cap L_{t}^{1}\left([0,\infty),\,L_{x}^{2}\right)}.\label{eq:inhomorevese}
\end{equation}
where $\tilde{p}',\,\tilde{q}'$ are H\"older conjugate of $\tilde{p},\,\tilde{q}$.
\end{cor}
We also have the endpoint Strichartz estimate with norm inhomogeneous
with respect to radial and angular variables
\begin{thm}
\label{thm:inhomStricE}Let $\left|v\right|<1$ and suppose $u$ solves
\begin{equation}
\partial_{tt}u-\Delta u+V_{1}(x)u+V_{2}(x-\vec{v}t)u=F
\end{equation}
with initial data 
\begin{equation}
u(x,0)=g(x),\,u_{t}(x,0)=f(x).
\end{equation}
Then for $1\leq p<\infty$, we have
\begin{equation}
\|P_{s}^{\mathbf{S}}\left(t\right)u\|_{L_{t}^{p}\left([0,\infty),\,L_{r}^{\infty}L_{\omega}^{p}\right)}\lesssim\|f\|_{L^{2}}+\|g\|_{\dot{H}^{1}}+\left\Vert F\right\Vert _{L_{t}^{1}\left([0,\infty),\,L_{x}^{2}\right)}.\label{eq:inhomst-1}
\end{equation}
\end{thm}

\subsection{Reversed endpoint Strichartz estimates with inhomogeneous terms in
revered norms}

In some nonlinear applications, the interactions among potentials
and solitons are strong which cause the inhomogeneous terms is not
in $L_{t}^{1}L_{x}^{2}$. So to finish this section, we discuss the
reversed endpoint Strichartz estimates with inhomogeneous terms in
revered norms. We need a slightly different formulation. As we did
in the homogeneous, we recall the definition of scattering states
in inhomogeneous setting.
\begin{defn}
\label{AO-1}Let 
\begin{equation}
\partial_{tt}u-\Delta u+V_{1}(x)u+V_{2}(x-\vec{v}t)u=F,\label{eq:eqBSsec-1-1}
\end{equation}
with initial data
\begin{equation}
u(x,0)=g(x),\,u_{t}(x,0)=f(x).
\end{equation}
If $u$ also satisfies 
\begin{equation}
\left\Vert P_{b}\left(H_{1}\right)u(t)\right\Vert _{L_{x}^{2}}\rightarrow0,\,\,\left\Vert P_{b}\left(H_{2}\right)u_{L}(t')\right\Vert _{L_{x'}^{2}}\rightarrow0\,\,\,t,t'\rightarrow\infty,\label{eq:ao2-1-1}
\end{equation}
we call it a scattering state. 
\end{defn}
Set the space $I$
\begin{equation}
I=\left\{ G(x,t)\in L_{x}^{\frac{3}{2},1}L_{t}^{2}\bigcap L_{x_{1}}^{1}L_{\widehat{x_{1}}}^{2,1}L_{t}^{2}\bigcap L_{t,x}^{2}\right\} \label{eq:Ispace}
\end{equation}
for the strong interactions terms.
\begin{thm}
\label{thm:EndRStriInhomo}Let $\left|v\right|<1$. Suppose $u$ is
a scattering state in the sense of Definition \ref{AO-1} which solves
\begin{equation}
\partial_{tt}u-\Delta u+V_{1}(x)u+V_{2}(x-\vec{v}t)u=F
\end{equation}
with initial data
\begin{equation}
u(x,0)=g(x),\,u_{t}(x,0)=f(x).
\end{equation}
Then
\begin{align}
\left(\sup_{x\in\mathbb{R}^{3}}\int_{0}^{\infty}\left|u(x,t)\right|^{2}dt\right)^{\frac{1}{2}} & \lesssim\|f\|_{L^{2}}+\|g\|_{\dot{H}^{1}}+\left\Vert F\right\Vert _{I}
\end{align}
and 
\begin{align}
\left(\sup_{x\in\mathbb{R}^{3}}\int_{0}^{\infty}\left|u(x+vt,t)\right|^{2}dt\right)^{\frac{1}{2}} & \lesssim\|f\|_{L^{2}}+\|g\|_{\dot{H}^{1}}+\left\Vert F\right\Vert _{I}.
\end{align}
One can replace $F$ in the above estimates by $F^{S}$.
\end{thm}
These estimates can be proved by the same ideas as the homogeneous
case. Here we briefly sketch the arguments since many steps are identical
as the homogeneous case.

First of all, we need the energy comparison. By similar arguments
as we did as the homogeneous case, one has the following comparison
results.
\begin{thm}
\label{thm:generalC-1}Let $\left|v\right|<1$. Suppose 
\begin{equation}
\partial_{tt}u-\Delta u+V(x,t)u=F(x,t)
\end{equation}
and 
\begin{equation}
\left|V(x,\mu x_{1})\right|\lesssim\frac{1}{\left\langle x\right\rangle ^{2}}
\end{equation}
 for $0\leq\left|\mu\right|<1$. Then 

\begin{eqnarray}
\int\left|\nabla_{x}u\left(x_{1},x_{2},x_{3},vx_{1}\right)\right|^{2}+\left|\partial_{t}u\left(x_{1},x_{2},x_{3},vx_{1}\right)\right|^{2}dx\nonumber \\
\text{\ensuremath{\lesssim}}\int\left|\nabla_{x}u\left(x_{1},x_{2},x_{3},0\right)\right|^{2}+\left|\partial_{t}u\left(x_{1},x_{2},x_{3},0\right)\right|^{2}dx\label{eq:generalC-1}\\
+\int_{\mathbb{R}}\int_{\mathbb{R}^{3}}\left|F(x,t)\right|^{2}dxdt\nonumber 
\end{eqnarray}
and 
\begin{eqnarray}
\int\left|\nabla_{x}u\left(x_{1},x_{2},x_{3},0\right)\right|^{2}+\left|\partial_{t}u\left(x_{1},x_{2},x_{3},0\right)\right|^{2}dx\nonumber \\
\text{\ensuremath{\lesssim}}\int\left|\nabla_{x}u\left(x_{1},x_{2},x_{3},vx_{1}\right)\right|^{2}+\left|\partial_{t}u\left(x_{1},x_{2},x_{3},vx_{1}\right)\right|^{2}dx\\
+\int_{\mathbb{R}}\int_{\mathbb{R}^{3}}\left|F(x,t)\right|^{2}dxdt\nonumber 
\end{eqnarray}
where the implicit constant depends on $v$ and $V$.
\end{thm}
From the theorem above, we know initial energy with respect to different
frames stays comparable up to $\left\Vert F\right\Vert _{L_{t,x}^{2}}$.
For a detailed proof, please see \cite{GC2}.

Next, from Section \ref{sec:Slanted}, we have all necessary reversed
type estimates. So one has all the basic tools to run the bootstrap
arguments as in the homogeneous case. We just need to understand the
evolution of bound states more carefully. Nothing changes substantially
but with one more inhomogeneous term in the ODE.

Let $u(x,t)$ be a scattering state. Following the notations from
\ref{subsec:Boundstates}, we decompose the evolution as following,
\begin{equation}
u(x,t)=a(t)w(x)+b\left(\gamma(t-vx_{1})\right)m_{v}\left(x,t\right)+r(x,t)\label{eq:evolution-1}
\end{equation}
where 
\[
m_{v}(x,t)=m\left(\gamma\left(x_{1}-vt\right),x_{2},x_{3}\right).
\]
With our decomposition, we know 
\begin{equation}
P_{c}\left(H_{1}\right)r=r
\end{equation}
 and 
\begin{equation}
P_{c}\left(H_{2}\right)r_{L}=r_{L}
\end{equation}
where the Lorentz transformation $L$ makes $V_{2}$ stationary. 

Plugging the evolution \eqref{eq:evolution-1} into the equation \eqref{eq:chargeeq}
and taking inner product with $w$, we get 
\begin{eqnarray}
\ddot{a}(t)-\lambda^{2}a(t)+a(t)\left\langle V_{2}\left(x-\vec{v}t\right)w,w\right\rangle \qquad\qquad\qquad\qquad\qquad\qquad\nonumber \\
\qquad\qquad\qquad+\left\langle V_{2}\left(x-\vec{v}t\right)\left(b\left(\gamma(t-vx_{1})\right)m_{v}\left(x,t\right)+r(x,t)\right),w\right\rangle =\left\langle F,w\right\rangle .
\end{eqnarray}
One can write 
\begin{equation}
\ddot{a}(t)-\lambda^{2}a(t)+a(t)c(t)+h(t)+h_{1}(t)=0,\label{eq:aode-1}
\end{equation}
where 
\begin{equation}
c(t):=\left\langle V_{2}\left(x-\vec{v}t\right)w,w\right\rangle ,
\end{equation}
\begin{equation}
h_{1}\left(t\right)=\left\langle F,w\right\rangle 
\end{equation}
and
\begin{equation}
h(t):=\left\langle V_{2}\left(x-\vec{v}t\right)\left(b\left(\gamma(t-vx_{1})\right)m_{v}\left(x,t\right)+r(x,t)\right),w\right\rangle .
\end{equation}
The existence of the solution to the ODE \eqref{eq:aode-1} is clear.
We study the long-time behavior of the solution. Write the equation
as 

\begin{equation}
\ddot{a}(t)-\lambda^{2}a(t)=-\left[a(t)c(t)+h(t)+h_{1}(t)\right],
\end{equation}
and denote 
\begin{equation}
N(t):=-\left[a(t)c(t)+h(t)+h_{1}(t)\right].
\end{equation}
Then
\begin{equation}
a(t)=\frac{e^{\lambda t}}{2}\left[a(0)+\frac{1}{\lambda}\dot{a}(0)+\frac{1}{\lambda}\int_{0}^{t}e^{-\lambda s}N(s)\,ds\right]+R(t)
\end{equation}
where 
\begin{equation}
\left|R(t)\right|\lesssim e^{-\beta t},
\end{equation}
for some positive constant $\beta>0$. Therefore, the stability condition
forces 
\begin{equation}
a(0)+\frac{1}{\lambda}\dot{a}(0)+\frac{1}{\lambda}\int_{0}^{\infty}e^{-\lambda s}N(s)\,ds=0.\label{eq:stability-1}
\end{equation}
Then under the stability condition \eqref{eq:stability-1}, 
\begin{equation}
a(t)=e^{-\lambda t}\left[a(0)+\frac{1}{2\lambda}\int_{0}^{\infty}e^{-\lambda s}N(s)ds\right]+\frac{1}{2\lambda}\int_{0}^{\infty}e^{-\lambda\left|t-s\right|}N(s)\,ds.
\end{equation}
As in the homogeneous case, we just need to estimate the non-local
term,
\begin{equation}
\int_{0}^{\infty}e^{-\lambda s}N(s)\,ds.
\end{equation}
The same idea as the homogeneous case, for $t\in[0,T]$, we construct
the following truncated version of the evolution: 
\begin{equation}
u_{T}(x,t)==a_{T}(t)w(x)+b_{T}\left(\gamma(t-vx_{1})\right)m_{v}\left(x,t\right)+r_{T}(x,t).
\end{equation}
For $a_{T}(t)$, we analyze the same ODE for $a(t)$ again but restricted
to $[0,T]$ and instead of the stability condition 
\begin{equation}
a(0)+\frac{1}{\lambda}\dot{a}(0)+\frac{1}{\lambda}\int_{0}^{\infty}e^{-\lambda s}N(s)\,ds=0
\end{equation}
we impose the condition that
\begin{equation}
a_{T}(0)+\frac{1}{\lambda}\dot{a}_{T}(0)+\frac{1}{\lambda}\int_{0}^{T}e^{-\lambda s}N(s)\,ds=0.
\end{equation}
The same construction can be applied to \textbf{$b_{T}$.}

In current setting, we only estimate the $L^{2}$ norms of $a_{T}$
and $b_{T}$. 
\begin{lem}
\label{lem:boundInhom}From the construction above, we have the following
estimates: for $0\ll A\ll T$,
\begin{equation}
\left\Vert a_{T}\right\Vert _{L^{\infty}[0,T]}\lesssim\left(C(A,\lambda)+\frac{1}{\lambda A}C_{1}(T)\right)\left(\|f\|_{L^{2}}+\|g\|_{\dot{H}^{1}}+\left\Vert F\right\Vert _{I}\right),\label{eq:abootL0-1}
\end{equation}
\begin{equation}
\left\Vert a_{T}\right\Vert _{L^{2}[0,T]}\lesssim\left(C(A,\lambda)+\frac{1}{\lambda A}C_{1}(T)\right)\left(\|f\|_{L^{2}}+\|g\|_{\dot{H}^{1}}+\left\Vert F\right\Vert _{I}\right),\label{eq:abootL1-1}
\end{equation}
\begin{equation}
\left\Vert b_{T}\right\Vert _{L^{\infty}[0,T]}\lesssim\left(C(A,\mu)+\frac{1}{\mu A}C_{1}(T)\right)\left(\|f\|_{L^{2}}+\|g\|_{\dot{H}^{1}}+\left\Vert F\right\Vert _{I}\right),\label{eq:bbootL0-1}
\end{equation}
and 
\begin{equation}
\left\Vert b_{T}\right\Vert _{L^{2}[0,T]}\lesssim\left(C(A,\mu)+\frac{1}{\mu A}C_{1}(T)\right)\left(\|f\|_{L^{2}}+\|g\|_{\dot{H}^{1}}+\left\Vert F\right\Vert _{I}\right).\label{eq:bbootL1-1}
\end{equation}
\end{lem}
\begin{proof}
First of all, as in the homogeneous case, by the bootstrap assumption,
\[
\left\Vert b_{T}\left(\gamma(t-vx_{1})\right)m_{v}\left(x,t\right)+r_{T}(x,t)\right\Vert _{L_{x}^{\infty}L_{t}^{2}[0,T]}\leq C_{1}(T)\left(\|f\|_{L^{2}}+\|g\|_{\dot{H}^{1}}+\left\Vert F\right\Vert _{I}\right).
\]
For $a_{T}(t)$, we know that 
\begin{eqnarray}
\ddot{a}_{T}(t)-\lambda^{2}a_{T}(t)+a_{T}(t)\left\langle V_{2}\left(x-\vec{v}t\right)w,w\right\rangle \qquad\qquad\qquad\qquad\qquad\qquad\nonumber \\
\qquad\qquad\qquad+\left\langle V_{2}\left(x-\vec{v}t\right)\left(b_{T}\left(\gamma(t-vx_{1})\right)m_{v}\left(x,t\right)+r_{T}(x,t)\right),w\right\rangle =\left\langle w,F\right\rangle .
\end{eqnarray}
We obtain
\begin{equation}
a_{T}(t)=\frac{e^{\lambda t}}{2}\left[a_{T}(0)+\frac{1}{\lambda}\dot{a}_{T}(0)+\frac{1}{\lambda}\int_{0}^{t}e^{-\lambda s}N(s)ds\right]+R(t)
\end{equation}
where 
\begin{equation}
\left|R(t)\right|\lesssim e^{-\beta t},
\end{equation}
With notations introduced above, we consider the truncated version
of the stability condition, 
\begin{equation}
a_{T}(0)+\frac{1}{\lambda}\dot{a}_{T}(0)+\frac{1}{\lambda}\int_{0}^{T}e^{-\lambda s}N(s)\,ds=0.
\end{equation}
So 
\begin{equation}
a_{T}(t)=e^{-\lambda t}\left[a_{T}(0)+\frac{1}{2\lambda}\int_{0}^{T}e^{-\lambda s}N(s)ds\right]+\frac{1}{2\lambda}\int_{0}^{T}e^{-\lambda\left|t-s\right|}N(s)\,ds.
\end{equation}
where 
\begin{equation}
N(t)=-\left[a_{T}(t)c(t)+h(t)+h_{1}(t)\right]
\end{equation}
with 
\begin{equation}
\left|c(t)\right|\lesssim e^{-\alpha\left|t\right|},\:h_{1}(t)=\left\langle w,F\right\rangle 
\end{equation}
\begin{equation}
h(t):=\left\langle V_{2}\left(x-\vec{v}t\right)\left[b_{T}(t-vx_{1})m_{v}\left(x,t\right)+r_{T}(x,t)\right],w\right\rangle .
\end{equation}
For $0\ll A\ll T$ fixed, we can always bound the $L^{\infty}$ norm
of $a_{T}$ on the interval $[0,A]$ by Gr\"onwall's inequality.
Therefore, it suffices to estimate the $L^{\infty}$ norm of $a_{T}$
from $A$ to $T$. Note that $\left|c(t)\right|\lesssim e^{-\alpha\left|t\right|}$,
for $A$ large, one can always absorb the effects from $\int_{A}^{T}$$a_{T}(t)c(t)\,dt$
into the left-hand side. Hence it reduces to estimate the $L_{t}^{1}$
norm of $h(t)$ restricted to $[A,T]$ and the $L_{t}^{2}$ norm of
$h_{1}(t)$. 

From the computations in \ref{subsec:Boundstates}, we have

\begin{align}
\int_{A}^{T}\left|h(t)\right|\,dt & \lesssim C(A,\lambda)\left(\|f\|_{L^{2}}+\|g\|_{\dot{H}^{1}}+\left\Vert F\right\Vert _{I}\right)\\
 & +\frac{1}{\lambda A}\left(\int_{A}^{T}\left|\left(b_{T}\left(\gamma(t-vx_{1})\right)m_{v}\left(x,t\right)+r_{T}(x,t)\right)\right|^{2}dt\right)^{\frac{1}{2}}.\nonumber 
\end{align}
Clearly,
\begin{equation}
\int_{0}^{T}\left|h_{1}(t)\right|^{2}\lesssim\int_{0}^{T}\left\Vert F(t)\right\Vert _{L_{x}^{2}}dt.
\end{equation}
We can estimate the $L^{\infty}$ norm of $a_{T}(t)$, 
\begin{eqnarray}
\left\Vert a_{T}\right\Vert _{L^{\infty}[0,T]} & \lesssim & C(A,\lambda)\left(\|f\|_{L^{2}}+\|g\|_{\dot{H}^{1}}+\left\Vert F\right\Vert _{I}\right)\nonumber \\
 &  & ++\frac{1}{\lambda}\int_{A}^{T}\left|h(t)\right|dt+\frac{1}{\lambda}\left(\int_{0}^{T}\left|h_{1}(t)\right|^{2}\right)^{\frac{1}{2}}\\
 & \lesssim & C(A,\lambda)\left(\|f\|_{L^{2}}+\|g\|_{\dot{H}^{1}}+\left\Vert F\right\Vert _{I}\right)\nonumber \\
 &  & +\frac{1}{\lambda A}\left(\int_{A}^{T}\left|\left(b_{T}\left(\gamma(t-vx_{1})\right)m_{v}\left(x,t\right)+r_{T}(x,t)\right)\right|^{2}dt\right)^{\frac{1}{2}}\nonumber \\
 & \lesssim & \left(C(A,\lambda)+\frac{1}{\lambda A}C_{1}(T)\right)\left(\|f\|_{L^{2}}+\|g\|_{\dot{H}^{1}}+\left\Vert F\right\Vert _{I}\right).\nonumber 
\end{eqnarray}
Similarly, for the $L^{2}$ norm of $a_{T}(t)$,  
\begin{equation}
\left\Vert a_{T}\right\Vert _{L^{2}[0,T]}\lesssim\left(C(A,\lambda)+\frac{1}{\lambda A}C_{1}(T)\right)\left(\|f\|_{L^{2}}+\|g\|_{\dot{H}^{1}}+\left\Vert F\right\Vert _{I}\right).
\end{equation}
After applying a Lorentz transformation, we have analogous estimates
for $b_{T}(t)$:
\begin{equation}
\left\Vert b_{T}\right\Vert _{L^{\infty}[0,T]}\lesssim\left(C(A,\mu)+\frac{1}{\mu A}C_{1}(T)\right)\left(\|f\|_{L^{2}}+\|g\|_{\dot{H}^{1}}+\left\Vert F\right\Vert _{I}\right),
\end{equation}
\begin{equation}
\left\Vert b_{T}\right\Vert _{L^{2}[0,T]}\lesssim\left(C(A,\mu)+\frac{1}{\mu A}C_{1}(T)\right)\left(\|f\|_{L^{2}}+\|g\|_{\dot{H}^{1}}+\left\Vert F\right\Vert _{I}\right).
\end{equation}
The lemma is proved.
\end{proof}
With the preparations above, we can run the bootstrap argument and
channel decomposition as in the homogeneous case.
\begin{proof}[Proof of Theorem \ref{thm:EndRStriInhomo}]
As in the homogeneous
case, let $\chi_{1}(x,t)$ be a smooth cutoff function such that 
\begin{equation}
\chi_{1}(x,t)=1,\;\forall x\in B_{\delta t}(0),\qquad\chi_{1}(x,t)=0,\;\forall x\in\mathbb{R}^{3}\backslash B_{2\delta t}(0).
\end{equation}
One might assume $t\geq t_{0}$ for some large $t_{0}$. We also define
\begin{equation}
\chi_{2}(x,t)=\chi_{1}(x-\vec{v}t,t),\qquad\chi_{3}=1-\chi_{1}-\chi_{2}.
\end{equation}
Note that we only consider the estimates for large $t$, so one might
also assume the support of $\chi_{1}(x,t)$ contains the support of
$V_{1}\left(x\right)$ and support of $\chi_{2}(x,t)$ contains the
support of $V_{2}\left(\cdot-vt\right)$.

With the partition above, we rewrite the evolution as
\begin{equation}
u_{T}(x,t)=\chi_{1}(x,t)u_{T}(x,t)+\chi_{2}(x,t)u_{T}(x,t)+\chi_{3}(x,t)u_{T}(x,t).
\end{equation}
We will discuss $\chi_{i}(x,t)u_{T}(x,t),\,i=1,2,3$, separately.
We only analyze $\chi_{1}(x,t)u_{T}(x,t)$ here since other pieces
are can be done in the same manner in the homogeneous case. T
\begin{eqnarray}
\chi_{1}(x,t)u_{T}(x,t) & = & \chi_{1}(x,t)W_{1}(t)f+\chi_{1}(x,t)\dot{W}_{1}(t)g\nonumber \\
 &  & -\chi_{1}(x,t)\int_{0}^{t}W_{1}(t-s)V_{2}(\cdot-sv)u_{T}(s)\,ds\label{eq:first-1}\\
 &  & +\chi_{1}(x,t)\int_{0}^{t}W_{1}(t-s)F(s)\,ds.\nonumber 
\end{eqnarray}
We will use the notations 
\begin{eqnarray}
u_{T}(x,t) & = & a_{T}(t)w(x)+b_{T}\left(\gamma(t-vx_{1})\right)m_{v}\left(x,t\right)+r_{T}(x,t)\nonumber \\
 & =: & a_{T}(t)w(x)+u_{T,1}\left(x,t\right)\label{eq:decomposition-1}\\
 & =: & b_{T}\left(\gamma(t-vx_{1})\right)m_{v}\left(x,t\right)+u_{T,2}(x,t).\nonumber 
\end{eqnarray}
Note that 
\begin{equation}
P_{c}\left(H_{1}\right)\left(u_{T,1}\right)=u_{T,1}
\end{equation}
and 
\begin{equation}
P_{c}\left(H_{2}\right)\left(u_{T,2}\right)_{L}=\left(u_{T,2}\right)_{L}.
\end{equation}
As in the homogeneous case, we can further reduce to 
\begin{eqnarray}
\chi_{1}(x,t)u_{T}(x,t) & = & \chi_{1}(x,t)W_{1}(t)f+\chi_{1}(x,t)\dot{W}_{1}(t)g\nonumber \\
 &  & -\chi_{1}(x,t)\int_{0}^{t-A}W_{1}(t-s)V_{2}(\cdot-sv)u_{T}(s)\,ds\\
 &  & +\chi_{1}(x,t)\int_{0}^{t}W_{1}(t-s)F(s)\,ds.\nonumber 
\end{eqnarray}
First, we consider the endpoint reversed Strichartz estimate, 
\begin{eqnarray*}
\int_{0}^{T}\left|\chi_{1}(x,t)u_{T,1}(x,t)\right|^{2}dt & \lesssim & \int_{0}^{T}\left|\chi_{1}(x,t)W_{1}(t)P_{c}\left(H_{1}\right)f+\chi_{1}(x,t)\dot{W}_{1}(t)P_{c}\left(H_{1}\right)g\right|^{2}dt\\
 &  & +\int_{0}^{T}\left|\chi_{1}(x,t)\int_{0}^{t-A}W_{1}(t-s)P_{c}\left(H_{1}\right)V_{2}(\cdot-sv)u_{T}(s)\,ds\right|^{2}dt\\
 &  & \int_{0}^{T}\left|\chi_{1}(x,t)\chi_{1}(x,t)\int_{0}^{t}W_{1}(t-s)P_{c}\left(H_{1}\right)F(s)\,ds\right|^{2}dt\\
 & \lesssim & \left(\|f\|_{L^{2}}+\|g\|_{\dot{H}^{1}}+\left\Vert F\right\Vert _{I}\right)^{2}\\
 &  & +\int_{0}^{B}\left|\chi_{1}(x,t)\int_{0}^{t}W_{1}(t-s)P_{c}\left(H_{1}\right)V_{2}(\cdot-sv)u_{T}(s)\,ds\right|^{2}dt\\
 &  & +\int_{B}^{T}\left|\chi_{1}(x,t)\int_{0}^{t-A}W_{1}(t-s)P_{c}\left(H_{1}\right)V_{2}(\cdot-sv)u_{T}(s)\,ds\right|^{2}\\
 & \lesssim & \left(\|f\|_{L^{2}}+\|g\|_{\dot{H}^{1}}+\left\Vert F\right\Vert _{I}\right)^{2}+C(B)\left(\|f\|_{L^{2}}+\|g\|_{\dot{H}^{1}}+\left\Vert F\right\Vert _{I}\right)^{2}\\
 &  & +\int_{B}^{T}\left|\chi_{1}(x,t)\int_{0}^{t-A}W_{1}(t-s)P_{c}\left(H_{1}\right)V_{2}(\cdot-sv)u_{T}(s)\,ds\right|^{2}dt\\
 & \lesssim & \left(\|f\|_{L^{2}}+\|g\|_{\dot{H}^{1}}+\left\Vert F\right\Vert _{I}\right)^{2}+C(B)\left(\|f\|_{L^{2}}+\|g\|_{\dot{H}^{1}}+\left\Vert F\right\Vert _{I}\right)^{2}\\
 &  & +\frac{1}{A}C_{2}(T)\left(\|f\|_{L^{2}}+\|g\|_{\dot{H}^{1}}+\left\Vert F\right\Vert _{I}\right)^{2}.
\end{eqnarray*}
Therefore, 
\begin{equation}
\int_{0}^{T}\left|\chi_{1}(x,t)u_{T,1}(x,t)\right|^{2}dt\lesssim\left(C_{0}+C(A)+\frac{1}{A}C_{2}(T)\right)\left(\|f\|_{L^{2}}+\|g\|_{\dot{H}^{1}}+\left\Vert F\right\Vert _{I}\right)^{2}.\label{eq:firstEnd-1}
\end{equation}
For the remaining piece, by Lemma \ref{lem:boundInhom}
\begin{equation}
\int_{0}^{T}\left|\chi_{1}(x,t)a_{T}(t)w(x)\right|^{2}dt\lesssim\left(C(A,\lambda)+\frac{1}{\lambda A}C_{1}(T)\right)\left(\|f\|_{L^{2}}+\|g\|_{\dot{H}^{1}}+\left\Vert F\right\Vert _{I}\right)^{2}.\label{eq:firstEndB-1}
\end{equation}
Therefore, with estimates \eqref{eq:firstEnd-1} and \eqref{eq:firstEndB-1},
for the endpoint reversed estimate, we obtain
\begin{equation}
C_{1}(T)\lesssim C_{0}+C(A,B)+\frac{1}{A}C_{2}(T)\label{eq:boot1first-1}
\end{equation}
in the first channel. So for $A$ large, in this channel, we have
the condition for the bootstrap argument.

Next we consider the estimate along the slanted line $(x+vt,t)$. 

Denoting 
\begin{equation}
u_{T,1}^{S}(x,t)=\chi_{1}(x+vt,t)u_{T,1}(x+vt,t),
\end{equation}
we want to estimate 
\begin{equation}
\int_{0}^{T}\left|\chi_{1}(x+vt,t)u_{T,1}(x+vt,t)\right|^{2}dt=\int_{0}^{T}\left|u_{T,1}^{S}(x,t)\right|^{2}dt.
\end{equation}
Furthermore, we introduce
\begin{equation}
D_{1}^{S}(x,t):=D_{1}\left(x+vt,t\right)
\end{equation}
where 
\begin{equation}
D_{1}(x,t):=\chi_{1}(x,t)W_{1}(t)P_{c}\left(H_{1}\right)f+\chi_{1}(x,t)\dot{W_{1}}(t)P_{c}\left(H_{1}\right)g;
\end{equation}
 
\begin{equation}
k_{1}^{S}(x,t):=k_{1}\left(x+vt,t\right)
\end{equation}
where 
\begin{equation}
k_{1}(x,t):=\chi_{1}(x,t)\int_{0}^{t}W_{1}(t-s)P_{c}\left(H_{1}\right)V_{2}(\cdot-sv)u_{T}(s)\,ds;
\end{equation}

\begin{equation}
E_{1}^{S}(x,t):=E_{1}\left(x+vt,t\right)
\end{equation}
where 
\begin{equation}
E_{1}(x,t):=\chi_{1}(x,t)\int_{0}^{t-A}W_{1}(t-s)P_{c}\left(H_{1}\right)V_{2}(\cdot-sv)u_{T}(s)\,ds.
\end{equation}
\[
E_{2}(x,t):=\chi_{1}(x,t)\int_{0}^{t}W_{1}(t-s)P_{c}\left(H_{1}\right)F(s)u_{T}(s)\,ds.
\]
Then we can conclude 
\begin{eqnarray}
\int_{0}^{T}\left|u_{T,1}^{S}\right|^{2}dt & \lesssim & \int_{0}^{T}\left|D_{1}^{S}\right|^{2}dt+\int_{0}^{B}\left|k_{1}^{S}\right|^{2}dt+\int_{B}^{T}\left|E_{1}^{S}\right|^{2}dt+\int_{0}^{T}\left|E_{2}^{S}\right|^{2}dt\nonumber \\
 & \lesssim & \left(\|f\|_{L^{2}}+\|g\|_{\dot{H}^{1}}+\left\Vert F\right\Vert _{I}\right)^{2}+C(B)\left(\|f\|_{L^{2}}+\|g\|_{\dot{H}^{1}}+\left\Vert F\right\Vert _{I}\right)^{2}\nonumber \\
 &  & +\frac{1}{A}C_{2}(T)\left(\|f\|_{L^{2}}+\|g\|_{\dot{H}^{1}}+\left\Vert F\right\Vert _{I}\right)^{2}.\label{eq:firstSE-1}
\end{eqnarray}
For the piece with bound states, by Lemma \ref{lem:boundInhom} and
Agmon's estimate, 
\begin{eqnarray}
\int_{0}^{T}\left|\chi_{1}(x+vt,t)a_{T}(t)w(x+vt)\right|^{2}dt\label{eq:firstSEB-1}\\
\lesssim\left(C(A,\lambda)+\frac{1}{\lambda A}C_{1}(T)\right)\left(\|f\|_{L^{2}}+\|g\|_{\dot{H}^{1}}+\left\Vert F\right\Vert _{I}\right)^{2}.\nonumber 
\end{eqnarray}
Therefore, with estimates \eqref{eq:firstSE-1} and \eqref{eq:firstSEB-1},
we obtain
\begin{equation}
C_{2}(T)\lesssim C_{0}+C(A,B)+\frac{1}{A}C_{2}(T)\label{eq:boot2first-1}
\end{equation}
in the first channel. So for $A$ large, in this channel, we obtain
the desired reduction for the bootstrap argument.

The other two channels can be analyzed by the same steps as above.
Therefore, after passing $T$ to $\infty,$ we can conclude that 
\begin{align}
\left(\sup_{x\in\mathbb{R}^{3}}\int_{0}^{\infty}\left|u(x,t)\right|^{2}dt\right)^{\frac{1}{2}} & \lesssim\|f\|_{L^{2}}+\|g\|_{\dot{H}^{1}}+\left\Vert F\right\Vert _{I}
\end{align}
and 
\begin{align}
\left(\sup_{x\in\mathbb{R}^{3}}\int_{0}^{\infty}\left|u(x+vt,t)\right|^{2}dt\right)^{\frac{1}{2}} & \lesssim\|f\|_{L^{2}}+\|g\|_{\dot{H}^{1}}+\left\Vert F\right\Vert _{I}.
\end{align}
The same arguments work for $F$ replaced by $F^{S}$.

We are done.
\end{proof}

\subsection{Reversed type local decay estimates\label{subsec:RevLocal}}

To handle the strong interactions of solitons and potentials, in this
subsection we establish some reversed type local decay estimates.
These estimates are also important to handle multisoliton structures
as in \cite{GC4}.
\begin{thm}
\label{thm:localreversed}Let $\left|v\right|<1$. Suppose $u$ is
a scattering state in the sense of Definition \ref{AO-1} which solves
\begin{equation}
\partial_{tt}u-\Delta u+V_{1}(x)u+V_{2}(x-\vec{v}t)u=F
\end{equation}
with initial data
\begin{equation}
u(x,0)=g(x),\,u_{t}(x,0)=f(x).
\end{equation}
Then
\begin{align}
\left\Vert \left\langle x\right\rangle ^{-\frac{3}{2}}u(x,t)\right\Vert _{L_{x}^{3,2}L_{t}^{\infty}\bigcap L_{x_{1}}^{2}L_{\widehat{x_{1}}}^{4,2}L_{t}^{\infty}} & \lesssim\|f\|_{L^{2}}+\|g\|_{\dot{H}^{1}}+\left\Vert F\right\Vert _{D}
\end{align}
\begin{align}
\left\Vert \left\langle x\right\rangle ^{-\frac{3}{2}}u(x+vt,t)\right\Vert _{L_{x}^{3,2}L_{t}^{\infty}\bigcap L_{x_{1}}^{2}L_{\widehat{x_{1}}}^{4,2}L_{t}^{\infty}} & \lesssim\|f\|_{L^{2}}+\|g\|_{\dot{H}^{1}}+\left\Vert F\right\Vert _{D}.
\end{align}
Here the space $D$ is 
\begin{equation}
D:=\left\{ G(x,t)\in L_{x}^{\frac{3}{2},1}L_{t}^{\infty}\bigcap L_{x_{1}}^{1}L_{\widehat{x_{1}}}^{2,1}L_{t}^{\infty}\bigcap L_{t}^{2}L_{x}^{2}\right\} .
\end{equation}
Again, one can replace $F$ by $F^{S}$ in the above estimates. 
\end{thm}
This theorem can be proved by the same way as Theorem \ref{thm:EndRStriInhomo}
provided we have the energy comparison and the related reversed type
estimates for the free wave equations and perturbed equations.

We start with the free equation again. We set 
\begin{equation}
u_{F}(x,t)=\frac{\sin\left(t\sqrt{-\Delta}\right)}{\sqrt{-\Delta}}f+\cos\left(t\sqrt{-\Delta}\right)g,
\end{equation}
and
\begin{equation}
D(\cdot,t)=\int_{0}^{t}\frac{\sin\left((t-s)\sqrt{-\Delta}\right)}{\sqrt{-\Delta}}F(s)\,ds.
\end{equation}

\begin{thm}
Let $\left|v\right|<1$. Then first of all, for the standard case,
one has 
\begin{equation}
\left\Vert u_{F}\right\Vert _{L_{x}^{6,2}L_{t}^{\infty}}\lesssim\|f\|_{L^{2}}+\|g\|_{\dot{H}^{1}},\label{eq:localrevS}
\end{equation}
in particular, 
\begin{equation}
\left\Vert \left\langle x\right\rangle ^{-\frac{3}{2}}u_{F}\right\Vert _{L_{x}^{3,2}L_{t}^{\infty}\bigcap L_{x_{1}}^{2}L_{\widehat{x_{1}}}^{4,2}L_{t}^{\infty}}\lesssim\|f\|_{L^{2}}+\|g\|_{\dot{H}^{1}}.\label{eq:local2}
\end{equation}
Also for the inhomogeneous term, 
\begin{equation}
\left\Vert \left\langle x\right\rangle ^{-\frac{3}{2}}D\right\Vert _{L_{x}^{\frac{3}{2},1}L_{t}^{\infty}\bigcap L_{x_{1}}^{1}L_{\widehat{x_{1}}}^{2,1}L_{t}^{\infty}}\lesssim\left\Vert F\right\Vert _{L_{x}^{\frac{3}{2},1}L_{t}^{\infty}}.\label{eq:local3}
\end{equation}
 We can also estimate these pieces along slanted lines and obtain
\begin{equation}
\left\Vert u_{F}^{S}\right\Vert _{L_{x}^{6,2}L_{t}^{\infty}}\lesssim\|f\|_{L^{2}}+\|g\|_{\dot{H}^{1}},\label{eq:local33}
\end{equation}
\begin{equation}
\left\Vert \left\langle x\right\rangle ^{-\frac{3}{2}}u_{F}^{S}\right\Vert _{L_{x}^{3,2}L_{t}^{\infty}\bigcap L_{x_{1}}^{2}L_{\widehat{x_{1}}}^{4,2}L_{t}^{\infty}}\lesssim\|f\|_{L^{2}}+\|g\|_{\dot{H}^{1}},\label{eq:local4}
\end{equation}
and 
\begin{equation}
\left\Vert \left\langle x\right\rangle ^{-\frac{3}{2}}D^{S}\right\Vert _{L_{x}^{3,2}L_{t}^{\infty}\bigcap L_{x_{1}}^{2}L_{\widehat{x_{1}}}^{4,2}L_{t}^{\infty}}\lesssim\left\Vert F\right\Vert _{L_{x_{1}}^{1}L_{\widehat{x_{1}}}^{2,1}L_{t}^{\infty}}.\label{eq:local5}
\end{equation}
We can replace the $L_{t}^{\infty}$ norm of $F$ by the $L_{t}^{1}$
norm. Also one can replace $F$ in the above estimates by $F^{S}$. 
\end{thm}
\begin{proof}
We will only prove \eqref{eq:localrevS}. \eqref{eq:local2} is a consequence
of estimate \eqref{eq:localrevS} after applying H\"older's inequality.
\eqref{eq:local4} and \eqref{eq:local33} follow from estimate \eqref{eq:localrevS}
after performing a Lorentz transformation and energy comparison as
in Section \ref{sec:Slanted}. For the inhomogeneous estimates, we
do the same arguments as in Section \ref{sec:Slanted} with $L_{t}^{2}$
replaced by $L_{t}^{\infty}$. For example, 
\begin{eqnarray}
\left\Vert \int_{0}^{t}\frac{\sin\left((t-s)\sqrt{-\Delta}\right)}{\sqrt{-\Delta}}F(s)\,ds\right\Vert _{L_{t}^{\infty}} & = & \left\Vert \int_{0}^{t}\int_{\left|x-y\right|=t-s}\frac{1}{\left|x-y\right|}F(y,s)\,\sigma\left(dy\right)ds\right\Vert _{L_{t}^{\infty}}\nonumber \\
 & = & \left\Vert \int_{\left|x-y\right|\leq t}\frac{1}{\left|x-y\right|}F\left(y,t-\left|x-y\right|\right)\,dy\right\Vert _{L_{t}^{\infty}}\nonumber \\
 & \lesssim & \int\frac{1}{\left|x-y\right|}\left\Vert F\left(y,t-\left|x-y\right|\right)\right\Vert _{L_{t}^{2}}dy\nonumber \\
 & \lesssim & \sup_{x\in\mathbb{R}^{3}}\int\frac{1}{\left|x-y\right|}\left\Vert F\left(y,t\right)\right\Vert _{L_{t}^{\infty}}dy\\
 & \lesssim & \left\Vert F\right\Vert _{L_{x}^{\frac{3}{2},1}L_{t}^{\infty}}.\nonumber 
\end{eqnarray}
Therefore, 
\begin{equation}
\left\Vert D\right\Vert _{L_{x}^{\infty}L_{t}^{\infty}}\lesssim\left\Vert F\right\Vert _{L_{x}^{\frac{3}{2},1}L_{t}^{\infty}},
\end{equation}
and \eqref{eq:local3} follows after applying H\"older's inequality.

Now we prove \eqref{eq:localrevS}. Consider $t\ge0$ and define 
\begin{equation}
Tf=\frac{\sin\left(t\sqrt{-\Delta}\right)}{\sqrt{-\Delta}}f
\end{equation}
then 
\begin{equation}
T^{*}F=\int_{0}^{\infty}\frac{\sin\left(t\sqrt{-\Delta}\right)}{\sqrt{-\Delta}}F(t)\,dt,
\end{equation}
and
\begin{align}
TT^{*}F & =\int_{0}^{\infty}\frac{\sin\left(t\sqrt{-\Delta}\right)}{\sqrt{-\Delta}}\frac{\sin\left(s\sqrt{-\Delta}\right)}{\sqrt{-\Delta}}F(s)\,ds\\
 & =\frac{1}{2}\int_{0}^{\infty}\left(\frac{\cos\left(\left(t-s\right)\sqrt{-\Delta}\right)}{-\Delta}-\frac{\cos\left(\left(t+s\right)\sqrt{-\Delta}\right)}{-\Delta}\right)F(s)\,ds.\nonumber 
\end{align}
We compute the kernel of 
\begin{equation}
\frac{\cos\left(h\sqrt{-\Delta}\right)}{-\Delta}F=\int_{\mathbb{R}^{3}}K(x,y,h)F(y)\,dy.
\end{equation}
By straightforward computations, one has
\begin{equation}
\frac{\cos\left(h\sqrt{-\Delta}\right)}{-\Delta}=\frac{1}{-\Delta}-\int_{0}^{h}\frac{\sin\left(s\sqrt{-\Delta}\right)}{\sqrt{-\Delta}}\,ds=\int_{h}^{\infty}\frac{\sin\left(s\sqrt{-\Delta}\right)}{\sqrt{-\Delta}}\,ds.
\end{equation}
By the explicit kernel of $\frac{\sin\left(s\sqrt{-\Delta}\right)}{\sqrt{-\Delta}}$,
we know that 
\begin{equation}
K(x,y,h)=\begin{cases}
\frac{1}{\left|x-y\right|} & \left|x-y\right|\geq h\\
0 & \left|x-y\right|<h
\end{cases}.\label{eq:kernelC}
\end{equation}
Notice that in $\mathbb{R}^{3}$, $\frac{1}{\left|x\right|}\in L^{3,\infty}$,
so 
\begin{equation}
\left\Vert \int_{0}^{\infty}\left(\frac{\cos\left(\left(t-s\right)\sqrt{-\Delta}\right)}{-\Delta}\right)F(s)\,ds\right\Vert _{L_{x}^{6,2}L_{t}^{\infty}}\lesssim\left\Vert F\right\Vert _{L_{x}^{\frac{6}{5},2}L_{t}^{1}}
\end{equation}
by Young's inequality for convolution. It follows that 
\begin{equation}
\left\Vert Tf\right\Vert _{L_{x}^{6,2}L_{t}^{\infty}}=\left\Vert \frac{\sin\left(t\sqrt{-\Delta}\right)}{\sqrt{-\Delta}}f\right\Vert _{L_{x}^{6,2}L_{t}^{\infty}}\lesssim\left\Vert f\right\Vert _{L^{2}}.
\end{equation}
We are done.
\end{proof}
By the same arguments in Section \ref{sec:Slanted}, we can extend
all the above estimates to perturbed cases. Define 
\begin{equation}
u_{H}(x,t)=\frac{\sin\left(t\sqrt{H}\right)}{\sqrt{H}}P_{c}f+\cos\left(t\sqrt{H}\right)P_{c}g,
\end{equation}
and 
\begin{equation}
k(\cdot,t):=\int_{0}^{t}\frac{\sin\left((t-s)\sqrt{H}\right)}{\sqrt{H}}P_{c}F(s)\,ds.\label{eq:Pinhom-1}
\end{equation}

\begin{thm}
\label{thm:perinhomlocal}Let $\left|v\right|<1$ and suppose $H=-\Delta+V$
has neither resonances nor eigenfunctions at $0$. Then first of all,
for the standard case, one has 
\begin{equation}
\left\Vert u_{H}\right\Vert _{L_{x}^{6,2}L_{t}^{\infty}}\lesssim\|f\|_{L^{2}}+\|g\|_{\dot{H}^{1}},\label{eq:localrevS-1}
\end{equation}
in particular, 
\begin{equation}
\left\Vert \left\langle x\right\rangle ^{-\frac{3}{2}}u_{H}\right\Vert _{L_{x}^{3,2}L_{t}^{\infty}\bigcap L_{x_{1}}^{2}L_{\widehat{x_{1}}}^{4,2}L_{t}^{\infty}}\lesssim\|f\|_{L^{2}}+\|g\|_{\dot{H}^{1}}.\label{eq:local2-1}
\end{equation}
Also for the inhomogeneous term, 
\begin{equation}
\left\Vert \left\langle x\right\rangle ^{-\frac{3}{2}}k\right\Vert _{L_{x}^{3,2}L_{t}^{\infty}\bigcap L_{x_{1}}^{2}L_{\widehat{x_{1}}}^{4,2}L_{t}^{\infty}}\lesssim\left\Vert F\right\Vert _{L_{x}^{\frac{3}{2},1}L_{t}^{\infty}}.\label{eq:local3-1}
\end{equation}
 We can also estimate these pieces along slanted lines and obtain
\begin{equation}
\left\Vert u_{H}^{S}\right\Vert _{L_{x}^{6,2}L_{t}^{\infty}}\lesssim\|f\|_{L^{2}}+\|g\|_{\dot{H}^{1}},\label{eq:local33-1}
\end{equation}
\begin{equation}
\left\Vert \left\langle x\right\rangle ^{-\frac{3}{2}}u_{H}^{S}\right\Vert _{L_{x}^{3,2}L_{t}^{\infty}\bigcap L_{x_{1}}^{2}L_{\widehat{x_{1}}}^{4,2}L_{t}^{\infty}}\lesssim\|f\|_{L^{2}}+\|g\|_{\dot{H}^{1}},\label{eq:local4-1}
\end{equation}
and 
\begin{equation}
\left\Vert \left\langle x\right\rangle ^{-\frac{3}{2}}k^{S}\right\Vert _{L_{x}^{3,2}L_{t}^{\infty}\bigcap L_{x_{1}}^{2}L_{\widehat{x_{1}}}^{4,2}L_{t}^{\infty}}\lesssim\left\Vert F\right\Vert _{L_{x_{1}}^{1}L_{\widehat{x_{1}}}^{2,1}L_{t}^{\infty}}.\label{eq:local5-1}
\end{equation}
We can replace the $L_{t}^{\infty}$ norm of $F$ by the $L_{t}^{1}$
norm. Also one can replace $F$ in the above estimates by $F^{S}$. 
\end{thm}
By identical discussions to Section \ref{sec:Slanted} with $L_{t}^{2}$
replaced by $L_{t}^{\infty}$, we have all the estimates for $D_{A}$
and $k_{A}$, the truncated version of $D$ and $k$, with factor
$\frac{1}{A}$. So we have all the necessary ingredients for our bootstrap
process. With the decomposition of channels and all the estimates above,
we can conclude Theorem \ref{thm:localreversed}. We omit the details
since they are identical as the proof for Theorem \ref{thm:EndRStriInhomo}.

\section{Scattering\label{sec:Scattering}}

In this section, we show some applications of the results in this
paper. We will study the long-time behaviors for a scattering state
in the sense of Definition \ref{AO}. 

Following the notations from above section, we will still use the
short-hand notation 
\begin{equation}
L_{t}^{p}L_{x}^{q}:=L_{t}^{p}\left([0,\infty),\,L_{x}^{q}\right).
\end{equation}

In general, we can write a general wave equation as
\begin{equation}
\partial_{tt}u-\Delta u=F(u,t)
\end{equation}
with initial data
\begin{equation}
u(x,0)=g(x),\,u_{t}(x,0)=f(x).
\end{equation}
Also consider the homogeneous free wave equation,
\begin{equation}
\partial_{tt}u_{0}-\Delta u_{0}=0
\end{equation}
with initial data
\begin{equation}
u_{0}(x,0)=g_{0}(x),\,\left(u_{0}\right)_{t}(x,0)=f_{0}(x).
\end{equation}
For scattering states, we consider the following question: given data
$\left(g,f\right)\in\dot{H}^{1}\times L^{2}$ and a corresponding
solution $u\in\dot{H}^{1}\times L^{2}$ to the perturbed problem $\square u=F(u,t)$
with initial data $\left(g,f\right)\in\dot{H}^{1}\times L^{2}$, can
we find data $\left(g_{0},f_{0}\right)\in\dot{H}^{1}\times L^{2}$
such that the solution $u_{0}\in\dot{H}^{1}\times L^{2}$ to the corresponding
homogeneous problem $\partial_{tt}u_{0}-\Delta u_{0}=0$, $\left(g_{0},f_{0}\right)\in\dot{H}^{1}\times L^{2}$
is such that 
\begin{equation}
\left\Vert u(t)-u_{0}(t)\right\Vert _{\dot{H}^{1}\times L^{2}}\rightarrow0,\,\,\,\,t\rightarrow\infty
\end{equation}
To do this, as we discuss about projections, we reformulate the wave
equation as a Hamiltonian system, 
\begin{equation}
\partial_{t}\left(\begin{array}{c}
u\\
\partial_{t}u
\end{array}\right)-\left(\begin{array}{cc}
0 & 1\\
-1 & 0
\end{array}\right)\left(\begin{array}{cc}
-\Delta & 0\\
0 & 1
\end{array}\right)\left(\begin{array}{c}
u\\
\partial_{t}u
\end{array}\right)=\left(\begin{array}{c}
0\\
F(u)
\end{array}\right).
\end{equation}
Setting 
\begin{equation}
U:=\left(\begin{array}{c}
u\\
\partial_{t}u
\end{array}\right),\,J:=\left(\begin{array}{cc}
0 & 1\\
-1 & 0
\end{array}\right),\,H_{F}:=\left(\begin{array}{cc}
-\Delta & 0\\
0 & 1
\end{array}\right)\,\text{and }F(U):=\left(\begin{array}{c}
0\\
F(u,t)
\end{array}\right),\label{eq:bigU}
\end{equation}
we can rewrite the free wave equation as 
\begin{equation}
\dot{U}_{0}-JH_{F}U_{0}=0,
\end{equation}
\begin{equation}
U_{0}[0]=\left(\begin{array}{c}
g_{0}\\
f_{0}
\end{array}\right)
\end{equation}
and the perturbed wave equation as 
\begin{equation}
\dot{U}-JH_{F}U=F(U),
\end{equation}
\begin{equation}
U[0]=\left(\begin{array}{c}
g\\
f
\end{array}\right).
\end{equation}
The solution of the free wave equation is given by 
\begin{equation}
U_{0}=e^{tJH_{F}}U_{0}[0],
\end{equation}
on the other hand, by Duhamel's formula, the solution to the perturbed
wave equation is given by 
\begin{equation}
U[t]=e^{tJH_{F}}U[0]+\int_{0}^{t}e^{(t-s)JH_{F}}F\left(U(s)\right)\,ds.
\end{equation}
We consider the charge transfer model, 
\begin{equation}
\partial_{tt}u-\Delta u+V_{1}(x)u+V_{2}(x-\vec{v}t)u=0
\end{equation}
for which 
\begin{equation}
F(u,t)=-\left(V_{1}(x)u+V_{2}(x-\vec{v}t)u\right)
\end{equation}

\begin{thm}
\label{thm:scattering}Suppose $u$ is a scattering state in the sense
of Definition \ref{AO} which solves 
\begin{equation}
\partial_{tt}u-\Delta u+V_{1}(x)u+V_{2}(x-\vec{v}t)u=0.\label{eq:scattE}
\end{equation}
Write 
\begin{equation}
U=\left(u,u_{t}\right)^{t}\in C^{0}\left([0,\infty);\,\dot{H}^{1}\right)\times C^{0}\left([0,\infty);\,L^{2}\right),
\end{equation}
with initial data \textup{$U[0]=\left(g,f\right)^{t}\in\dot{H}^{1}\times L^{2}$.
Then there exist free data 
\[
U_{0}[0]=\left(g_{0},f_{0}\right)^{t}\in\dot{H}^{1}\times L^{2}
\]
 such that 
\begin{equation}
\left\Vert U[t]-e^{tJH_{F}}U_{0}[0]\right\Vert _{\dot{H}^{1}\times L^{2}}\rightarrow0
\end{equation}
as $t\rightarrow\infty$.}
\end{thm}
\begin{proof}
We will still use the formulation in Theorem \ref{thm:StrichaWOB}.
We set $A=\sqrt{-\Delta}$ and notice that 
\begin{equation}
\left\Vert Af\right\Vert _{L^{2}}\simeq\left\Vert f\right\Vert _{\dot{H}^{1}},\,\,\forall f\in C^{\infty}\left(\mathbb{R}^{3}\right).
\end{equation}
For real-valued $u=\left(u_{1},u_{2}\right)\in\mathcal{H}=\dot{H}^{1}\left(\mathbb{R}^{3}\right)\times L^{2}\left(\mathbb{\mathbb{R}}^{3}\right)$,
we write 
\begin{equation}
U:=Au_{1}+iu_{2}.
\end{equation}
We know 
\begin{equation}
\left\Vert U\right\Vert _{L^{2}}\simeq\left\Vert \left(u_{1},u_{2}\right)\right\Vert _{\mathcal{H}}.
\end{equation}
We also notice that $u$ solves \eqref{eq:scattE} if and only if 
\begin{equation}
U:=Au+i\partial_{t}u
\end{equation}
satisfies 
\begin{equation}
i\partial_{t}U=AU+V_{1}u+V_{2}\left(x-\vec{v}t\right)u,
\end{equation}
\begin{equation}
U(0)=Ag+if\in L^{2}\left(\mathbb{R}^{3}\right).
\end{equation}
By Duhamel's formula, for fixed $T$ 
\begin{equation}
U(T)=e^{iTA}U(0)-i\int_{0}^{T}e^{-i\left(T-s\right)A}\left(V_{1}u(s)+V_{2}\left(\cdot-vs\right)u(s)\right)\,ds.
\end{equation}
Applying the free evolution backwards, we obtain
\begin{equation}
e^{-iTA}U(T)=U(0)-i\int_{0}^{T}e^{isA}\left(V_{1}u(s)+V_{2}\left(\cdot-vs\right)u(s)\right)\,ds.
\end{equation}
Letting $T$ go to $\infty$, we define 
\begin{equation}
U_{0}(0):=U(0)-i\int_{0}^{\infty}e^{isA}\left(V_{1}u(s)+V_{2}\left(\cdot-vs\right)u(s)\right)\,ds
\end{equation}
By construction, we just need to show $U_{0}[0]$ is well-defined
in $L^{2}$, then automatically, 
\begin{equation}
\left\Vert U(t)-e^{itA}U_{0}(0)\right\Vert _{L^{2}}\rightarrow0.
\end{equation}
 It suffices to show 
\begin{equation}
\int_{0}^{\infty}e^{isA}\left(V_{1}u(s)+V_{2}\left(\cdot-vs\right)u(s)\right)\,ds\in L^{2}.
\end{equation}
Then following the argument as in the proof of Theorem \ref{thm:StrichaWOB},
we write $V_{1}=V_{3}V_{4}$, $V_{2}=V_{5}V_{6}$. 

We consider 
\begin{equation}
\left\Vert \int_{0}^{\infty}e^{isA}V_{3}V_{4}u(s)\,ds\right\Vert _{L_{x}^{2}}\leq\left\Vert K_{1}\right\Vert _{L_{t,x}^{2}\rightarrow L_{x}^{2}}\left\Vert V_{4}u\right\Vert _{L_{t,x}^{2}},
\end{equation}
where 
\begin{equation}
\left(K_{1}F\right)(t):=\int_{0}^{\infty}e^{isA}V_{3}F(s)\,ds.
\end{equation}
Similarly, 
\begin{equation}
\left\Vert \int_{0}^{\infty}e^{isA}V_{5}V_{6}(\cdot-vs)u(s)\,ds\right\Vert _{L_{x}^{2}}\leq\left\Vert \widetilde{K}_{1}\right\Vert _{L_{t,x}^{2}\rightarrow L_{x}^{2}}\left\Vert V_{6}(x-\vec{v}t)u\right\Vert _{L_{t,x}^{2}},
\end{equation}
where 
\begin{equation}
\left(\widetilde{K}_{1}F\right)(t):=\int_{0}^{\infty}e^{isA}V_{3}(\cdot-vs)F(s)\,ds.
\end{equation}
By the same argument in the proof of Theorem \ref{thm:EnergyCharge},
one has 
\begin{equation}
\left\Vert K_{1}\right\Vert _{L_{t,x}^{2}\rightarrow L_{x}^{2}}\leq C_{1},\,\left\Vert \widetilde{K}_{1}\right\Vert _{L_{t,x}^{2}\rightarrow L_{x}^{2}}\leq C_{2}.
\end{equation}
Therefore

\begin{equation}
\left\Vert \int_{0}^{\infty}e^{isA}V_{3}V_{4}u(s)\,ds\right\Vert _{L_{x}^{2}}\lesssim\left\Vert V_{4}u\right\Vert _{L_{t,x}^{2}},
\end{equation}
\begin{equation}
\left\Vert \int_{0}^{\infty}e^{isA}V_{5}V_{6}(\cdot-vs)u(s)\,ds\right\Vert _{L_{x}^{2}}\lesssim\left\Vert V_{6}(x-\vec{v}t)u\right\Vert _{L_{t,x}^{2}}.
\end{equation}
By estimates \eqref{eq:WeightCharWOB1} and \eqref{eq:WeightCharWOB2}
from Corollary \ref{cor:weightChargeWOB}, 
\begin{equation}
\left\Vert V_{4}u\right\Vert _{L_{t,x}^{2}}\lesssim\left(\int_{\mathbb{R}^{+}}\int_{\mathbb{R}^{3}}\frac{1}{\left\langle x\right\rangle ^{\alpha}}\left|u(x,t)\right|^{2}dxdt\right)^{\frac{1}{2}}\lesssim\|f\|_{L^{2}}+\|g\|_{\dot{H}^{1}},
\end{equation}
\begin{equation}
\left\Vert V_{6}(x-\vec{v}t)u\right\Vert _{L_{t,x}^{2}}\lesssim\left(\int_{\mathbb{R}^{+}}\int_{\mathbb{R}^{3}}\frac{1}{\left\langle x-\vec{v}t\right\rangle ^{\alpha}}\left|u(x,t)\right|^{2}dxdt\right)^{\frac{1}{2}}\lesssim\|f\|_{L^{2}}+\|g\|_{\dot{H}^{1}}.
\end{equation}
We conclude 
\[
\int_{0}^{\infty}e^{isA}\left(V_{1}u(s)+V_{2}\left(\cdot-vs\right)u(s)\right)\,ds\in L^{2}
\]
with 
\begin{equation}
\left\Vert \int_{0}^{\infty}e^{isA}\left(V_{1}u(s)+V_{2}\left(\cdot-vs\right)u(s)\right)\,ds\right\Vert _{L^{2}}\lesssim\|f\|_{L^{2}}+\|g\|_{\dot{H}^{1}}.
\end{equation}
So 
\begin{equation}
U_{0}(0):=U(0)-i\int_{0}^{\infty}e^{isA}\left(V_{1}u(s)+V_{2}\left(\cdot-vs\right)u(s)\right)\,ds
\end{equation}
is well-defined in $L^{2}$ and 
\begin{equation}
\left\Vert U(t)-e^{itA}U_{0}(0)\right\Vert _{L^{2}}\rightarrow0.
\end{equation}
Define 
\begin{equation}
\left(g_{0},f_{0}\right):=\left(A^{-1}\Re U_{0}(0),\,\Im U_{0}(0)\right).
\end{equation}
By construction, notice that 
\begin{equation}
U[t]=\left(A^{-1}\Re U(t),\,\Im U(t)\right)
\end{equation}
and 
\begin{equation}
\left\Vert U[t]-e^{tJH_{F}}U_{0}[0]\right\Vert _{\dot{H}^{1}\times L^{2}}\rightarrow0.
\end{equation}
We are done.
\end{proof}
The above theorem confirms that scattering states indeed scatter to
free waves.

\end{document}